\documentclass[11pt,reqno]{amsart}
\usepackage{amssymb}
\usepackage{amsmath,amsthm}
\usepackage{amsfonts}
\usepackage{leftidx}
\usepackage{mathrsfs}
\usepackage{enumerate}  
\usepackage{abstract}
\usepackage{stmaryrd}
\usepackage{ulem}
\usepackage{slashed}
 \usepackage{float}
\usepackage{graphicx}
 \usepackage{hyperref}
 
 \usepackage[top=1in, bottom=1in, left=0.8in, right=0.8in]{geometry}

\def\R{\mathbb{R}}
\def\Z{\mathbb{Z}}

\def\d{|\nabla|}

\def\p{\partial}

\def\vo{\vspace{1\baselineskip}}

\def\be{\begin{equation}}
\def\ee{\end{equation}}

\newtheorem{theorem}{Theorem}[section]

\newtheorem{lemma}{Lemma}[section]
\newtheorem{proposition}{Proposition}[section]
\theoremstyle{definition}

\theoremstyle{remark}
\newtheorem{remark}{Remark}[section]

\numberwithin{equation}{section}
\begin{document}
\title[  3D anisotropic  NLW   ]{  Global solution of 3D anisotropic   wave equations with the same speed in one direction }
\author{Xuecheng Wang}
\address{YMSC, Tsinghua University \& BIMSA\\ Beijing, China 100084}
\email{xuecheng@tsinghua.edu.cn}

\maketitle 
\begin{abstract}
Motivated by the study of uniaxial crystal optics, we consider a coupled system of $3D$ anisotropic wave equations, in which they have the same speed only in one direction but distinct speeds in the other directions.   Moreover,  this system has a $1D$-type null structure along the common direction. 

Unlike the isotropic case, in which the celebrated Klainerman vector fields method is very successful,  the main difficulty of the anisotropic case lies in the absence of enough joint commuting vector fields. After carefully analyzing the differences between group velocities of two waves, which are delicate near the same speed direction,   the role of null structure, and a fine structure of the phase of oscillation in time,   we prove global stability and scattering for small localized initial data. 
\end{abstract}

\tableofcontents

\section{Introduction}

\subsection{Motivation and the problem set-up}

 According to Courant-Hilbert \cite{CouHil62}[Chapter VI: section 3a] and Born-Wolf \cite{BornWolf}, the system of crystal optics is   Maxwell's system in homogeneous   media. Moreover, from Born-Wolf\cite{BornWolf}[page 841], based on the optical properties, we can classify transparent crystals into three distinct groups. Both group (II) and group (III) are anisotropic crystals. More precisely,  
 \begin{enumerate}

\item[(I)] Cubic systems. Cubic system is optically isotropic. Dielectric constants are all same, e.g., $\epsilon_x=\epsilon_y=\epsilon_z$. An example of cubic systems is glass. 
\item[(II)] Biaxial crystals.  Dielectric constants are all different, e.g., $\epsilon_x\neq \epsilon_y\neq \epsilon_z$.     An example of biaxial crystals is biotite. 
\item[(III)] Uniaxial crystals. Only two out of three dielectric constants are  same, e.g., $\epsilon_x=\epsilon_y\neq \epsilon_z$. An example of uniaxial crystals is quartz. 

\end{enumerate}

In the seminal work of Klainerman \cite{Kl85}, Klainerman makes use of the conformal structure of the Minkowski spacetime,  more precisely, that is   the existence of  a collection of  vector fields  commute with the wave operator, which allows him to prove decay estimate of the nonlinear solution. This celebrated Klainerman vector fields method has been very influential in the study of NLW for the past 30 years. We will discuss more about history in the next subsection.   Due to the isotropic property of cubic systems, Klainerman vector field method is very successful for the study of crystal optics in group (I).

Comparing with the rich literature in the study of the NLW with the same speed, the study of  anisotropic system of nonlinear waves is rather at the preliminary stage. A general picture of thinking questions related to the anisotropic system of NLW started with Fritz John, which in particular  includes   questions on anisotropic elastic dynamics and the crystal optics.   A fascinating  folklore program, which initiated by Sergiu Klainerman in 80s, is to improve our understanding  of anisotropic nonlinear wave system, in particular the strength of vector field method in the case of without enough symmetry. This folklore program includes at least the following questions,
\begin{enumerate}

\item[$\bullet$] The role of   joint symmetries, which is much less than the isotropic case. In particular, how to make use of   joint symmetries even at the linear level? Due to the absence of enough joint commuting vector fields, one don't expect to have an analogue of the  Klainerman-Sobolev embedding for the anisotropic case, but is there a good linear estimate that makes the most of joint symmetries, e.g., the scaling vector field?

\item[$\bullet$] Global existence (Blow up) of    anisotropic  nonlinear wave system for small data? It would be interesting to know that  the linear estimate obtained in the above question plays a role in the nonlinear problem.   

\end{enumerate}

Although the above questions are very general and far from complete, moreover, the level of difficulty varies from case to case, there are many interesting questions can be explored. To the best knowledge of the author, the first result in this line of research is due to Otto Liess \cite{Lie89,Lie91}, which initially  suggested to him by Klainerman. Comparing with the $t^{-1 }$ decay rate in the  $3D$ isotropic case,   the   linear estimate obtained in  \cite{Lie91}[Theorem 1.3] for the anisotropic case  is only $t^{-1/2}.$ We refer readers to  \cite{Lie91}[page 2] right after ``Two difficulties then appear right away$\cdots$'' for more detailed explanation about this drastic difference. 

Recently, there is a beautiful paper by John Anderson \cite{anderson}, in which he  studies the small data global stability problem for  the \textit{2D cubic}   \textit{transversal}  anisotropic system of quasilinear wave equations. This system considered by Anderson \cite{anderson} is related to the crystal optics in group (II), i.e., biaxial crystals.

The goal of this paper is to add an important missing piece to the picture, which is the tangential case. Moreover, this is also related to the crystal optics in group (III), i.e., uniaxial crystals. As mentioned in  \cite{PWang},  ``optically uniaxial crystals are important materials used in fabricating various optical elements and optical devices$\cdots$''.

Comparing this paper with  \cite{anderson}, there are two main differences. Firstly, we study \textit{3D quadratic} systems. Lastly and most importantly, we study the \textit{tangential} case instead of the transversal case.  

 The difference between the tangential case and the transversal is drastic, subtle and technical. In terms of the space-time resonance language,  there is no space-time resonance  set for  the transversal case. However, the space-time resonance  set for the  {tangential} case is not empty. In fact, the space-time resonance  set    lies exactly in the  direction of same speed  or  the tangent hypersurface if there are more than one  direction of same speed. 

When light propagates in uniaxial crystals, there are two types of wave,  ordinary wave (o-wave) and  extraordinary wave (e-wave). The normal surfaces of these two waves are explained in Figure \ref{pic1}, see also Born-Wolf \cite{BornWolf}[section 15.3] for more physical backgrounds.
\begin{figure}[H]
  \includegraphics[width=0.5\linewidth]{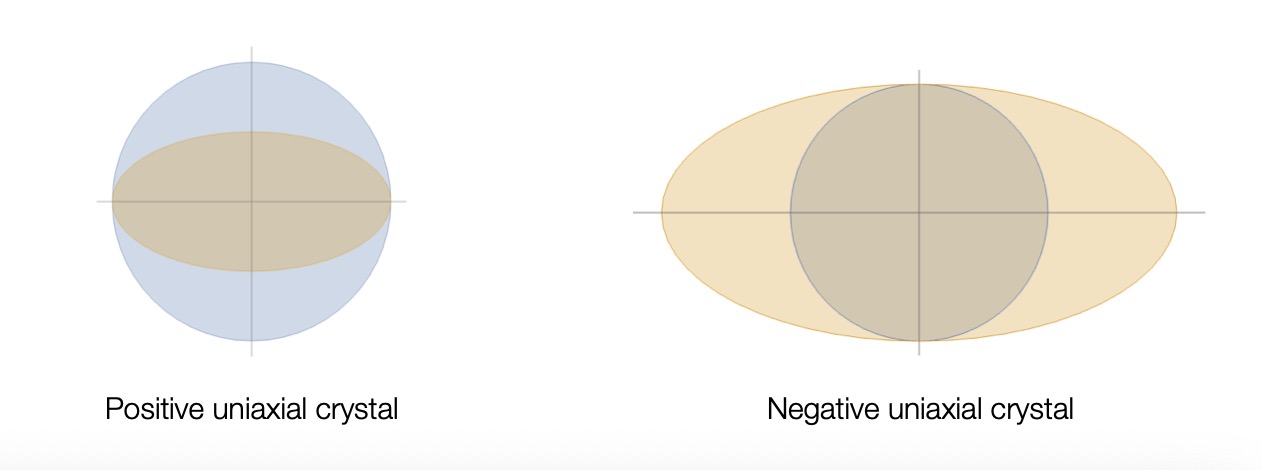}
  \caption{Normal surfaces of uniaxial crystals}\label{pic1}
\end{figure}

   We are interested in the question of   nonlinear interaction between these two types of wave. Motivated from the above discussion,  in this paper, we consider the following system of nonlinear wave (NLW) equations, 
\be\label{march8eqn1} 
\left\{\begin{array}{l}
\p_t^2 u_a-\Delta_a u_a= \mathcal{N}_a(\nabla_{t,x}u_1, \nabla_{t,x} u_2), \quad a\in\{1,2\}, \\
\Delta_1:= \p_{x_1}^2 + \p_{x_2}^2 + \p_{x_3}^2,\quad \Delta_2:= \p_{x_1}^2 + c_1^2 \p_{x_2}^2 +c_2^2 \p_{x_3}^2,\\ 
u_a=f_a,\quad \p_t u_a= g_a, \quad a=1,2,
\end{array}\right.
\ee
where either  $c_1, c_2 > 1$ or  $0< c_1, c_2 < 1$.  That's to say, $u_1$ and $u_2$ have the same speed, which is $1$, in $x_1$ direction and have different speeds in the other two directions, i.e., $x_2,x_3$ directions.

Since the goal of this paper is to study the interaction of nonlinear waves with different speeds, for concreteness of discussion,   we focus on the  $\mathcal{N}_i(\nabla_{t,x}u_1, \nabla_{t,x} u_2)$-type interaction and rule out the self-interactions of $u_1$ and $u_2$ in the nonlinearities.

From  the well-known finite time blow up result of John \cite{Joh81}, we know that certain null structure needs to be imposed to ensure the global existence of the nonlinear  solution. Since $u_1$ and $u_2$ have the same speed in $x_1$ direction, to prevent the blow up mechanism  due to lack of null structure  in $x_1$ direction, we impose a $1D$-type  null structure in $x_1$ direction.

More precisely, let
\[
\tilde{u}_i(t,x_1):=\int_{\R} \int_{\R} {u}_i(t,x_1,x_2,x_3) d x_2 d x_3,\qquad i \in\{1,2\},\quad \tilde{u}_i:\R_t\times \R\longrightarrow \R.
\]

It's expected that, which will also be clear in later $L^\infty_\xi$-norm analysis, $\nabla_{x_2,x_3}$ plays the role of null structure. Hence,  the asymptotic system of \eqref{march8eqn1} is determined the system satisfied by $\tilde{u}_i(t,x_1),i\in\{1,2\}$ as follows, 
\be\label{asympto}
\p_t^2 \tilde{u}_i(t,x_1)-\p_{x_1}^2\tilde{u}_i(t,x_1) = \mathcal{N}_i(\nabla_{t,x_1}\tilde{u}_1, \nabla_{t,x_1} \tilde{u}_2), \quad i\in\{1,2\}.
\ee

 To prevent the  blow up mechanism of the asymptotic system \eqref{asympto} because of lacking null structure,  we impose   $1D$-type null structure for  the nonlinearities  $  N_i(\nabla_{t,x_1}\tilde{u}_1, \nabla_{t,x_1} \tilde{u}_2), i=1,2.$ For simplicity, we consider $Q_{0,0}(\cdot, \cdot)$-type null structure for $N_i (\nabla_{t,x_1}u_1, \nabla_{t,x_1}u_2)$ in this paper. For $1D$ NLW with $Q_{0,0}(\cdot, \cdot)$ type null structure, the global existence is known, see Luli-Yang-Yu \cite{Luli}. 

 With  the above discussion,  we study the following system of equations, which aims to reveal the   subtle interaction of waves with different speeds in general but with the same speed in only one direction, 
\be\label{june27eqn4}
\left\{\begin{array}{l}
(\p_t^2 - \Delta_1) u_1 = \p_t u_1 \p_t u_2 - \p_{x_1} u_1 \p_{x_1}u_2  + \sum_{i,j\in\{2,3\} }\mathcal{Q}^1_{i,j}(\nabla_{x_i} u_1, \nabla_{x_j} u_2 ),   \\ 
(\p_t^2 - \Delta_2) u_2 =\p_t u_1 \p_t u_2 - \p_{x_1} u_1 \p_{x_1} u_2 + \sum_{i,j\in\{2,3\} } \mathcal{Q}^2_{i,j}(\nabla_{x_i} u_1, \nabla_{x_j} u_2 )   ,  \\ 
\Delta_1:= \p_{x_1}^2 + \p_{x_2}^2 + \p_{x_3}^2,\quad \Delta_2:= \p_{x_1}^2 + c_1^2 \p_{x_2}^2 +c_2^2 \p_{x_3}^2,\\ 
u_a=f_a,\quad \p_t u_a= g_a, \quad a=1,2.
\end{array}\right.
\ee
where $\mathcal{Q}^a_{i,j}, a\in\{1,2\}, i,j\in\{2,3\},$   are  bilinear operators with  zero order symbols $\mathfrak{q}^a_{i,j}(\cdot,\cdot),  $ s.t., the following estimate holds for any $k,k_1,k_2\in \Z$, 
\be\label{addsymbolest}
\sup_{|\alpha|, |\beta|\leq 100}\sum_{a\in\{1,2\}, i,j\in\{2,3\} } 2^{|\alpha|k_1 + |\beta| k_2 }\| \nabla_\xi^\alpha \nabla_\eta^\beta \mathfrak{q}^a_{i,j}(\xi-\eta, \eta)\psi_{k}(\xi)\psi_{k_1}(\xi-\eta)\psi_{k_2}(\eta)\|_{L^\infty_{\xi,\eta}}\lesssim 1. 
\ee

Despite we mainly work with  the  system \eqref{june27eqn4} in this paper, we remark that the Fourier method used in this paper is applicable for a larger  class of nonlinearities. 
\subsection{Previous results }

For concreteness, nonlinear wave equations  that we discuss in this subsection usually have nonlinearities depending  on the \textit{derivatives} of solution.  To better understand the state of the art,    we   provide a short history on  the global stability problem of  nonlinear waves equations and related results  for readers' convenience. Also, we refer readers to the review by  Anderson in \cite{anderson} for more discussion.

The study of the long time behavior of NLW with small initial data started a long time ago, e.g., classic works of F. John. We refer readers to the review by Klainerman in the proceeding of ICM 1983 \cite{Kla82} for more details. Despite many previous results on long time behavior, a natural question is that whether smooth solutions exist for all time for small initial data. The first work that answers such question is due to Klainerman \cite{Kla80}, in which he showed that this is true at least for space dimensions greater than or  equal to six.  Later a simplified proof was provided by Klainerman-Ponce in \cite{KlaPon83}.  In the seminal work   \cite{Kl85}, the celebrated Klainerman vector field method (as well as the Klainerman-Sobolev embedding) was introduced,  in which the symmetries    of the wave operator are exploited to prove decay estimate of the nonlinear solution. As  corollaries, small data global regularity result is improved to four and five space dimensions. Moreover, it provides another proof of the almost global result for  three space dimensions case  obtained in   Klainerman \cite{Kla83}, and John-Klainerman \cite{JohKla84}.  

Due to the critical decay rate of $3D$ wave,   from the work of John\cite{Joh81}, it's not expected to have global solution in general. In \cite{Kla82}, \textit{null condition} was introduced by Klainerman, which eliminates the parallel interaction of nonlinearity. As shown by Klainerman \cite{Kla86} and  Christodoulou \cite{Chr86},  for NLW with null structure, global stability for small data indeed holds. A great  achievement of the vector field method lies in the monumental work of  Christodoulou-Klainerman \cite{ChrKl93}, in which they proved global stability for the Einstein-Vacuum equation around the Minkowski spacetime. A crucial step in \cite{Kla86} is to identify certain kind of null condition inside the coupled system, which is different from the classic null structures.   

It turns out that the \textit{classic} null condition is sufficient  but not always necessary to ensure the  global existence  of 3D NLW for small data. As first shown by Lindblad \cite{Lin92} for the radial case and later by Alinac \cite{Ali03} for the general case, the equation $\square \phi= \phi \Delta \phi$  has global solution for small initial data.   Very interestingly, this equation doesn't have classic null structure. In nowadays language, this equation has the  so-called weak null structure. For the Einstein-vacuum equation,  
Lindblad-Rodnianski \cite{LinRod05,LinRod10} identified that  it  also has the  weak  null structure  in the wave coordinates system   which allows them to give an alternative and conceptually simpler proof of the stability of Minkowski spacetime in the wave coordinates system in \cite{LinRod10}, see also  Klainerman-Nicol\`o \cite{Klainerman15},  Bieri\cite{Bieri1}. This observation plays a key role in later study of   Einstein-scalar field systems, e.g., the study of Einstein-Maxwell system in \cite{Zipser}, the study of   Einstein-Klein-Gordon system in \cite{IP,LeFloch2,Qian}, the study of Einstein-Vlasov system in \cite{Bigorgne,FJS,taylor,Wang}.

However,   for the anisotropic system of NLW, due to different speeds,  the main mathematical challenge is that we don't enough symmetry.  The first   work that only uses the scaling vector field and rotational vector fields was due to  Klainerman-Sideris  \cite{KlaSid96}.   A refinement was introduced by  Sideris \cite{Sideris} to study the prestressed nonlinear elastic waves.  For the $2D$ transversal anisotropic system of NLW,  Anderson \cite{anderson} performed a careful analysis of the spacetime geometry together with the  vector field method.  One should expect that the geometry would be more difficult in the $3D$  case  and even more delicate for the tangential case. 
  For the system   \eqref{march8eqn1}, as a result of computations, we know that   $ S=t\p_t + x\cdot\nabla_x$  and $ L_1:= x_1\p_t + t\p_{x_1}$ commute  with the system \eqref{march8eqn1} but $\Omega_{12} $ and $\Omega_{13}$ ($\Omega_{ij}:=x_i\p_{x_j}-x_j\p_{x_i}$) don't commute  with  the system   \eqref{march8eqn1}.   Generally speaking, the set of joint symmetry varies from case to case for the study of crystal optics.  However, due to the fact that wave operators of different waves are all second order operators, we do always have the scaling vector field. How to make use of the  scaling vector field appears to be a key.

Besides the vector field method, there is another approach  for the small data global regularity type problems.  It mainly studies    nonlinear solutions on the Fourier side.  Conceptually,  the goal  is same, which is to  control  the pull back of the nonlinear solution along the linear flow  over time. At the technical level, it varies dramatically from case to case.  For the past ten years,  this approach plays an important role in the study of nonlinear dispersive PDEs and NLW. 

 Without trying to be exhaustive, we briefly mention two methods on the Fourier approach here, which are the spacetime resonance method and the $Z$-norm method.  The  spacetime resonance method was   introduced by Germain-Masmoudi-Shatah \cite{GerMasSha09} in the study of NLS. Now, it has very wide applications in the study of nonlinear dispersive equations, see e.g., \cite{GerMasSha12,GerMasSha14,IP1,IP2,IP3,IP} and  the nonlinear wave equations, see \cite{DenPus20,PusSha13}.   The designer $Z$-norm method  was firstly introduced by Ionescu-Pausader in \cite{IP2}. This method is often used together with the spacetime resonance method. It   depends essentially on identifying the “correct”  $Z$-norm (often atomic space type norm),  depending on the problem, to prove sharp  or almost sharp  decay estimates for the nonlinear solution.

Due to the internal structure of the system \eqref{june27eqn4}, it turns out that we are able to consider the small data problem for the system  \eqref{june27eqn4} purely on the Fourier side  in this paper. But  we believe that the combination of the Fourier approach and the vector field method by exploiting the benefit of the joint commuting vector fields has wider application in the study of crystal optics. This is also our interest in the    future. Using the strength of    both methods is not new, see e.g., \cite{IP,Wang}. 

\subsection{The main result }
For convenience in studying the evolution of nonlinear waves  on the Fourier side, we define the following half-waves and their corresponding profiles as follows, 
\be\label{march19eqn1}
U_a: =(\p_t - i \Lambda_a) u_a,  
\quad 
  \Lambda_a: = \sqrt{-\Delta_a},    \quad \forall a \in \{1,2\}.
\ee

Our main result is stated as follows, 
\begin{theorem}\label{maintheorem}
Let $N_0=10^5,{} \delta\in (0, 10^{-10}), \alpha=10^{-3}, (c_1,c_2)\in (0,1)\times (0,1)\cup (1,\infty)\times (1,\infty)$. There exists a small constant $\epsilon_0 \in \R_+$ such that if initial data $(f_a, g_a)\in H^{N_0}(\R^3)$ satisfies the following estimate, 
\be\label{jan15eqn1}
\sum_{a=1,2}\|U_a(0)\|_{H^{N_0}}  + \| \langle x \rangle^{1+\alpha}  U_a(0)\|_{L^2_x} +  \|\langle \xi\rangle^{8} \widehat{U}_a(0,\xi) \|_{L^\infty_\xi }\leq \epsilon_0, \quad U_a(0):= g_a -i \Lambda_a f_a.
\ee
then the system \eqref{june27eqn4} admits global solutions. The following estimate holds for the nonlinear solution, 
\be\label{finalestimate}
\sup_{t\in [0, \infty)}\sum_{a=1,2} \langle t \rangle^{-\delta } \|  U_a(t) \|_{H^{N_0}} + \|  \langle \xi\rangle^{8}  U_a(t, \xi)\|_{L^\infty_\xi} + \langle t\rangle  \|  U_a(t)\|_{W^{4,\infty}}\lesssim \epsilon_0.
\ee
Moreover, the nonlinear solution of \eqref{june27eqn4} scatters to a linear solution in $L^\infty_\xi$-space as time goes to infinity. 
 
\end{theorem}

A few remarks are in order.

\begin{remark}
For the purpose of studying the anisotropic system of NLW in larger context, e.g., with variable coefficients or in a curved spacetime,  it would be interesting to identify a good set of first order pseduo-differential operators  and   second order pseduo-differential operators from the study of this paper, which allows us to prove sharp (or almost sharp) decay estimates for the nonlinear solution. This will be our future project. Moreover, it's also our interest to   study the $3D$ anisotropic system of waves with the same speed along two different directions and with different speeds along the other direction. 
\end{remark}
\begin{remark}
 Since the method we use here  only depends on the properties of symbols of nonlinearities, the above theorem also holds for a class of anisotropic systems of NLW. Moreover, to run the argument presented in this paper, we don't need null structure as strong as the null structure presented in \eqref{june27eqn4}.

\end{remark}
 
\begin{remark}
For the purpose of simplifying the presentation of theorem, we choose to use a  weighted $L^2$-type space for the initial data $U_a(0),a\in\{1,2\}$, which is stronger than necessary. From our later bootstrap assumption \eqref{bootstrap}, it would be sufficient if the   initial data $U_a(0),a\in\{1,2\}$ lies in the atomic  $Z$-normed space defined in \eqref{weightednormspace}. 
\end{remark}
\begin{remark}
Last and also the least, the plausible  goal of optimizing $N_0,\alpha$, etc., is not pursued here. 
\end{remark}
\subsection{Main ideas of proof and the plan of this paper}

Due to the critical $1/t$ decay rate of the $3D$ wave equation, by using the standard energy estimate, the small data global regularity problem is reduced to prove that the \textit{nonlinear} solution indeed decays sharply over time. To this end, in the following, we discuss main steps and  main ideas in each step. 

$\bullet$\qquad \textit{Step (i)}\qquad The pointwise linear decay estimate. 

Since the small data global existence problem has been a subject under active investigation for many years, there are many  linear decay estimates available in literature. We don't try to give a complete list here. Instead, we give two representative estimates. Firstly, we have the celebrated Klainerman-Sobolev embedding, which states as follows in $3D$, 
\[
|u(t,x)| \lesssim \sum_{|\alpha|\leq 2} \sum_{\Gamma\in \{S, \Omega_{ij}, L_i\}} \frac{1}{(1+|t|+|x|)(1+||t|-|x||)^{1/2} } \| \Gamma^\alpha u(t,\cdot)\|_{L^2_x}. 
\]
Moreover, we also have the following linear decay  estimate by Pusateri-Shatah \cite{PusSha13}[Lemma A.1], which doesn't use any vector field, 
\[
\| e^{i t\d} f \|_{L^\infty} \lesssim \frac{1}{t}\| \langle x\rangle \d^2 f\|_{L^2}^{1/2}\| \langle x\rangle^2 \d^2 f\|_{L^2}^{1/2}, \quad \d:=\sqrt{-\Delta_1}.
\]

Generally speaking, choosing what kind of decay estimate to work with is the very part of  whole argument. Since one still needs to  control the right hand side  over time to show that the nonlinear solution indeed decays sharply over time. Whether such a goal can be achieved depends on the nonlinear system itself. 

For the anisotropic system \eqref{june27eqn4}, since we don't have the full set of vector fields and it doesn't look promising to propagate a very high order weight for the profile due to the presence of the nonempty spacetime resonance set, we can't use the above two estimates directly. Instead, roughly speaking, we use the following linear decay estimate, 
\be\label{jan15eqn11}
\textup{(Rough version)}  \qquad  \| e^{ \pm i t\Lambda_a } f \|_{L^\infty} \lesssim \frac{1}{t}\big[\| |\xi|^2 \widehat{f}(\xi)\|_{L^\infty_\xi} + \langle t\rangle^{-\beta} \| \langle x\rangle^{1+} \d^{(3/2-\beta) +} f\|_{L^2}  \big],
\ee
where $a\in \{1,2\}$ and  $\beta$ is some absolute constant, see Lemma \ref{lineardecay} for the precise estimate.   

$\bullet$\qquad \textit{Step (ii)}\qquad Reduction to the estimates of profiles.  

To obtain   decay estimate of the nonlinear solution, we pull back the nonlinear solution along linear flow, which is called the profile, as follows
\be\label{jan16eqn1}
\forall a\in \{1,2\}, \quad h_a(t) = e^{it \Lambda_a} U_a, \quad \Longrightarrow U_a(t) = e^{-it \Lambda} h_a(t). 
\ee
From the linear decay estimate \eqref{jan15eqn11},  the goal of sharp decay is reduced as follows, 
\be\label{jan15eqn12}
\textup{(Rough version)}\quad   \| U_a(t) \|_{L^\infty} \lesssim \frac{1}{t}\big[ \underbrace{\| |\xi|^2 \widehat{h_a}(t,\xi)\|_{L^\infty_\xi}}_{\text{doesn't grow over time }} + \langle t\rangle^{-\beta} \underbrace{\| \langle x\rangle^{1+} \d^{(3/2-\beta) +} f\|_{L^2}}_{\text{only grow slightly over time }}  \big]. 
\ee

The weighted $L^2$-type space we actually use is the atomic type ($Z$-norm space). More precisely, we view the profile as a superposition of different atoms, which have different scales of localization in both the  frequency space and the physical space, see \eqref{weightednormspace} for the detailed definition. This type of $Z$-norm space was firstly introduced by Ionescu-Pausader \cite{IP2} in the study of Klein-Gordon equations.

We remark that,  we reluctantly use a fractional weighted space because it  doesn't promising to propagate $\langle x\rangle^{2}$ for the profiles over time due to the nonempty spacetime resonance set, which will be discussed in the next step.

$\bullet$\qquad \textit{Step (iii)}\qquad Exploiting properties of the  spacetime resonance set. 

As mentioned in the previous results, we refer readers to \cite{GerMasSha09,GerMasSha12,GerMasSha14,IP1,IP2,IP3,IP} for more details on the spacetime resonance method. At large scale, the methodology is robust. However, at technical level, difficulties  vary from case to case. 

Since $(c_1,c_2)\in (0,1)\times (0,1)\cup (1,\infty)\times (1,\infty)$, we know that the group velocities of $u_1$ and $u_2$ are different as long as one of their frequencies is not parallel to $(1,0,0)$, which is the direction of two characteristic hypersurfaces tangential to each other. More precisely, the  spacetime resonance set is contained in the following set, 
\[
\{(\xi, \eta): \xi=\lambda(1,0,0), \eta= \mu (1,0,0), \lambda, \mu\in \R\}. 
\]

It's  the above structure motivates to consider the asymptotic system \eqref{asympto}. Moreover, we know that $(\xi_2, \xi_3)$ and $ (\eta_2, \eta_3)$ are very small when close to the spacetime resonance set. Therefore, one expects that $\nabla_{x_2, x_3}$ play the role of null structure. 

Moreover, because of this observation, the dyadic decomposition we do for the frequency is   also anisotropic. More precisely, besides the size of $|\xi|$, we  also localize the size of $| \xi_2| $ and $|\xi_3|$. 

 $\bullet$\qquad \textit{Step (iv)}\qquad Splitting on the Fourier side. 

To control the $L^\infty_\xi$-norm and the weighted $L^2$-type norm, we use the Duhamel's formula. The question is reduced to control bilinear forms in different spaces. 
 Because of the structure   of the  spacetime resonance set, we split into the non-resonance case and the resonance case. Thanks to the null structure presented in the system \eqref{june27eqn4}, it enables us to choose the threshold of splitting on the Fourier side  not too small. However, due to the $1+\alpha$ order weight we are propagating for the $Z$-norm, the threshold appears to be critical. 

For the resonance case, we obtain  fine expansion formulas   for the phases, see  \eqref{2023jan29eqn61}, which are crucial. The time resonance set is very delicate due to the possible cancellation because of the opposite sign in the leading term in \eqref{2023jan29eqn61}. Fortunately, this happens only in the frequency comparable case. For this case, we will do more sophisticated localization  in subsection \ref{frecompcase}. Roughly speaking, the key observation for the worst scenario is that  the   frequency is localized inside a very thin annulus instead of a normal  annulus in which the difference between two radiuses is comparable to   radiuses. 

  For the   non-resonance case, we perform normal form transformation if away from the time resonance set and exploiting the difference of group velocities if away from the space resonance set.

The rest of this paper is organized as follows. 
 \begin{enumerate}
\item[$\bullet$] In section \ref{pre}, we introduce   notation and the bootstrap assumption and finish the bootstrap argument with a divided plan.
\item[$\bullet$] In section \ref{Z}, we finish the energy estimate part and     the $L^\infty$-norm estimate part. Moreover, we prove a fixed time $Z$-norm estimate for   nonlinearities, which play a role in later normal form transformation.  
\item[$\bullet$] In section \ref{W}, we finish   the $Z$-norm estimate part.  

\end{enumerate}

 \vo 

\noindent \textbf{Acknowledgment}\qquad   This problem was  introduced to me by my colleague Pin Yu, to him, I am very grateful.  I believe this problem was introduced to him by Sergiu Klainerman. The author thanks Pin Yu for many helpful discussions.   The author acknowledges  support from   NSFC-12141102  and MOST-2020YFA0713003.

\section{A bootstrap argument with a divided plan}\label{pre}

The plan of this section is listed as follows. Firstly, we  introduce the notation used throughout this paper. Secondly, we  state our main bootstrap assumptions and then obtain the $L^\infty$-decay estimate of nonlinear solutions under the bootstrap assumption by using the    linear decay estimate in Lemma \ref{lineardecay}. Moreover, a $L^1_{}$-estimate of   kernel for general Fourier symbol is also provided in Lemma \ref{kernelest}.   Lastly, we finish the bootstrap argument, hence finishing the proof of our main theorem \ref{maintheorem}, by assuming the validity of divided propositions in later sections. 
\subsection{Notation}

For $\xi=(\xi_1, \xi_2, \xi_3)\in \R^3$, we let $\slashed \xi:=(\xi_2,\xi_3)\in \R^2. $ 
For any $x\in \cup_{n\in\Z_+}\R^n$, we use the Japanese bracket $\langle x \rangle$ to denote $(1+|x|^2)^{1/2}$. 

  We  fix an even smooth function $\tilde{\psi}:\R \rightarrow [0,1]$, which is supported in $[-3/2,3/2]$ and equals to ``$1$'' in $[-5/4, 5/4]$. For any $k\in \mathbb{Z}$, we define the cutoff functions $\psi_k, \psi_{\leq k}, \psi_{\geq k}:\cup_{n\in \Z_+}\R^n\longrightarrow \R$ as follows, 
\[
\psi_{k}(x) := \tilde{\psi}(|x|/2^k) -\tilde{\psi}(|x|/2^{k-1}), \quad \psi_{\leq k}(x):= \tilde{\psi}(|x|/2^k)=\sum_{l\leq k}\psi_{l}(x), \quad \psi_{\geq k}(x):= 1-\psi_{\leq k-1}(x).
\]

Let $P_k$, $P_{\leq k}$, and $P_{\geq k}$ be the Fourier multiplier operators with symbols $\psi_{k}(\xi),\psi_{\leq k}(\xi), $ and $\psi_{\geq k}(\xi)$ respectively.  We use $f^{+}$ to denote $f$ and use $f^{-}$ to denote $\bar{f}$. For $k\in\Z$, we use $f_k$ to denote $P_k f$.

  For an integer $k\in\mathbb{Z}$, we use $k_{+}$ to denote $\max\{k,0\}$ and  use $k_{-}$ to denote $\min\{k,0\}$.    For any unit vectors $u, v\in \mathbb{S}^2$, we use $\angle(u, v)$ to denote the angle between $u$ and $v$ and use the convention that   $\angle(u, v)\in[0,\pi] $.   For any $k\in \Z$, we define the following set of integers, 
\be\label{2020april10eqn1}
\begin{split}
\chi_k^1&:=\{(k_1,k_2): k_1, k_2\in \Z, |k_1-k_2|\leq 10, k\leq k_1+10\}, \\
\chi_k^2&:=\{(k_1,k_2): k_1, k_2\in \Z, k_1\leq k_2-10,  |k-k_2|\leq 5\}, \\
\chi_k^3&:=\{(k_1,k_2): k_1, k_2\in \Z, k_2\leq k_1-10,  |k-k_1|\leq 5\}.
\end{split}
\ee

For any $k\in\Z, j\in \Z_+ $, we define the   cutoff function  $\varphi_{j,k}:\cup_{n\in \Z_+}\R^n\longrightarrow \R$ as follows, 
\be\label{cutoffnon}
\varphi_{j,k}= \left\{\begin{array}{ll}
\psi_{\leq -k}(x) & j= -k_{-}\\ 
\psi_j(x) & j>  -k_{-}. 
\end{array}\right. 
\ee
Correspondingly, we define  the localization operators $Q_{j,k}$, $Q_{\leq j,k}$   and  $Q_{> j,k}$ as follows, 
\[
Q_{j,k} f(x):= P_k\big(\varphi_{j,k}(\cdot) P_kf(\cdot)\big) (x), \quad Q_{\leq j,k}  f(x):= \sum_{j'\in   [-k_{-}, j]\cap\Z } Q_{j' ,k} f(x), \quad  Q_{> j,k}:= P_k-Q_{\leq j,k} . 
\]
From \eqref{cutoffnon}, the following partition of unity holds, 
\[
Id=\sum_{k\in \Z} P_k, \quad P_k=\sum_{j\in  [-k_{-}, \infty)\cap\Z } Q_{j,k}
\]
For convenience, we use $f_{j,k}$,  $f_{\leq j, k}$, and $f_{> j, k}$ to abbreviate $Q_{j,k} f$,  $Q_{\leq j,k} f$, and  $Q_{ >j,k} f$  respectively. 

With the above notation, we define the following weighted atomic   $Z$-normed space as follows, 
\be\label{weightednormspace}
\|f \|_{Z}:=\sup_{k\in \Z} \sup_{j\in [-k_{-}, \infty)\cap\Z}  2^{(1+\alpha)j} \|Q_{j,k} f(x)\|_{L^2_x}, \qquad \alpha:=10^{-3}.
\ee

To measure the distance of $|c_1|,|c_2|$ with respect to $1$, we define
\be\label{2023jan28eqn1}
\begin{split}
c  & :=\min\big\{\inf\big\{k: k\in \Z ,  2^k\geq \min\{|1-|c_1||, |1-|c_2||\}\big\}, 0 \big\},  \\ 
 \tilde{c} & := \max\{\inf\big\{k: k\in \Z ,  2^k\geq \max\{ |c_1| ,  |c_2| \}\big\}, 0\}, \qquad c \leq 0 \leq  \tilde{c}.\\ 
\end{split}
\ee

We remark that   constants ``$c$'' and $\tilde{c}$ will not play much role in the argument. Since we will measure the difference between different waves in $x_2, x_3$ directions, it's convenient to use these constants to measure the lower bound and the upper bound.

\subsection{The bootstrap assumption}

From \eqref{march19eqn1}  and \eqref{jan16eqn1}, we can recover $u_a, a\in\{1,2\},$ from the half-waves $U_a, a\in\{1,2\},$ as follows, 
\be\label{march19eqn2}
\p_t u_a = \sum_{\mu\in\{+,-\}} \frac{1}{2} U_a^\mu, \quad u_a = \sum_{\mu\in\{+,-\}} c_{\mu}\Lambda_a^{-1} U_a^\mu, \quad c_{\mu} = \frac{\mu}{2i}, \quad U_a^+:=U_a, \quad U_a^{-}:=\overline{U_a}.
\ee

Recall (\ref{june27eqn4}) and (\ref{march19eqn1}). On the Fourier side, after doing dyadic decomposition for $\xi-\eta$ and $\eta$,  we have 
\be\label{systemeqnprof}
\begin{split}
\p_t \widehat{h_a}(t, \xi)&= \sum_{\mu, \nu \in \{+,-\}}\sum_{k\in \Z} \sum_{(k_1,k_2)\in \chi_k^1\cup \chi_k^2\cup\chi_k^3}  \mathcal{J}_{k;k_1,k_2}^{a;\mu,\nu}(t, \xi),\quad a \in \{1,2\}, \\ 
 \mathcal{J}_{k;k_1,k_2}^{a;\mu,\nu}(t, \xi)&:=  \int_{\R^3} e^{i t \Phi_{\mu, \nu}^a ( \xi ,\eta)} q^a_{\mu, \nu} (\xi-\eta, \eta)\widehat{h_1^{\mu}}(t, \xi-\eta) \widehat{h_2^{\nu}}(t, \eta) \psi_k(\xi)\psi_{k_1}(\xi-\eta)\psi_{k_2}(\eta)  d\eta, 
\end{split}
\ee
where  the oscillating  phases $\Phi_{\mu, \nu}^a ( \xi ,\eta), a\in \{1,2\}$,  are defined as follows, 
\be\label{phases}
\begin{split}
\Phi_{\mu, \nu}^a ( \xi ,\eta)&:= \Lambda_a(\xi) - \mu\Lambda_1(\xi-\eta)-\nu \Lambda_2(\eta), \\ 
\forall \xi \in \R^3, \quad \Lambda_1(\xi)&:=\sqrt{\xi_1^2 + \xi_2^2 + \xi_3^2}, \quad \Lambda_2(\xi):=\sqrt{\xi_1^2 + c_1^2 \xi_2^2 + c_2^2\xi_3^2}. 
\end{split}
\ee
  and symbols $q^a_{\mu, \nu} (\xi-\eta, \eta),\mu, \nu \in \{+,-\}, a\in \{+,-\},$ are given as follows,
\be\label{june21eqn51}
 q^a_{\mu, \nu} (\xi-\eta, \eta) = \frac{1}{4} + c_{\mu} c_{\nu} \frac{(\xi_1-\eta_1)}{\Lambda_1(\xi-\eta)}\frac{\eta_1}{\Lambda_2( \eta)} - \sum_{a\in\{1,2\}, i,j\in\{2,3\}} \mathfrak{q}^a_{i,j}(\xi-\eta, \eta)\frac{c_{\mu} c_{\nu} (\xi_i-\eta_i)\eta_j }{\Lambda_1(\xi-\eta) \Lambda_2( \eta)}.
\ee

Let $\delta\in (0,10^{-10}) $ be a small absolute constant.  We   make the following   bootstrap assumption, 
\be\label{bootstrap}
\begin{split}
 \sup_{t\in[0, T]}\sum_{a=1,2} \langle  t \rangle^{-\delta }&   \| h_a(t)\|_{H^{N_0}}      + \langle  t \rangle^{-\delta } \| h_a(t)\|_{Z} + \sup_{k\in \Z} 2^{8k_+} \|\widehat{h_a}(t,\xi)\psi_k(\xi)\|_{L^\infty_\xi}  \lesssim \epsilon_1:=\epsilon^{5/6}_0, 
\end{split}
\ee
where  the  $Z$-normed space was defined in \eqref{weightednormspace}.

\subsection{ $L^\infty_x$-decay estimate  under the bootstrap assumption}

Let $t\in [2^{m-1}, 2^m]\subset [0, T]$.   
From the Sobolev embedding and the bootstrap assumption \eqref{bootstrap}, we can rule out the case $k\notin [-m/2, m/(N_0-12)] $ as follows, 
\be\label{decayrough}
\begin{split}
& \sum_{k\in \Z, k\notin [-m/2, m/(N_0-12)] } 2^{-  k +5k_+ }\| e^{-i t\Lambda_a} P_k(h_a(t))\|_{L^\infty} \\
&\lesssim \sum_{k\in \Z, \notin [-m/2, m/(N_0-12)] } 2^{- k+  5k_+ } \min\{2^{3k}\|\widehat{h_a}(t, \xi)\psi_k(\xi)\|_{L^\infty_\xi }, 2^{3k/2} \|\widehat{h_a}(t, \xi)\psi_k(\xi)\|_{L^2_\xi }  \}\\
&\lesssim 2^{-m} \epsilon_1. 
\end{split}
\ee

It remains to consider the case $k\in  [-m/2, m/(N_0-12)]\cap \Z$. For this case,  from the linear decay estimates  \eqref{linearwavedecay}   in Lemma \ref{lineardecay}, we have 
\be\label{decayprof}
 \| e^{-i t\Lambda_i} P_k(h_a(t))\|_{L^\infty}\lesssim 2^{-m+   k  -5k_+}  \epsilon_1.
\ee
To sum up, after combining the obtained estimates \eqref{decayrough} and \eqref{decayprof}, we have
\be\label{decayestimatefinal}
\sup_{t\in [0, T]}\sum_{k\in \Z, a=1,2} 2^{- k + 5k_+} \langle t \rangle  \| e^{-i t\Lambda_a} P_k(h_a(t))\|_{L^\infty}\lesssim \epsilon_1.
\ee

In later study of the High $\times$ High type interaction, we will localize    $\xi$ inside a thin annulus. Hence, in the following Lemma, we provide a super-localized version of the linear decay estimate.

\begin{lemma}\label{lineardecay}
Let $\delta\in (0, 10^{-10})$ be some fixed absolute constant,  $a\in\{1,2\}, t\in [2^{m-1}, 2^m]\subset \R, x\in \R^3, \mu \in\{+,-\}$, $k\in\mathbb{Z}$,   $n,\tilde{k}\in \mathbb{Z}$, s.t., $|t|\gg 1$,  $\tilde{k}\leq k$, $k, \tilde{k}\geq - m + \delta m $, $n\in [2^{k-10-\tilde{c} }, 2^{k+10+\tilde{c}}]\cap \Z $, we have 
\be\label{linearwavedecay}
\begin{split}
 &  \big|\int_{\R^3}  e^{ i x\cdot \xi-i \mu t  \Lambda_a(\xi)   }     m(\xi) \widehat{f}(\xi)  \psi_{\tilde{k} }(\Lambda_a(\xi)- n )  d \xi \big| \lesssim   2^{-m}\|m(\xi)\|_{\mathcal{S}^\infty_k}  \big[ 2^{ k+\tilde{k}+k_{+}}\|\widehat{f}(\xi )  \psi_k(\xi)   \|_{L^\infty_\xi }    \\ 
&\quad   +   2^{- (\alpha-4\delta)  m/2+(\alpha+\delta) k /2}2^{(2k +\tilde{k})/2 } \big(\sup_{j\in[-k_{-},\infty)\cap \Z} 2^{(1+\alpha)j} \|T^a_{\tilde{k},n} Q_{j,k}f \|_{L^2}\big) \big],
  \end{split}
\ee
where the Fourier multiplier operator $T^a_{\tilde{k},n}$ is defined by the symbol  $\psi_{\tilde{k} }(\Lambda_a(\xi)- n )$.

\end{lemma}
\begin{proof}
Let $n\in [2^{k-10}, 2^{k+10}]$ be  fixed. Note that, as a result of direct computations, we have  
\[
\forall a \in\{1,2\}, \quad N_a(\xi):=\nabla_\xi \Lambda_a(\xi),  \quad N_1(\xi) = \frac{(\xi_1, \xi_2, \xi_3)}{|\Lambda_1(\xi)|}, \quad  N_2(\xi) = \frac{(\xi_1, c_1^2\xi_2, c_2^2 \xi_3)}{|\Lambda_2(\xi)|}, \quad |N_a(\xi)|\sim 1. 
\]
Based on the possible size of $|x|$, we split into two  cases as follows. 

$\bullet$\qquad If $|x|\leq 2^{m-10}$. \qquad Note that, for this case, we have
\[
\big|\nabla_\xi\big[   x\cdot \xi-  \mu t  \Lambda_a(\xi)  \big]\big|\sim 2^m. 
\]

Base on the possible size of $j$, we split into two parts as follows.
\[
\begin{split}
\int_{\R^3} & e^{ i x\cdot \xi-i \mu t  \Lambda_a(\xi)   }     m(\xi) \widehat{f}(\xi)  \psi_{\tilde{k} }(\Lambda_a(\xi)- n )  d \xi= H_1+ H_2,\\
H_1&=\int_{\R^3}  e^{ i x\cdot \xi-i \mu t  \Lambda_a(\xi)   }     m(\xi) \widehat{f_{< m-\delta m, k}}(\xi)  \psi_{\tilde{k} }(\Lambda_a(\xi)- n )  d \xi,\\ 
H_2&=\int_{\R^3}  e^{ i x\cdot \xi-i \mu t  \Lambda_a(\xi)   }     m(\xi) \widehat{f_{\geq m-\delta m,k}}(\xi)  \psi_{\tilde{k} }(\Lambda_a(\xi)- n )  d \xi. 
\end{split}
\]
Recall that $k,\tilde{k}\geq - m+\delta m$. Note that, by doing integration in $\xi$ once for $H_1$, we gain at least $2^{-\delta m }$.  Therefore, after doing integration by parts in $\xi$ $\delta^{-2}$ times, the following estimate holds from the volume of support of $\xi$,
\be\label{june27eqn1}
|H_1(t,x)|\lesssim  2^{-10m+2k+\tilde{k}} \| \widehat{f_{< m-\delta m, k}}(\xi) \|_{L^\infty_\xi} \lesssim  2^{-10m+2k+\tilde{k}} \| \widehat{f  }(\xi) \psi_k(\xi) \|_{L^\infty_\xi}.
\ee
From the volume of support of $\xi$ and  the Cauchy-Schwarz inequality, we have
\be\label{june27eqn2}
\begin{split}
|H_1(t,x)|&\lesssim   \sum_{j\geq m-\delta m }2^{(2k+\tilde{k})/2 }  \|T^a_{\tilde{k},n} Q_{j,k}f \|_{L^2} \\ 
& \lesssim 2^{-(1+\alpha)(1-\delta)m +(2k+\tilde{k})/2  }  \big(\sup_{j\in[-k_{-},\infty)\cap \Z} 2^{(1+\alpha)j} \|T^a_{\tilde{k},n} Q_{j,k}f \|_{L^2}\big).
\end{split}
\ee
After combining the obtained estimates \eqref{june27eqn1} and \eqref{june27eqn2}, we have
\be\label{june27eqn11}
\begin{split}
 &\Big| \int_{\R^3}  e^{ i x\cdot \xi-i \mu t  \Lambda_a(\xi)   }     m(\xi) \widehat{f}(\xi)  \psi_{\tilde{k} }(\Lambda_i(\xi)- n )  d \xi\big|\\
&\lesssim 2^{-m} \big[ 2^{ 2k+\tilde{k}} \| \widehat{f  }(\xi) \psi_k(\xi) \|_{L^\infty_\xi} +  2^{-( \alpha -\delta-\alpha\delta)m +(2k+\tilde{k})/2  }  \big(\sup_{j\in[-k_{-},\infty)\cap \Z} 2^{(1+\alpha)j} \|T^a_{\tilde{k},n} Q_{j,k}f \|_{L^2}\big) \big].
\end{split}
\ee

$\bullet$\qquad If $|x|\geq 2^{m-10}$. \qquad 
 Note that, as a result of direct computations, we have 
\[
 N_a(\xi)\times \nabla_\xi\big(  x\cdot \xi-  \mu t  \Lambda_a(\xi)   \big) =  N_a(\xi)\times x.
\]

   Let $\bar{l} $ be the least integer such that  $2^{ \bar{l} } \geq   2^{ -k/2} |x|^{-1/2}$. From the volume of support of $\xi$ and the Sobolev embedding in angular variables, we have 
\be\label{july26eqn32}
    \Big| 
\int_{\R^3}    e^{ i x\cdot \xi-i \mu t  \Lambda_a(\xi) }  
 \widehat{f}(\xi) m(\xi)   \psi_{\tilde{k} }(\Lambda_a(\xi)- n )   \psi_{\leq  \bar{l} }(\tilde{x}\times N_a(\xi)) 
  d \xi \Big| \lesssim   2^{\tilde{k} + 2k+2\bar{l}}\|m(\xi)\|_{\mathcal{S}^\infty_k}   \|\widehat{f}(\xi )  \psi_k(\xi)   \|_{L^\infty_\xi }  .
\ee

For the case when the angle is localized around $2^l$ where $l> \bar{l}$, we first do integration by parts in $\xi$  along directions perpendicular to $N_a(\xi)$   once. As a result, after splitting $f$ into two parts based on the size of $j$ when $\nabla_\xi$ hits $f$,  we have
\be
\begin{split}
& \int_{\R^3}    e^{ i x\cdot \xi-i \mu t  \Lambda_a(\xi) }  
 \widehat{f}(\xi) m(\xi)   \psi_{\tilde{k} }(\Lambda_a(\xi)- n ) \psi_{ {l} }(\tilde{x}\times N_a(\xi)) 
  d \xi =I_l^0 + I_l^1+I_l^2,\\
  &   I_l^0=\int_{\R^3}  e^{ ix\cdot \xi - i \mu t \Lambda_a(\xi)    }      \big[ ( N_a(\xi) \times 
 \nabla_\xi)\cdot  \big(   \frac{i x \times N_a(\xi) }{|x\times N_a(\xi) |^2}     m(\xi)    \psi_{\tilde{k} }(\Lambda_a(\xi)- n ) \psi_{ {l} }(\tilde{x}\times N_a(\xi))   \big)\big]  \widehat{f}(\xi) 
  d \xi, \\ 
  &  I_l^1=  \int_{\R^3}  e^{ ix\cdot \xi - i \mu t \Lambda_a(\xi)   }   \big(  \frac{i x \times N_a(\xi) }{|x\times N_a(\xi) |^2}    \cdot    ( N_a(\xi) \times 
 \nabla_\xi)\widehat{f_{\leq m+l,k}   }(\xi)   m(\xi)    \psi_{\tilde{k} } (\Lambda_a(\xi)- n ) \psi_{ {l} }(\tilde{x}\times N_a(\xi))  
  d \xi,  \\
    &  I_l^2=  \int_{\R^3}  e^{ ix\cdot \xi - i \mu t \Lambda_a(\xi)   }   \big(  \frac{i x \times N_a(\xi) }{|x\times N_a(\xi) |^2}    \cdot    ( N_a(\xi) \times 
 \nabla_\xi)\widehat{f_{> m+l,k}   }(\xi)   m(\xi)    \psi_{\tilde{k} }(\Lambda_a(\xi)- n ) \psi_{ {l} }(\tilde{x}\times N_a(\xi)) 
  d \xi,  \\
  \end{split}
  \ee
From the volume of support of $\xi$ and the Cauchy-Schwarz inequality, we have 
\be
\begin{split}
\sum_{l> \bar{l}}\big| I_l^2 \big| & \lesssim  \sum_{l> \bar{l}} \sum_{j\geq m+l} |x|^{-1} 2^{-l} 2^{(2k+2l+\tilde{k})/2 +j }   \|T^a_{\tilde{k},n} Q_{j,k}f \|_{L^2}  \\ 
& \lesssim \sum_{l> \bar{l}} \sum_{j\geq m+l} |x|^{-1}  2^{(2k +\tilde{k})/2 -\alpha j  }   \big(\sup_{j\in[-k_{-},\infty)\cap \Z} 2^{(1+\alpha)j} \|T^a_{\tilde{k},n} Q_{j,k}f \|_{L^2}\big) \\ 
&\lesssim 2^{-(1+\alpha/2)m+\alpha k /2}2^{(2k +\tilde{k})/2 }  \big(\sup_{j\in[-k_{-},\infty)\cap \Z} 2^{(1+\alpha)j} \|T^a_{\tilde{k},n} Q_{j,k}f \|_{L^2}\big).
\end{split}
\ee

 For  $  I_l^0 $ and $ I_l^1$, we  do integration by parts in $\xi$  along directions perpendicular to $N_a(\xi)$    one more time. As a result, we have 
  \be\label{aug2eqn3}
  \begin{split}
   I_l^0=\int_{\R^3}  e^{ ix\cdot \xi - i \mu t \Lambda_a(\xi)    } & ( N_a(\xi) \times \nabla_\xi)\cdot\Big[ \frac{i x \times N_a(\xi) }{|x\times  N_a(\xi) |^2}   \big[ (N_a(\xi) \times 
 \nabla_\xi)\cdot  \big(  \frac{i x \times N_a(\xi) }{|x\times N_a(\xi) |^2}  \\ 
 & \times    m(\xi)    \psi_{\tilde{k} }(\Lambda_a(\xi)- n )   \psi_{ {l} }(\tilde{x}\times N_a(\xi)) \big)  \widehat{f}(\xi) \big]\Big]
  d \xi,\\ 
   I_l^1=\int_{\R^3}  \int_{\R^3}  e^{ ix\cdot \xi - i \mu t \Lambda_a(\xi)   }&  ( N_a(\xi) \times \nabla_\xi)\cdot\Big[ \frac{i x \times N_a(\xi) }{|x\times  N_a(\xi) |^2}   \big(  \frac{i x \times N_a(\xi)  }{|x\times  N_a(\xi) |^2}    \cdot    ( N_a(\xi) \times 
 \nabla_\xi)\widehat{ f_{ \leq  m+l,k} }(\xi) \\ 
   &\times  m(\xi)    \psi_{\tilde{k} } (\Lambda_a(\xi)- n )   \psi_{ {l} }(\tilde{x}\times N_a(\xi)) \big) \big]
  d \xi. 
  \end{split}
\ee

From the volume of support of $\xi$,  the Cauchy-Schwarz inequality,  and the Sobolev embedding on sphere $\mathbb{S}^2$,  we have 
\be\label{june26eqn1}
\begin{split}
\sum_{l> \bar{l}}\big| I_l^0 \big| &  \lesssim  \sum_{l> \bar{l}} |x|^{-2} 2^{-2 l} 2^{- k- l}  \|m(\xi)\|_{\mathcal{S}^\infty_k}   \big[2^{- k- l}  2^{(2k+2l+\tilde{k})    }   \|\widehat{f}(\xi )  \psi_k(\xi)   \|_{L^\infty_\xi } \\ 
&+ 2^{(2k+2l+\tilde{k})/2  }   \| \nabla_\xi \widehat{ f}(\xi )    \psi_{\tilde{k} }(\Lambda_a(\xi)- n )  \psi_{ {l} }(\tilde{x}\times N_a(\xi))  \|_{L^2_\xi }  \big]   \\ 
& \lesssim   2^{-m+k+\tilde{k}}\|m(\xi)\|_{\mathcal{S}^\infty_k}    \|\widehat{f}(\xi )  \psi_k(\xi)   \|_{L^\infty_\xi }  + \sum_{l> \bar{l}}   |x|^{-2} 2^{-2 l} 2^{- k- l}   \big(\sup_{j\in[-k_{-},\infty)\cap \Z} 2^{(1+\alpha)j} \|T^a_{\tilde{k},n} Q_{j,k}f \|_{L^2}\big)  \\
& \times \|m(\xi)\|_{\mathcal{S}^\infty_k}  2^{(2k+2l+\tilde{k})/2  }  \big[  \sum_{  j\in [-k_{-}, m+l]\cap\Z }  2^{k+(1-\delta)l+(1-\alpha)j}  +   \sum_{  j\in [  m+l, \infty)\cap\Z }  2^{-\alpha j} \big]\\ 
& \lesssim  2^{-m}\|m(\xi)\|_{\mathcal{S}^\infty_k}    \big[ 2^{ k+\tilde{k}}\|\widehat{f}(\xi )  \psi_k(\xi)   \|_{L^\infty_\xi }\\ 
&  +  2^{- (\alpha-\delta)  m/2+(\alpha+\delta) k /2}2^{(2k +\tilde{k})/2 }   \big(\sup_{j\in[-k_{-},\infty)\cap \Z} 2^{(1+\alpha)j} \|T^a_{\tilde{k},n} Q_{j,k}f \|_{L^2}\big)\big].
\end{split}
\ee
Similarly, from the volume of support of $\xi$,  the Cauchy-Schwarz inequality,  and the Sobolev embedding on sphere $\mathbb{S}^2$,  we have 
\be\label{june26eqn2}
\begin{split}
\sum_{l> \bar{l}}\big| I_l^1 \big| &  \lesssim  \sum_{l> \bar{l}} |x|^{-2} 2^{-2 l}    \|m(\xi)\|_{\mathcal{S}^\infty_k}   \big[2^{- k- l}  2^{(2k+2l+\tilde{k})/2    }   \| \widehat{ f_{ \leq  m+l,k} }(\xi)\psi_{ {l} }(\tilde{x}\times N_a(\xi))   \|_{L^2_\xi }\\
&+ 2^{(2k+2l+\tilde{k})/2  }   \| \nabla_\xi  \widehat{ f_{ \leq  m+l,k} }(\xi)   \psi_{\tilde{k} }(\Lambda_a(\xi)- n )  \|_{L^2_\xi }  \big] \\ 
& \lesssim  \sum_{l> \bar{l}}   \sum_{  j\in [-k_{-}, m+l]\cap\Z }  2^{-2m-2 l}        2^{(2k+2l+\tilde{k})/2 -\delta l +(1-\alpha)j } \|m(\xi)\|_{\mathcal{S}^\infty_k} \big(\sup_{j\in[-k_{-},\infty)\cap \Z} 2^{(1+\alpha)j} \|T^a_{\tilde{k},n} Q_{j,k}f \|_{L^2}\big)  \\ 
& \lesssim   2^{-m- (\alpha-\delta)  m/2+(\alpha+\delta) k /2}2^{(2k +\tilde{k})/2 } \big(\sup_{j\in[-k_{-},\infty)\cap \Z} 2^{(1+\alpha)j} \|T^a_{\tilde{k},n} Q_{j,k}f \|_{L^2}\big).
\end{split}
\ee
To sum up, our desired estimate \eqref{linearwavedecay} holds after combining the obtained estimates \eqref{june27eqn11}, \eqref{july26eqn32}, \eqref{june26eqn1}, and \eqref{june26eqn2}.
\end{proof}
 
 \begin{lemma}\label{kernelest}
Let $\delta\in (0,10^{-10})$ be some fixed absolute constant. For any $m(\xi, \eta)\in C^4(\R^3\times \R^3)$, the following estimate holds for the kernel of the symbol $m(\xi, \eta)$,
\be\label{july27eqn3}
\begin{split}
\sup_{k,k_1\in \R^3 }\|&\mathcal{F}^{-1}_{(\xi, \eta)\rightarrow (x,y)}\big[ m(\xi, \eta)\psi_k(\xi)\psi_{k_1}(\eta)\big](x,y)\|_{L^1_{x,y}}\\
&\lesssim \big(\sum_{|\alpha|, |\beta|\leq 3} \| |\xi|^{\alpha}|\eta|^{\beta}\nabla_\xi^\alpha \nabla_\eta^\beta m(\xi, \eta) \psi_{[k-1,k+1]}(\xi)\psi_{[k_1-1, k_1+1]}(\eta)\|_{L^\infty_{\xi, \eta}}\big)^{1-\delta}\\ 
&\times \big(\sum_{|\alpha|, |\beta|\leq 4} \| |\xi|^{\alpha}|\eta|^{\beta}\nabla_\xi^\alpha \nabla_\eta^\beta m(\xi, \eta) \psi_{[k-1,k+1]}(\xi)\psi_{[k_1-1, k_1+1]}(\eta)\|_{L^\infty_{\xi, \eta}}\big)^{\delta}.
\end{split}
\ee
 \end{lemma}
 \begin{proof}
Note that
\[ 
\begin{split}
K_{k,k_1}(x,y)& = \int_{\R^3} \int_{\R^3 } e^{i x\cdot  \xi + i y \cdot \eta} m(\xi, \eta) \psi_k(\xi)\psi_{k_1}(\eta)  d \xi d \eta, \\ 
(i x)^\alpha (i y )^{\beta} K_{k,k_1}(x,y)& = \int_{\R^3} \int_{\R^3 } e^{i x\cdot  \xi + i y \cdot \eta} \nabla_{\xi}^\alpha \nabla_\eta^\beta\big( m(\xi, \eta) \psi_k(\xi)\psi_{k_1}(\eta)  \big) d \xi d \eta.
\end{split}
\]
From the above equality and the volume of support of $\xi, \eta$, we have 
\[ 
\begin{split}
\langle 2^k x \rangle^{3} \langle 2^{k_1} y \rangle^{3} \big|K_{k,k_1}(x,y) \big| & \lesssim \sum_{|\alpha|, |\beta|\leq 3} 2^{3k+3k_1} \| |\xi|^{\alpha}|\eta|^{\beta}\nabla_\xi^\alpha \nabla_\eta^\beta m(\xi, \eta) \psi_{[k-1,k+1]}(\xi)\psi_{[k_1-1, k_1+1]}(\eta)\|_{L^\infty_{\xi, \eta}}, \\ 
\langle 2^k x \rangle^{4} \langle 2^{k_1} y \rangle^{4} \big|K_{k,k_1}(x,y) \big| & \lesssim \sum_{|\alpha|, |\beta|\leq 4} 2^{3k+3k_1} \| |\xi|^{\alpha}|\eta|^{\beta}\nabla_\xi^\alpha \nabla_\eta^\beta m(\xi, \eta) \psi_{[k-1,k+1]}(\xi)\psi_{[k_1-1, k_1+1]}(\eta)\|_{L^\infty_{\xi, \eta}}. \\ 
\end{split}
\]
Hence our  desired estimate \eqref{july27eqn3}  holds   after interpolating the above two estimates.
 \end{proof}
 \subsection{Proof of the main theorem \ref{maintheorem} }

From the estimate \eqref{june21eqn1} in Proposition \ref{energyest}, the estimate \eqref{july24eqn1} in Proposition \ref{West}, and the estimate \eqref{july23eqn91} in Proposition  \ref{Zest}, we have 
\[
 \sup_{t\in[0, T]}\sum_{a=1,2} \langle  t \rangle^{-\delta }    \| h_a(t)\|_{H^{N_0}}      + \langle  t \rangle^{-\delta } \| h_a(t)\|_{Z } + \sup_{k\in \Z}\|\langle \xi \rangle^{8}\widehat{h_a}(t, \xi)\psi_{k}(\xi)\|_{L^\infty_\xi } \lesssim \epsilon_0+\epsilon_1^2\lesssim \epsilon_0.
\]
Hence the bootstrap assumption is improved, i.e., nonlinear solutions exist globally. Moreover the desired estimate \eqref{finalestimate} holds from the above bootstrap assumption and the obtained estimate \eqref{decayestimatefinal}. 

As a byproduct, from the estimate \eqref{july23eqn91} in Proposition  \ref{Zest}, for any $t\in [0, \infty)$, we have
\[
\sum_{a=1,2}\sup_{\xi\in \R^3} \Big| \widehat{U_a}(t, \xi) - e^{-i t\Lambda_a(\xi)}\widehat{U_{\infty}^a}( \xi)  \Big|\lesssim \langle t \rangle^{-\delta}, 
\]
\[
\widehat{U_{\infty}^a}( \xi) :=\widehat{U_{a} }( \xi)  +\int_0^\infty e^{is \Lambda_a(\xi)}  \sum_{\mu, \nu \in \{+,-\}}\sum_{k\in \Z} \sum_{(k_1,k_2)\in \chi_k^1\cup \chi_k^2\cup\chi_k^3}  \mathcal{J}_{k;k_1,k_2}^{a;\mu,\nu}(s, \xi) d s,  
\]
where  $\mathcal{J}_{k;k_1,k_2}^{a;\mu,\nu}(s, \xi) $ are defined in \eqref{systemeqnprof}. That's to say, the nonlinear solution scatters to a linear solution in $L^\infty_\xi$ as time goes to infinity.

\section{Energy estimate and the $L^\infty_\xi$-estimate of profiles}\label{Z}

In this section, we first close the bootstrap assumption for the energy estimate and the $L^\infty_\xi$-estimate, which are relatively easier than the $Z$-norm estimate part. This allows us to focus on the $Z$-norm estimate part for the rest of this paper. Moreover, as a preparation for readers for more complicated analysis in the next section, we give a fixed time $Z$-norm estimate for nonlinearities, which is not only useful when we do normal form transformation in the next section but also contains some standard analysis. 

Since the nonlinear solution has the sharp $1/t$ decay rate over time, the energy estimate part is standard. More precisely, we have 

\begin{proposition}\label{energyest}
Under the bootstrap assumption \eqref{bootstrap}, for any $t_1, t_2\in [2^{m-1}, 2^m]\subset[0, T], m\in \Z_+$, we have
\be\label{june21eqn1}
 \sum_{a=1,2} \| h_a(t_2)- h_a(t_1) \|_{H^{N_0}} \lesssim 2^{\delta m }\epsilon_1^2.
\ee 
\end{proposition}
\begin{proof}
Recall (\ref{june27eqn4}) and (\ref{march19eqn1}). For any $  t_1, t_2\in [2^{m-1}, 2^m]\subset[0, T],$ s.t., $t_1\leq t_2$, we have
\be\label{june20eqn1}
\begin{split}
\sum_{a=1,2}  \| h_a(t_2)- h_a(t_1) \|_{H^{N_0}}\lesssim  &\sum_{a=1,2} \sum_{i,j\in\{2,3\} }  \int_{t_1}^{t_2} \| \p_t u_1(t) \p_t u_2(t) - \p_{x_1} u_1(t) \p_{x_1}u_2(t)\|_{H^{N_0}}\\
 & + \|\mathcal{Q}^a_{i,j}(\nabla_{x_i} u_1, \nabla_{x_j} u_2 )\|_{H^{N_0}} d t. 
 \end{split}
\ee
After combining the above estimate with  the $L^\infty$-decay estimate \eqref{decayestimatefinal}, the estimate \eqref{addsymbolest}, and the $L^2-L^\infty$ type bilinear estimate, we have
\be\label{june21eqn2}
\sum_{a=1,2}  \| h_a(t_2)- h_a(t_1) \|_{H^{N_0}}\lesssim  \int_{t_1}^{t_2} 2^{-m+\delta m } \epsilon_1^2 d t\lesssim 2^{\delta m} \epsilon_1^2. 
\ee
Hence finishing the proof of our desired estimate \eqref{june21eqn1}.
\end{proof}

\subsection{A fixed time    $Z$-norm estimate of   nonlinearities}
In this subsection, we first prove a fixed time    $Z$-norm estimate of   nonlinearities, which can be used directly later when we do normal form transformation. Moreover, from the final estimate \eqref{july13eqn58}, we can see that, without using the oscillation in time and careful analysis of phases based on the sign, we are not off much from the  goal of $\langle t \rangle^{-1+\delta}\epsilon_1^2$, which is sufficient for closing the bootstrap argument.
 \begin{lemma}\label{fixedtimeest}
Under the bootstrap assumption \eqref{bootstrap}, for any $t \in [2^{m-1}, 2^m]\subset[0, T], m\in \Z_+$, $\mu, \nu\in\{+, -\}$,  we have
\be\label{july13eqn58}
\sup_{k\in \Z}\sum_{a=1,2} \sum_{(k_1,k_2)\in \cup_{i=1,2,3} \chi_k^i} \|\mathcal{F}^{-1}\big[   \mathcal{J}_{k;k_1,k_2}^{a;\mu,\nu}(t, \xi) \big]   \|_{Z} \lesssim  2^{-3m/4+(\alpha+10\delta)m}\epsilon_1^2.
\ee
\end{lemma}
\begin{proof}
Recall \eqref{systemeqnprof}. Note that, after doing space localization for two inputs, we have
\be\label{july13eqn70}
\begin{split}
\mathcal{F}^{-1}\big[   \mathcal{J}_{k;k_1,k_2}^{a;\mu,\nu}(t, \xi) \big] &  =\sum_{j_1\in [-k_{1,-}, \infty)\cap \Z, j_2\in [-k_{2,-}, \infty)\cap \Z } \mathcal{F}^{-1}\big[   \mathcal{J}_{k;k_1,k_2}^{a;\mu,\nu;j_1, j_2}(t, \xi) \big],  \\ 
 \mathcal{F}^{-1}\big[   \mathcal{J}_{k;k_1,k_2}^{a;\mu,\nu;j_1, j_2}(t, \xi) \big]  & = \int_{\R^3}  \int_{\R^3} e^{ i x\cdot \xi +  i t \Phi_{\mu, \nu}^a ( \xi ,\eta)} q^a_{\mu, \nu} (\xi-\eta, \eta)\widehat{h_{1;j_1,k_1}^{\mu}}(t, \xi-\eta) \widehat{h_{2;j_2,k_2}^{\nu}}(t, \eta) \psi_k(\xi)   d\eta d \xi.
\end{split}
\ee

Based on the possible size of $x$, i.e., the size of $j$, we split into two cases as follows.

$\bullet$\qquad If $j\geq  \max\{m, -k_{-}\}+ \delta (m+j) .$

Based on the possible sizes of $j_1, j_2$, we split into two sub-cases as follows.

$\oplus$\qquad If $\min\{j_1,j_2\}\geq (1-\delta)j.$

For this case, from the $L^2-L^\infty$ type bilinear estimate and the $L^\infty\rightarrow L^2$-type  Sobolev embedding, we have
\be\label{july13eqn72}
\begin{split}
&\sum_{(k_1,k_2)\in  \cup_{i=1,2,3} \chi_k^i } \sum_{\min\{j_1,j_2\}\geq (1-\delta)j} 2^{(1+\alpha)j}\|P_k\big[ \varphi_{j,k}(\cdot)\mathcal{F}^{-1}\big[   \mathcal{J}_{k;k_1,k_2}^{a;\mu,\nu;j_1, j_2}(t, \xi) \big]\big](x)\|_{L^2}\\ 
& \lesssim  \sum_{(k_1,k_2)\in  \cup_{i=1,2,3} \chi_k^i }  \sum_{\min\{j_1,j_2\}\geq (1-\delta)j}  2^{(1+\alpha)j} \min\big\{\|h_{1;j_1,k_1}^{\mu}(t,x)\|_{L^2} \\ 
&\times  \|h_{2;j_2,k_2}^{\mu}(t,x)\|_{L^\infty_x },\|h_{1;j_1,k_1}^{\mu}(t,x)\|_{L^\infty_x}\|h_{2;j_2,k_2}^{\mu}(t,x)\|_{L^2_x } \big\}\\ 
&\lesssim   \sum_{(k_1,k_2)\in  \cup_{i=1,2,3} \chi_k^i }  \sum_{\min\{j_1,j_2\}\geq (1-\delta)j}  2^{(1+\alpha)j} \big[ 2^{-(1+\alpha)j_1+\delta m } 2^{3k_2/2}  \\ 
&\times \min\{2^{-N_0 k_{2,+}+\delta m }, 2^{-(1+\alpha)j_2+\delta m } \} + 2^{-(1+\alpha)j_2 +\delta m } 2^{3k_1/2}\min\{2^{-N_0 k_{1,+}+\delta m }, 2^{-(1+\alpha)j_1+\delta m } \} \big]\epsilon_1^2 \\
&\lesssim \sum_{\min\{j_1,j_2\}\geq (1-\delta)j}  2^{(1+\alpha)j}  2^{-(1+\alpha)j_2-(1+\alpha)j_1+3\delta m  } \lesssim 2^{-m-10\delta m}\epsilon_1^2.\\
\end{split}
\ee

$\oplus$\qquad If $\min\{j_1,j_2\}\leq (1-\delta)j.$

Without loss of generality, we assume that $j_1=\min\{j_1,j_2\}.$ Otherwise, we change coordinates $(\xi,\eta)\longrightarrow (\xi, \xi-\eta)$ to switch the roles of $\xi-\eta$ and $\eta$. Note that, for the case we are considering, we have
\be
|\nabla_\xi\big[    x\cdot \xi +    t \Phi_{\mu, \nu}^a ( \xi ,\eta)\big]| \sim 2^{j}. 
\ee
Therefore, by doing integration by parts in $\xi$ once, we gain $2^{-j}$ and lose at most $\max\{2^{j_1}, 2^{-k_1}, 2^{-k}\}$. Therefore, we gain at least $2^{-\delta j}$ in this process. After doing integration by parts in ``$\xi$'' $\delta^{-2}$ times, the following estimate holds from the $L^2-L^\infty$ type bilinear estimate and the $L^\infty\rightarrow L^2$-type  Sobolev embedding,
\[
\sum_{(k_1,k_2)\in \cup_{i=1,2,3} \chi_k^i} \sum_{ \begin{subarray}{c}
j_i\in [-k_{i,-}, \infty)\cap \Z, i=1,2\\
\min\{j_1,j_2\}\leq (1-\delta)j \\
\end{subarray} } 2^{(1+\alpha)j}\|P_k\big[ \varphi_{j,k}(\cdot)\mathcal{F}^{-1}\big[   \mathcal{J}_{k;k_1,k_2}^{a;\mu,\nu;j_1, j_2}(t, \xi) \big]\big](x)\|_{L^2}
\]
 \[
\lesssim \sum_{(k_1,k_2)\in \cup_{i=1,2,3} \chi_k^i}  \sum_{ \begin{subarray}{c}
j_i\in [-k_{i,-}, \infty)\cap \Z, i=1,2\\
\min\{j_1,j_2\}\leq (1-\delta)j \\
\end{subarray} }   2^{-100j }  2^{-(1+\alpha)j_1+\delta m } 2^{3k_2/2}\min\{2^{-N_0 k_{2,+}+\delta m }, 2^{-(1+\alpha)j_2+\delta m } \} \epsilon_1^2 
\]
\be\label{july13eqn52}
\lesssim \sum_{ \begin{subarray}{c}
j_i\in [-k_{i,-}, \infty)\cap \Z, i=1,2
\end{subarray} }  2^{-10 j -j_1-j_2}\epsilon_1^2\lesssim 2^{-10 m  }\epsilon_1^2.
\ee

$\bullet$\qquad If $j\leq  \max\{m, -k_{-}\}+ \delta (m+j) .$

Note that, for this case, we have $j\leq (1-\delta)^{-1}\big(\max\{m, -k_{-}\}+ \delta m \big) $. As a result of direct computations, we have 
\be\label{july13eqn1}
\big|\nabla_{\eta } \big(  \Phi_{\mu, \nu}^a(\xi, \eta) \big)\big| = \big|-\mu  \frac{ \eta-\xi}{|\xi-\eta|} -   \nu \frac{ (\eta_1,c_1^2\eta_2, c_2^2\eta_3)}{ \Lambda_2(\eta)} \big|.
\ee
Moreover,  
\[
\big|1-\frac{  | (\eta_1,c_1^2\eta_2, c_2^2\eta_3)|}{ \Lambda_2(\eta)} \big|\sim \frac{|\slashed \eta|^2}{|\eta|^2}, \quad \big|\frac{(\eta_1, 0,0)}{ \sqrt{\eta_1^2 +c_1^4 \eta_2^2 + c_2^4\eta_3^2} }-(sgn(\eta_1),0,0)\big|\gtrsim \frac{|\slashed \eta|^2}{|\eta|^2}
\]
From above estimates, $\forall a\in\{1,2\},$ we have
\be\label{july23eqn1}
\big|\nabla_{\eta }  \Phi_{\mu, \nu}^a(\xi, \eta)  \big|\gtrsim \frac{|\slashed \eta|^2}{|\eta|^2}+\big| \mu  \frac{ \xi_1-\eta_1 }{|\xi_1-\eta_1|} -   \frac{  \nu \eta_1 }{|\eta_1|} \big|+\frac{|\slashed \xi-\slashed \eta|^2}{|\xi-\eta|^2} + \big| \frac{(\xi_2-\eta_2, \xi_3-\eta_3)}{|\xi_1-\eta_1|} -\mu\nu \frac{(c_1^2 \eta_2, c_2^2\eta_3)}{|\eta_1|}\big|:=l_{\mu,\nu}(\xi, \eta).
\ee
 
After   doing dyadic decomposition for the size of $ l_{\mu, \nu}(\xi, \eta)$ with threshold $-m/2+\delta m$, we decompose $\mathcal{J}_{k;k_1,k_2}^{a;\mu,\nu;j_1, j_2}(t, \xi)$ further as follows, 
\be\label{july21eqn71}
\begin{split}
\mathcal{J}_{k;k_1,k_2}^{a;\mu,\nu;j_1, j_2}(t, \xi) &= \sum_{l\in [-m/2+\delta m , \tilde{c}+10]\cap\Z}\mathcal{J}_{k;k_1,k_2;l}^{a;\mu,\nu;j_1, j_2}(t, \xi),\\ 
\mathcal{J}_{k;k_1,k_2;l}^{a;\mu,\nu;j_1, j_2}(t, \xi)&:=\int_{\R^3} e^{i t \Phi_{\mu, \nu}^a ( \xi ,\eta)} q^a_{\mu, \nu} (\xi-\eta, \eta) \widehat{h_{1;j_1,k_1}^{\mu}}(t, \xi-\eta) \widehat{h_{2;j_2,k_2}^{\nu}}(t, \eta)  \psi_k(\xi)    \varphi_{l;-m/2+\delta m}( l_{\mu,\nu}(\xi, \eta))  d\eta,
\end{split}
\ee 
where the absolute constant $\tilde{c}$ is defined in \eqref{2023jan28eqn1}. 

$\oplus$\qquad The case $l\in [ -100- 5\max\{\tilde{c},-c \},\tilde{c}+10]\cap \Z.$

Recall \eqref{july21eqn71}. From  \eqref{july23eqn1}, we have
\be\label{july21eqn81}
  \forall a\in\{1,2\}, \quad \big|\nabla_{\eta }  \Phi_{\mu, \nu}^a(\xi, \eta)  \big| \gtrsim 1. 
\ee
From the above estimate, by doing integration by parts in $\eta$ many times, we can rule out the case $\max\{j_1,j_2\}\leq (1-\delta)m$. It would be sufficient to consider the case $\max\{j_1,j_2\}\geq (1-\delta)m$. 

From the decay estimate \eqref{decayestimatefinal},  the $L^2-L^\infty$-type bilinear estimate if $k\geq -m,$   and the volume of support of $\xi$ if $k\leq -m$, we have  
 \be\label{july13eqn31}
\begin{split}
& \sum_{(k_1,k_2)\in \cup_{i=1,2,3} \chi_k^i} \sum_{ \begin{subarray}{c}
j_i\in [-k_{i,-}, \infty)\cap \Z, i=1,2\\
\max\{j_1,j_2\}\geq (1-\delta)m \\
\end{subarray} } 2^{(1+\alpha)j}\|P_k\big[ \varphi_{j,k}(\cdot)\mathcal{F}^{-1}\big[   \mathcal{J}_{k;k_1,k_2;l}^{a;\mu,\nu;j_1, j_2}(t, \xi) \big]\big](x)\|_{L^2}\\
& \lesssim  \sum_{(k_1,k_2)\in \cup_{i=1,2,3} \chi_k^i} \sum_{ \begin{subarray}{c}
j_i\in [-k_{i,-}, \infty)\cap \Z, i=1,2\\
\max\{j_1,j_2\}\geq (1-\delta)m \\
\end{subarray} }2^{(1+\alpha)j} 2^{-(1+\alpha)\max\{j_1,j_2\}}2^{-8(\min\{k_1,k_2\})_{+}+3\delta m }\\ 
&\times  \min\{ 2^{3k/2}, 2^{-m}\} \epsilon_1^2   \lesssim 2^{-m+10\delta m}\epsilon_1^2.
\end{split}
\ee

$\oplus$\qquad The case $l\in [-m/2+\delta m , -100- 5\max\{\tilde{c},-c \} ]\cap \Z.$ 

Recall \eqref{june21eqn51}, \eqref{july13eqn1}, and the definition of  $ l_{\mu,\nu}(\xi, \eta)$  in  \eqref{july23eqn1}. From the estimate \eqref{addsymbolest}, the following estimate holds for   the localized symbol, 
 \be\label{july13eqn11}
 \big|  q^a_{\mu, \nu}(\xi-\eta, \eta)    \varphi_{l;-m/2+\delta m}( l_{\mu,\nu}(\xi, \eta))   \psi_k(\xi)\psi_{k_1}(\xi-\eta)\psi_{k_2}(\eta)   \big| \lesssim 2^{ l}.
 \ee

  We first consider  the threshold case $l=-m/2+\delta m $. 
After using the $L^2-L^\infty$ type bilinear estimate and     the volume of support of    frequency for the input putted in $L^2$, from the estimate \eqref{july13eqn11},  we have
\[
\sum_{(k_1,k_2)\in \cup_{i=1,2,3} \chi_k^i} 2^{(1+\alpha)j}\big\| \sum_{ \begin{subarray}{c}
j_i\in [-k_{i,-}, \infty)\cap \Z, i=1,2\\
\end{subarray} }  P_k\big[ \varphi_{j,k}(\cdot)\mathcal{F}^{-1}\big[   \mathcal{J}_{k;k_1,k_2;l}^{a;\mu,\nu;j_1, j_2}(t, \xi) \big]\big](x)\big\|_{L^2}
\]
\be\label{july13eqn15}
\lesssim \sum_{(k_1,k_2)\in \cup_{i=1,2,3} \chi_k^i}   2^{l + (1+\alpha)j} 2^{\min\{k_1,k_2\}/2+l/2} 2^{-m + \max\{k_1,k_2\}} 2^{-8(k_{1,+}+k_{2,+})}\epsilon_1^2\lesssim 2^{-3m/4+(\alpha+10\delta)m}\epsilon_1^2.
\ee
 
Now, we focus on the estimate the non-threshold case, i.e., $l\in( -m/2+\delta m,  -100-5\max\{\tilde{c}, -c\}]\cap\Z$. From  \eqref{july23eqn1}, we have
\be\label{july21eqn84}
  \forall a\in\{1,2\},\quad \big|\nabla_{\eta }  \Phi_{\mu, \nu}^a(\xi, \eta)  \big| \gtrsim 2^{l}. 
\ee

From the above estimate, by doing integration by parts in ``$\eta$'' many times, we can rule out the case $\max\{j_1,j_2\}\leq (1-\delta)m +  l $. It would be sufficient to consider the case $\max\{j_1,j_2\}\geq (1-\delta)m +  l$.
 From the decay estimate \eqref{decayestimatefinal} and the $L^2-L^\infty$-type bilinear estimate if $k\geq -m,$   and the volume of support of $\xi$ if $k\leq -m$, we have  
\be\label{july13eqn16}
 \begin{split}
&\sum_{(k_1,k_2)\in \cup_{i=1,2,3} \chi_k^i} \sum_{ \begin{subarray}{c}
j_i\in [-k_{i,-}, \infty)\cap \Z, i=1,2\\
\max\{j_1,j_2\}\geq (1-\delta)m+ l \\
\end{subarray} } 2^{(1+\alpha)j}\|P_k\big[ \varphi_{j,k}(\cdot)\mathcal{F}^{-1}\big[   \mathcal{J}_{k;k_1,k_2;l}^{a;\mu,\nu;j_1, j_2}(t, \xi) \big]\big](x)\|_{L^2}\\ 
&\lesssim  \sum_{(k_1,k_2)\in \cup_{i=1,2,3} \chi_k^i} \sum_{ \begin{subarray}{c}
j_i\in [-k_{i,-}, \infty)\cap \Z, i=1,2\\
\max\{j_1,j_2\}\geq (1-\delta)m + l \\
\end{subarray} }2^{(1+\alpha)j+l} 2^{-(1+\alpha)\max\{j_1,j_2\}}2^{-8\min\{k_1,k_2\}+3\delta m } \\
& \times \min\{ 2^{3k/2}, 2^{-m}\} \epsilon_1^2 \lesssim 2^{-m+\alpha m +10\delta m}\epsilon_1^2.
 \end{split}
 \ee
 After summing up with respect to $l$, which only causes a logarithmic loss, our desired estimate \eqref{july13eqn58} holds from the above obtained estimate, (\ref{july13eqn15}), (\ref{july13eqn31}), and \eqref{july13eqn52}.
\end{proof}

\subsection{The $L^\infty_\xi$-estimate of   profiles }
 
\begin{proposition}\label{Zest}
Under the bootstrap assumption \eqref{bootstrap}, for any $t \in [2^{m-1}, 2^m]\subset[0, T], m\in \Z_+$, $\mu, \nu\in\{+, -\}$,  we have
\be\label{july23eqn91}
 \sum_{a=1,2}  \| \langle \xi\rangle^8    \p_t \widehat{h_a}(t, \xi)  \|_{L^\infty_\xi} \lesssim  2^{-m- \delta m}\epsilon_1^2.
\ee
\end{proposition}
\begin{proof}
Recall \eqref{systemeqnprof}.  Same as what we did in \eqref{july13eqn70} and  \eqref{july21eqn71}, we do space localization for two inputs and decompose $\mathcal{J}_{k;k_1,k_2}^{a;\mu,\nu;j_1, j_2}(t, \xi) $  into      $\mathcal{J}_{k;k_1,k_2;l}^{a;\mu,\nu;j_1, j_2}(t, \xi)$ based on the size of $l_{\mu, \nu}(\xi,\eta)$.  Moreover, from the Cauchy-Schwarz inequality, we have
\[
 \sum_{(k_1,k_2)\in \cup_{i=1,2,3} \chi_k^i, \max\{k_1,k_2\}\geq 5m/N_0}  \|  \langle \xi\rangle^8  \mathcal{J}_{k;k_1,k_2}^{a;\mu,\nu}(t, \xi)      \|_{L^\infty_\xi}  
\]
\be\label{july23eqn93}
\lesssim  \sum_{(k_1,k_2)\in \cup_{i=1,2,3} \chi_k^i, \max\{k_1,k_2\}\geq 5m/N_0}  2^{ 8k_{+}} 2^{3\min\{k_1,k_2\}/2 -N_0\max\{k_1,k_2\} }\epsilon_1^2 \lesssim 2^{-2m}\epsilon_1^2.
\ee

Now, it would be sufficient to consider the case $\max\{k_1,k_2\}\leq 5m/N_0. $

$\oplus$\qquad The case $l\in [ -100- 5\max\{\tilde{c},-c \},\tilde{c}+10]\cap \Z.$

Recall \eqref{july21eqn71}. As in the proof of Lemma \ref{fixedtimeest},   by doing integration by parts in $\eta$ many times, we can rule out the case $\max\{j_1,j_2\}\leq (1-\delta)m$. It would be sufficient to consider the case $\max\{j_1,j_2\}\geq (1-\delta)m$. From the Cauchy-Schwarz inequality, we have
\be\label{july23eqn94}
\begin{split}
&\sum_{\begin{subarray}{c}
 (k_1,k_2)\in \cup_{i=1,2,3} \chi_k^i\\ 
 \max\{k_1,k_2\}\leq 5m/N_0
 \end{subarray}  } \sum_{ \begin{subarray}{c}
j_i\in [-k_{i,-}, \infty)\cap \Z, i=1,2\\
\max\{j_1,j_2\}\geq (1-\delta)m \\
\end{subarray} }\|  \langle \xi\rangle^8    \mathcal{J}_{k;k_1,k_2;l}^{a;\mu,\nu;j_1, j_2}(t, \xi)  \|_{L^\infty_\xi}\\
 & \lesssim  \sum_{\begin{subarray}{c}
 (k_1,k_2)\in \cup_{i=1,2,3} \chi_k^i\\ 
 \max\{k_1,k_2\}\leq 5m/N_0
 \end{subarray}  } \sum_{ \begin{subarray}{c}
j_i\in [-k_{i,-}, \infty)\cap \Z, i=1,2\\
\max\{j_1,j_2\}\geq (1-\delta)m \\
\end{subarray} }2^{ 8k_{+}- (1+\alpha)\max\{j_1,j_2\}+3\min\{k_1,k_2\}/2 + \delta m }\epsilon_1^2  \lesssim 2^{-m- 10 \delta m}\epsilon_1^2.
\end{split}
\ee

$\oplus$\qquad The case $l\in [-m/2+\delta m , -100- 5\max\{\tilde{c},-c \} ]\cap \Z.$

We first consider the threshold case $l=-m/2+\delta m .$ For this case, we use the volume of support of frequencies variable. 
Since the definition of $l_{\mu, \nu}(\xi, \eta)$, see  \eqref{july23eqn1},  is a little tricky,   we elaborate how to estimate the volume of support as follows.   From  \eqref{july23eqn1}, for any $(\xi, \eta)\in supp(  \varphi_{l;-m/2+\delta m}( l_{\mu,\nu}(\xi, \eta))  $ $  \psi_k(\xi)\psi_{k_1}(\xi-\eta)\psi_{k_2}(\eta) )$,  we have
\be\label{2023jan30eqn51}
 \xi_2  = (1 +  \frac{ \mu \nu c_1^2 |\xi_1-\eta_1|}{|\eta_1|})\eta_2 + \mathcal{O}(2^{\max\{k_1,k_2\}+l}), \quad 
    \xi_3  = (1 + \frac{ \mu \nu  c_2^2 |\xi_1-\eta_1|}{|\eta_1|})\eta_3 + \mathcal{O}(2^{ \max\{k_1,k_2\}+l}).
\ee

For any fixed $\xi\in \R^3$, if $\mu \nu=+$ or $k\leq \max\{k_1,k_2\}-100-2\max\{\tilde{c}, -c\}$ or $\min\{k_1,k_2\}\leq k-100-2\max\{\tilde{c}, -c\}$,    we have
\be\label{2023jan30eqn71}
\big|\{  \eta :  |l_{\mu, \nu}(\xi,\eta)|\leq 2^l, |\eta|\sim 2^{k_2}, |\xi-\eta|\sim 2^{k_1}, |\xi|\sim 2^k \}\big|\lesssim 2^{\min\{k_1, k_2\} + 2\max\{k_1,k_2\}+2l}. 
\ee

It remains to consider the case $\mu\nu=-$, $|\xi|\sim |\eta|\sim |\xi-\eta|$. For this case, after localizing   the size of $ |\eta_1|-c_1^2 |\xi_1-\eta_1|  $, from \eqref{2023jan30eqn51}, we have 
\[
\begin{split}
&\big|\{  \eta :  |l_{\mu, \nu}(\xi,\eta)|\leq 2^l, |\eta|\sim 2^{k_2}, |\xi-\eta|\sim 2^{k_1}, |\xi|\sim 2^k, ||\eta_1|-c_1^2 |\xi_1-\eta_1|  |\sim 2^b \}\big|\\ 
& \lesssim \min\{ 2^b ( 2^{-b+2k+l})^2, 2^{b}2^{2k+l}\}.
\end{split}
\]
From the above estimate, we have
\be\label{2023jan30eqn91}
\begin{split}
&\big|\{  \eta :  |l_{\mu, \nu}(\xi,\eta)|\leq 2^l, |\eta|\sim 2^{k_2}, |\xi-\eta|\sim 2^{k_1}, |\xi|\sim 2^k \}\big| \lesssim 2^{3k+3l/2}.
\end{split}
\ee
To sum up, in whichever case, the following estimate holds for any fixed $\xi\in \R^3$,
\be\label{2023jan30eqn1}
\big|\{  \eta :  |l_{\mu, \nu}(\xi,\eta)|\leq 2^l, |\eta|\sim 2^{k_2}, |\xi-\eta|\sim 2^{k_1}, |\xi|\sim 2^k \}\big|\lesssim 2^{\min\{k_1, k_2\} + 2\max\{k_1,k_2\}+3l/2}. 
\ee

From the estimate of the symbol in \eqref{july13eqn11}, the estimate of the volume of the support of $\eta$ in \eqref{2023jan30eqn1}, we have 
 \be
 \begin{split}
& \sum_{\begin{subarray}{c}
 (k_1,k_2)\in \cup_{i=1,2,3} \chi_k^i\\ 
 \max\{k_1,k_2\}\leq 5m/N_0
 \end{subarray}  } \|\sum_{ \begin{subarray}{c}
j_i\in [-k_{i,-}, \infty)\cap \Z, i=1,2 
\end{subarray} }  \langle \xi\rangle^8    \mathcal{J}_{k;k_1,k_2;l}^{\mu,\nu;j_1, j_2}(t, \xi)     \|_{L^\infty_\xi} \\
 &\lesssim \sum_{\begin{subarray}{c}
 (k_1,k_2)\in \cup_{i=1,2,3} \chi_k^i\\ 
 \max\{k_1,k_2\}\leq 5m/N_0
 \end{subarray}  }  2^{ 8k_{+} +l }  2^{\min\{k_1, k_2\} + 2\max\{k_1,k_2\}+3l/2} \epsilon_1^2\lesssim  2^{-m- 10 \delta m}\epsilon_1^2.
 \end{split}
 \ee

 Now, we focus on the non-threshold case, i.e,  $l\in (-m/2+\delta m , -100- 5\max\{\tilde{c},-c \} ]\cap \Z.$  As in the proof of Lemma \ref{fixedtimeest},   by doing integration by parts in $\eta$ many times, we can rule out the case $\max\{j_1,j_2\}\leq (1-\delta)m +  l$. It would be sufficient to consider the case $\max\{j_1,j_2\}\geq (1-\delta)m +  l$.

  From the Cauchy-Schwarz inequality, the estimate of kernel in \eqref{july13eqn11}, the estimate of the volume of the support of $\eta$ in \eqref{2023jan30eqn1}, we have  
\be\label{july23eqn99}
\begin{split}
& \sum_{\begin{subarray}{c}
 (k_1,k_2)\in \cup_{i=1,2,3} \chi_k^i\\ 
 \max\{k_1,k_2\}\leq 5m/N_0
 \end{subarray}  } \sum_{ \begin{subarray}{c}
j_i\in [-k_{i,-}, \infty)\cap \Z, i=1,2\\
\max\{j_1,j_2\}\geq (1-\delta)m +l \\
\end{subarray} }\|  \langle \xi\rangle^8    \mathcal{J}_{k;k_1,k_2;l}^{\mu,\nu;j_1, j_2}(t, \xi)    \|_{L^\infty_\xi}\\
 & \lesssim  \sum_{\begin{subarray}{c}
 (k_1,k_2)\in \cup_{i=1,2,3} \chi_k^i\\ 
 \max\{k_1,k_2\}\leq 5m/N_0
 \end{subarray}  } \sum_{ \begin{subarray}{c}
j_i\in [-k_{i,-}, \infty)\cap \Z, i=1,2\\
\max\{j_1,j_2\}\geq (1-\delta)m+l \\
\end{subarray} }2^l 2^{ 8k_{+}- (1+\alpha)\max\{j_1,j_2\}  + \delta m } \\
& \times \big(   2^{\min\{k_1, k_2\} + 2\max\{k_1,k_2\}+3l/2} \big)^{1/2} \epsilon_1^2 \lesssim 2^{-m- 10 \delta m}\epsilon_1^2.\\ 
\end{split}
\ee
Hence finishing the desired estimate \eqref{july23eqn91}.
\end{proof}

\section{ The $Z$-norm estimate of profiles}\label{W}

The rest of this section is devoted to prove the following proposition, which is the main $Z$-norm estimate part. A divided plan is provided in the proof of this Proposition.

\begin{proposition}\label{West}
Under the bootstrap assumption \eqref{bootstrap}, for any $t_1, t_2\in [2^{m-1}, 2^m]\subset[0, T], m\in \Z_+$, we have
\be\label{july24eqn1}
\sum_{a=1,2}  \| h_a(t_2)- h_a(t_1) \|_{Z} \lesssim 2^{\delta m }\epsilon_1^2.
\ee 
\end{proposition}
\begin{proof}
Recall \eqref{systemeqnprof}. The desired estimate \eqref{july24eqn1} holds after combining        \eqref{july3eqn11} in Lemma \ref{highhighWfixed},   \eqref{july3eqn41} in Lemma \ref{highhighWfixed2}, \eqref{july21eqn41} and \eqref{july21eqn43} in Lemma  \ref{lowhighosc}, \eqref{july21eqn45} and \ref{july21eqn46} in Lemma \ref{lowhighess}, and \eqref{july7eqn4} in Lemma \ref{comparableest}.
\end{proof}

\subsection{The High $\times$ High type interaction}

In this subsection, we consider the High $\times$ High type interaction with the output frequency much smaller than the input frequencies, e.g., $k\leq k_1-5 \max\{\tilde{c}, -c\}-100$ and $|\xi|\ll |\eta|$, where   absolute constants $\tilde{c}$ and $c$  are defined in \eqref{2023jan28eqn1}.

In the following Lemma, we first consider the non-resonance case, i.e., $\mu\nu=+.$

\begin{lemma}\label{highhighWfixed}
Under the bootstrap assumption \eqref{bootstrap}, for any $ t \in [2^{m-1}, 2^m]\subset[0, T], m\in \Z_+$, $\mu, \nu\in\{+, -\}$, s.t., $\mu\nu=+,$ we have
\be\label{july3eqn11}
\sup_{k\in \Z}  \sum_{(k_1,k_2)\in \chi_k^1, k_1\geq k+100+ 5 \max\{\tilde{c}, -c\} } \|\int_{t_1}^{t_2} \mathcal{F}^{-1}\big[   \mathcal{J}_{k;k_1,k_2}^{a;\mu,\nu }(t, \xi) \big]  d t  \|_{Z} \lesssim 2^{ -\alpha m/4 }\epsilon_1^2.
\ee
\end{lemma}
\begin{proof}
Recall \eqref{systemeqnprof}. Same as what we did in \eqref{july13eqn70}, we also do space localization  for two inputs. As in the proof of Lemma \ref{fixedtimeest}, we split into two cases   based on the possible size of $x$, i.e., the size of $j$. 

 Since the obtained estimates \eqref{july13eqn72} and \eqref{july13eqn52} are still valid, it would be sufficient to consider the case $j\leq   \max\{m, -k_{-}\}+ \delta (m+j) .$

Note that, for this case, we have $j\leq (1-\delta)^{-1}\big(\max\{m, -k_{-}\}+ \delta m \big) $. Moreover, 
 from the fact that $|\xi|\ll |\eta|$,  we have $sgn(\eta_1-\xi_1)=sgn(\eta_1) $ if $|\slashed \xi-\slashed \eta|\ll |\xi-\eta|$. As a result, from the estimate \eqref{july23eqn1}, we have 
\be\label{june28eqn31}
\big|\nabla_{\eta } \big(  \Phi_{\mu, \nu}^a(\xi, \eta) \big)\big| \gtrsim  \frac{|\slashed \eta|^2}{|\eta|^2} + | \frac{( \eta_1-\xi_1,  \eta_2-\xi_2, \eta_3-\xi_3)}{|\xi-\eta|} +\frac{(\eta_1, c_1^2\eta_2, c_2^2\eta_3)}{ \sqrt{\eta_1^2 + c_1^2 \eta_2^2 +  c_2^2 \eta_3^2} } |  \gtrsim 1. 
\ee

Based on the possible size of $\max\{j_1, j_2\}$, we split into two sub-cases as follows.

$\oplus$\qquad If $\max\{j_1, j_2\}\leq (1-\delta) m $.

Note that, from the estimate \eqref{june28eqn31}, by doing integration by parts in $\eta$ once, we gain $2^{-m}$ and lose at most $\max\{2^{-k_1}, 2^{\max\{j_1,j_2\}}\}=2^{\max\{j_1,j_2\}}.$ Therefore,  by doing integration by parts in $\eta$ once, we gain at least $2^{-\delta m }$. After doing integration by parts in ``$\eta$'' $\delta^{-2}$ times, the following estimate holds from using either  the $L^2-L^\infty$ type bilinear estimate and the $L^\infty\rightarrow L^2$-type  Sobolev embedding or using  the $L^2\rightarrow L^1$ type Sobolev embedding first and then using the $L^2-L^2$ type bilinear estimate, 
\[
\sum_{(k_1,k_2)\in \chi_k^1} \sum_{ \begin{subarray}{c}
j_i\in [-k_{i,-}, \infty)\cap \Z, i=1,2\\
\max\{j_1,j_2\}\leq (1-\delta)m \\
\end{subarray} } 2^{(1+\alpha)j}\|P_k\big[ \varphi_{j,k}(\cdot)\mathcal{F}^{-1}\big[   \mathcal{J}_{k;k_1,k_2}^{a;\mu,\nu;j_1, j_2}(t, \xi) \big]\big](x)\|_{L^2}
\]
\[
\lesssim \sum_{(k_1,k_2)\in \chi_k^1}  \sum_{ \begin{subarray}{c}
j_i\in [-k_{i,-}, \infty)\cap \Z, i=1,2\\
\max\{j_1,j_2\}\leq (1-\delta)m \\
\end{subarray} }   2^{-100m } 2^{(1+\alpha)j -(1+\alpha)j_2+\delta m } 2^{3\min\{k,k_2\}/2}   
\]
\be\label{july3eqn33}
\times \min\{2^{-N_0 k_{2,+}+\delta m }, 2^{-(1+\alpha)j_1+\delta m } \} \epsilon_1^2 \lesssim  2^{-2m}\epsilon_1^2. 
\ee

$\oplus$\qquad If $\max\{j_1, j_2\}\geq (1-\delta) m $ and $k_1\leq -\alpha m .$

Due to symmetry of inputs in the High $\times$ High type interaction case, without loss of generality, we assume that $j_1=\max\{j_1,j_2\}$. After putting $h_{1;j_1,k_1}$ in $L^2$ and putting $e^{-i t\Lambda_2  } h_{2;j_2,k_2}$ in $L^\infty$,  from \eqref{decayestimatefinal}, we have 
\be\label{july3eqn21}
\begin{split}
&\sum_{\begin{subarray}{c}
(k_1,k_2)\in \chi_k^1\\ 
k_1\leq -\alpha m\\
\end{subarray}} \sum_{ \begin{subarray}{c}
j_i\in [-k_{i,-}, \infty)\cap \Z, i=1,2\\
\max\{j_1,j_2\}\geq (1-\delta)m \\
\end{subarray} } 2^{(1+\alpha)j}\|P_k\big[ \varphi_{j,k}(\cdot)\mathcal{F}^{-1}\big[   \mathcal{J}_{k;k_1,k_2}^{a;\mu,\nu;j_1, j_2}(t, \xi) \big]\big](x)\|_{L^2}\\ 
&\lesssim \sum_{\begin{subarray}{c}
(k_1,k_2)\in \chi_k^1\\ 
k_1\leq -\alpha m\\
\end{subarray}}  \sum_{ \begin{subarray}{c}
j_i\in [-k_{i,-}, \infty)\cap \Z, i=1,2\\
\max\{j_1,j_2\}\geq (1-\delta)m \\
\end{subarray} }  2^{(1+\alpha)j-8k_{1,+} -(1+\alpha)\max\{j_1,j_2\}+2\delta m }2^{ -m+k_1} \lesssim 2^{-(1+\alpha/4)m}\epsilon_1^2\\
\end{split}
\ee
 
 $\oplus$\qquad If $\max\{j_1, j_2\}\geq (1-\delta) m $ and $k_1\geq -\alpha m .$

Note that, for this case, we have $j\leq (1+10\delta)m.$   Due to symmetry between two inputs in the High $\times$ High type interaction, without loss of generality, we assume that $j_1=\max\{j_1,j_2\}.$ As $\mu\nu=+, |\xi|\ll |\eta|$, we have 
\be\label{july3eqn4}
| \Phi_{\mu, \nu}^a(\xi, \eta)|\sim |\eta| \gtrsim 2^{k_1}\gtrsim 2^{-\alpha m}. 
\ee

Recall \eqref{july13eqn70}. To take advantage of high oscillation in time, we do integration by parts in time once. As a result, we have
\be\label{july3eqn1}
\begin{split}
&  \int_{t_1}^{t_2}\mathcal{F}^{-1}\big[   \mathcal{J}_{k;k_1,k_2}^{a;\mu,\nu;j_1, j_2}(t, \xi) \big]  d t   =     \sum_{b=1,2}  End_{k;k_1,k_2;b }^{a;\mu,\nu;j_1, j_2}(t_a,  \xi ) +   \int_{t_1}^{t_2}\mathcal{F}^{-1}\big[   \widetilde{\mathcal{J}}_{k;k_1,k_2;b}^{a;\mu,\nu;j_1, j_2}(t, \xi) \big]  d t  ,\\
& End_{k;k_1,k_2 ;b  }^{a;\mu,\nu;j_1, j_2}(t_b, \xi) :=  \int_{\R^3}  \int_{\R^3} e^{ i x\cdot \xi +  i t_b \Phi_{\mu, \nu}^a ( \xi ,\eta)} \frac{(-1)^{b-1}i q^a_{\mu, \nu} (\xi-\eta, \eta)}{  \Phi_{\mu, \nu}^a ( \xi ,\eta)}\widehat{h_{1;j_1,k_1}^{\mu}}(t_b, \xi-\eta) \widehat{h_{2;j_2,k_2}^{\nu}}(t_b, \eta) \psi_k(\xi)   d\eta d \xi  ,\\
&  \widetilde{\mathcal{J}}_{k;k_1,k_2;1 }^{a;\mu,\nu;j_1, j_2}(t, \xi):= \int_{\R^3}  \int_{\R^3} e^{ i x\cdot \xi +  i t  \Phi_{\mu, \nu}^a ( \xi ,\eta)} \frac{i q^a_{\mu, \nu} (\xi-\eta, \eta)}{\Phi_{\mu, \nu}^a ( \xi ,\eta)} \widehat{h_{1;j_1,k_1}^{\mu}}(t , \xi-\eta) \p_t \widehat{h_{2;j_2,k_2}^{\nu}}(t , \eta) \psi_k(\xi)   d\eta d \xi,\\
&  \widetilde{\mathcal{J}}_{k;k_1,k_2;2}^{a;\mu,\nu;j_1, j_2}(t, \xi):= \int_{\R^3}  \int_{\R^3} e^{ i x\cdot \xi +  i t  \Phi_{\mu, \nu}^a ( \xi ,\eta)} \frac{ i q^a_{\mu, \nu} (\xi-\eta, \eta)}{\Phi_{\mu, \nu}^a ( \xi ,\eta)}\p_t \widehat{h_{1;j_1,k_1}^{\mu}}(t , \xi-\eta)  \widehat{h_{2;j_2,k_2}^{\nu}}(t , \eta) \psi_k(\xi)   d\eta d \xi.\\ 
\end{split}
\ee
From the estimate of phases in \eqref{july3eqn4}, $\forall k_1 \in[-\alpha m , \infty)\cap \Z, (k_1,k_2)\in \chi_k^1,$ the following estimate holds for the kernel, 
\be\label{july3eqn3}
\big\|\int_{\R^3} \int_{\R^3} e^{i x\cdot\xi+ i y\cdot\eta} \frac{i q^a_{\mu, \nu} (\xi-\eta, \eta)}{\Phi_{\mu, \nu}^a ( \xi ,\eta)} \psi_k(\xi)  \psi_{k_1}(\xi-\eta)  \psi_{k_2}( \eta) d\eta d \xi \big\|_{L^1_{x,y}} \lesssim 2^{8\alpha m }.
\ee
From the $L^2-L^\infty$ type bilinear estimate and the above estimate of kernel, we have
\be\label{july3eqn2}
\begin{split}
 &\sum_{\begin{subarray}{c}
(k_1,k_2)\in \chi_k^1\\ 
k_1\geq -\alpha m\\
\end{subarray}} \sum_{ \begin{subarray}{c}
j_i\in [-k_{i,-}, \infty)\cap \Z, i=1,2\\
\max\{j_1,j_2\}\geq (1-\delta)m \\
\end{subarray} } 2^{(1+\alpha)j}  \|P_k\big[ \varphi_{j,k}(\cdot)\mathcal{F}^{-1}\big[  End_{k;k_1,k_2;b }^{a;\mu,\nu;j_1, j_2}(t_a, \xi) \big]\big](x)\|_{L^2}\\ 
 &\lesssim  \sum_{\begin{subarray}{c}
(k_1,k_2)\in \chi_k^1\\ 
k_1\geq -\alpha m\\
\end{subarray}}  \sum_{ \begin{subarray}{c}
j_i\in [-k_{i,-}, \infty)\cap \Z, i=1,2\\
\max\{j_1,j_2\}\geq (1-\delta)m \\
\end{subarray} } 2^{(1+10\alpha) m } \|h_{1;j_1,k_1}(t_a)\|_{L^2} \|e^{-i t_a\Lambda_2  }h_{2;j_2,k_2}(t_a)\|_{L^\infty} \\ 
& \lesssim  2^{-m/2}\epsilon_1^2.
 \end{split}
\ee

Similarly, from the $L^2-L^\infty$ type bilinear estimate,  the   estimate of kernel in \eqref{july3eqn3} and  the estimate \eqref{july13eqn58} in Lemma \ref{fixedtimeest},  we have  
\be\label{july3eqn7}
\begin{split}
 & \sum_{\begin{subarray}{c}
(k_1,k_2)\in \chi_k^1\\ 
k_1\geq -\alpha m\\
\end{subarray}}  \sum_{ \begin{subarray}{c}
j_i\in [-k_{i,-}, \infty)\cap \Z, i=1,2\\
\max\{j_1,j_2\}\geq (1-\delta)m \\
\end{subarray} } \sum_{b=1,2} 2^{(1+\alpha)j}  \|  \int_{t_1}^{t_2} P_k\big[ \varphi_{j,k}(\cdot)\mathcal{F}^{-1}\big[   \widetilde{\mathcal{J}}_{k;k_1,k_2;b }^{a;\mu,\nu;j_1, j_2}(t, \xi)  \big]\big](x) d t\|_{L^2}\\ 
 &\lesssim  \sum_{\begin{subarray}{c}
(k_1,k_2)\in \chi_k^1\\ 
k_1\geq -\alpha m\\
\end{subarray}} \sum_{ \begin{subarray}{c}
j_i\in [-k_{i,-}, \infty)\cap \Z, i=1,2\\
\max\{j_1,j_2\}\geq (1-\delta)m \\
\end{subarray} } \sup_{t\in [2^{m-1},2^m]} 2^{(2+10\alpha) m } \big(\|h_{1;j_1,k_1}(t )\|_{L^2} \|e^{-i t \Lambda_2  }\p_t h_{2;j_2,k_2}(t )\|_{L^\infty} \\ 
&  + \|\p_t h_{1;j_1,k_1}(t )\|_{L^2} \|e^{-i t \Lambda_2  } h_{2;j_2,k_2}(t )\|_{L^\infty} \big) \lesssim  2^{-m/2}\epsilon_1^2.
 \end{split}
\ee

Hence the desired estimate \eqref{july3eqn11} holds after combining the above estimate and the obtained estimates \eqref{july3eqn2},  \eqref{july3eqn33}, and  \eqref{july3eqn21}.
 
\end{proof}

Now, we consider the resonance  case $\mu\nu=- $. 
\begin{lemma}\label{highhighWfixed2}
Under the bootstrap assumption \eqref{bootstrap}, for any   $t \in [2^{m-1}, 2^m]\subset[0, T], m\in \Z_+$, $\mu, \nu\in\{+, -\}$, s.t., $\mu\nu=-,$ we have
\be\label{july3eqn41}
\sup_{k\in \Z}\sum_{ a=1,2}  \sum_{\begin{subarray}{c}
(k_1,k_2)\in \chi_k^1\\ 
 k_1\geq k+100 +5 \max\{\tilde{c}, -c\}
\end{subarray}} \|\mathcal{F}^{-1}\big[ \int_{t_1}^{t_2}   \mathcal{J}_{k;k_1,k_2}^{a;\mu,\nu}(t, \xi)  d t \big]   \|_{Z} \lesssim 2^{\delta m }\epsilon_1^2,
\ee
where the absolute constants $c$ and $\tilde{c}$ are defined in \eqref{2023jan28eqn1}. 
\end{lemma}
\begin{proof}
 Recall \eqref{systemeqnprof}. Same as what we did in \eqref{july13eqn70}, we also do space localization  for two inputs. As in the proof of Lemma \ref{fixedtimeest},  we split into two cases   based on the possible size of $x$, i.e., the size of $j$. Moreover, because the sign of $\mu  \nu$ doesn't play much role for the case $j\geq  \max\{m, -k_{-}\}+ \delta (m+j)  $ in the proof of Lemma \ref{highhighWfixed}, it would be sufficient for us to only consider the case $j\leq   \max\{m, -k_{-}\}+ \delta (m+j)  $.

 As in the decomposition in  \eqref{july21eqn71},  based on the size of $l_{\mu, \nu}(\xi, \eta),$ we do dyadic decomposition for   $\mathcal{J}_{k;k_1,k_2}^{a;\mu,\nu;j_1, j_2}(t, \xi) $.  We first consider the main  case $l$ is relatively small.  The case  $l\in [-100 -  5\max\{\tilde{c}, -c\}, \tilde{c}+10]\cap \Z$ (see \eqref{2023jan28eqn1} for the definitions of $c$ and $\tilde{c}$) can be handled in the same way as one sub-case of the main case. We will elaborate this point later. 

 Let $l\in [-m/2+\delta m, -100  - 5\max\{\tilde{c}, -c\}]\cap \Z$ be fixed. Recall the detailed formula of $ q_{\mu, \nu} (\xi-\eta, \eta)$ in \eqref{june21eqn51}.  From the estimate \eqref{addsymbolest}, the following estimate for  the localized symbol,
\be\label{L1kernelloc}
\big|  q_{\mu, \nu} (\xi-\eta, \eta)   \varphi_{l;-m/2+\delta m}( l_{\mu,\nu}(\xi, \eta)) \psi_k(\xi)\psi_{k_1}(\xi-\eta)\psi_{k_2}(\eta) \big|  \lesssim 2^{ l}.
\ee

 $\bullet$ \qquad If $k \leq -2m/3 -\alpha m $ or $k_2  \leq -m/4  -2\alpha m. $  

We first consider the threshold case i.e.,  $l=-m/2+\delta m$. Recall the definition of $ l_{\mu,\nu}(\xi, \eta)$ in \eqref{july23eqn1}. From the estimate of the symbol in \eqref{L1kernelloc}, the estimate of the volume of support of $\eta$ in  \eqref{2023jan30eqn71},    the $L^2-L^\infty$-type estimate, the decay estimate \eqref{decayestimatefinal},  we have 
\be\label{july23eqn18}
\begin{split}
\sum_{ \begin{subarray}{c} 
(k_1,k_2)\in \chi_k^1\\
  k_1\geq k+100 +5 \max\{\tilde{c}, -c\}
\end{subarray}} &2^{(1+\alpha)j} \big\| \sum_{ \begin{subarray}{c}
j_i\in [-k_{i,-}, \infty)\cap \Z, i=1,2\\
\end{subarray} }  P_k\big[ \varphi_{j,k}(\cdot)\mathcal{F}^{-1}\big[   \mathcal{J}_{k;k_1,k_2;l}^{a;\mu,\nu;j_1, j_2}(t, \xi) \big]\big](x)\big\|_{L^2}\\
& \lesssim   \sum_{ \begin{subarray}{c} 
(k_1,k_2)\in \chi_k^1\\
  k_1\geq k+100 +5 \max\{\tilde{c}, -c\}
\end{subarray}}    2^{l + (1+\alpha)j} \min\{ 2^{-m+  k_1}2^{l/2},2^{3k/2 + 2l+3k_1} \} \epsilon_1^2\lesssim 2^{-  m}\epsilon_1^2.
\end{split}
\ee

Now we focus on the non-threshold case, $l\in (-m/2+\delta m, -100 - 5\max\{\tilde{c}, -c\}]\cap \Z.$ From the estimate \eqref{july23eqn1},  by doing integration by parts in $\eta$ once, we gain at least $2^{-\delta m }$. After doing integration by parts in $\eta$ $\delta^{-2}$ times, we can rule out the case   $\max\{j_1,j_2\}\leq  m+ l-\delta m $ as follows, 
\be\label{july20eqn16}
\begin{split}
& \sum_{ \begin{subarray}{c} 
(k_1,k_2)\in \chi_k^1\\
  k_1\geq k+100 +5 \max\{\tilde{c}, -c\}
\end{subarray}}   \sum_{ \begin{subarray}{c}
j_i\in [-k_{i,-}, \infty)\cap \Z, i=1,2\\
\max\{j_1,j_2\}\leq   m+ l-\delta m \\
\end{subarray} } 2^{(1+\alpha)j}\|P_k\big[ \varphi_{j,k}(\cdot)\mathcal{F}^{-1}\big[   \mathcal{J}_{k;k_1,k_2;l}^{a;\mu,\nu;j_1, j_2}(t, \xi)  \big]\big](x)\|_{L^2}\\
& \lesssim  \sum_{ \begin{subarray}{c} 
(k_1,k_2)\in \chi_k^1\\
  k_1\geq k+100 +5 \max\{\tilde{c}, -c\}
\end{subarray}}    \sum_{ \begin{subarray}{c}
j_i\in [-k_{i,-}, \infty)\cap \Z, i=1,2\\
\max\{j_1,j_2\}\leq  m+2l-\delta m \\
\end{subarray} }   2^{-100m } 2^{(1+\alpha)j -(1+\alpha)j_2+\delta m } 2^{3\min\{k,k_2\}/2} \\
& \quad \times \min\{2^{-N_0 k_{2,+}+\delta m }, 2^{-(1+\alpha)j_1+\delta m } \} \epsilon_1^2 \lesssim  2^{-2m}\epsilon_1^2.
\end{split} 
\ee

It remains to consider the case $\max\{j_1,j_2\}\geq  m+ l-\delta m. $ Due to symmetry between two inputs in the High $\times$ High type interaction, without loss of generality, we assume that $j_1=\max\{j_1,j_2\}.$ From the   estimate \eqref{L1kernelloc}, the $L^2-L^\infty$ type bilinear estimate and the decay estimate  \eqref{decayestimatefinal}, the $L^2\rightarrow L^1$ type Sobolev embedding and the $L^2-L^2$ type bilinear estimate,  we have 
\[
\sum_{ \begin{subarray}{c} 
(k_1,k_2)\in \chi_k^1\\
  k_1\geq k+100 +5 \max\{\tilde{c}, -c\}
\end{subarray}}  \sum_{ \begin{subarray}{c}
j_i\in [-k_{i,-}, \infty)\cap \Z, i=1,2\\
j_1=\max\{j_1,j_2\}\geq   m+ l-\delta m \\
\end{subarray} } 2^{(1+\alpha)j}\|P_k\big[ \varphi_{j,k}(\cdot)\mathcal{F}^{-1}\big[   \mathcal{J}_{k;k_1,k_2;l}^{a;\mu,\nu;j_1, j_2}(t, \xi) \big]\big](x)\|_{L^2}
\]
\be\label{july20eqn18}
\lesssim   \sum_{ \begin{subarray}{c} 
(k_1,k_2)\in \chi_k^1\\
  k_1\geq k+100 +5 \max\{\tilde{c}, -c\}
\end{subarray}}    \sum_{ \begin{subarray}{c}
j_i\in [-k_{i,-}, \infty)\cap \Z, i=1,2\\
j_1=\max\{j_1,j_2\}\geq   m+ l-\delta m \\
\end{subarray} }    2^{  l} 2^{(1+\alpha)j -(1+\alpha)j_1+\delta m }\min\{2^{3k/2}, 2^{-m+ k_2-5k_{2,+}}\}\epsilon_1^2 \lesssim  2^{- m}\epsilon_1^2. 
\ee
 
 $\bullet$ \qquad If $k \geq -2m/3 -\alpha m $ and $k_2  \geq  -m/4  -2\alpha m . $  
 
 We decompose $  \mathcal{J}_{k;k_1,k_2;l}^{a;\mu,\nu;j_1, j_2}(t, \xi)$ further by using   the following partition of unity function,
\be\label{july20eqn4}
\begin{split}
 &\mathcal{J}_{k;k_1,k_2;l}^{a;\mu,\nu;j_1, j_2}(t, \xi)=\sum_{i=1,2,3}  \mathcal{J}_{k;k_1,k_2;l;i}^{a;\mu,\nu;j_1, j_2}(t, \xi), \\
&\varphi_{1}(\xi, \eta)=\psi_{> -3}(1+\mu\xi_1\eta_1/|\xi||\eta|)    \psi_{[-3+c,3+\tilde{c}] }(|\slashed \xi|/|\slashed \eta| )  , \\ 
&\varphi_{2 }(\xi, \eta)= \psi_{\leq -3}(1+\mu\xi_1\eta_1/|\xi||\eta|)  \psi_{[-3+c,3+\tilde{c}] }(|\slashed \xi|/|\slashed \eta| ),  \quad 
\varphi_{3}(\xi, \eta)=   1-  \psi_{[-3+c,3+\tilde{c}] }(|\slashed \xi|/|\slashed \eta| )  ,\\ 
 &\mathcal{J}_{k;k_1,k_2;l;i}^{a;\mu,\nu;j_1, j_2}(t, \xi)=\int_{\R^3} e^{i t \Phi_{\mu, \nu}^a ( \xi ,\eta)} q^a_{\mu, \nu} (\xi-\eta, \eta) \widehat{h_{1;j_1,k_1}^{\mu}}(t, \xi-\eta) \widehat{h_{2;j_2,k_2}^{\nu}}(t, \eta)  \\ 
 &\quad \times  \psi_k(\xi)    \varphi_{l;-m/2+\delta m}( l_{\mu,\nu}(\xi, \eta))    \varphi_{i}(\xi, \eta)  d\eta
\end{split}
\ee
where $c, \tilde{c}$ are defined in \eqref{2023jan28eqn1}.

 $\oplus$\qquad The estimate of $\mathcal{J}_{k;k_1,k_2;l;1}^{a;\mu,\nu;j_1, j_2}(t, \xi).$

 Recall the definition of the cutoff function $\varphi_{1}(\xi, \eta)$ in  \eqref{july20eqn4}.    As $l\leq -100  - 2\max\{\tilde{c}, -c\}$, $k\leq k_1-100 - 5\max\{\tilde{c}, -c\}$, and $1+\mu\xi_1\eta_1/|\xi||\eta|\geq 2^{-4}$, For any $(\xi,\eta)\in supp\big(\varphi_{1}(\xi, \eta) \psi_k(\xi)\psi_{k_1}(\xi-\eta)\psi_{k_2}(\eta)  \varphi_{l;-m/2+\delta m}( l_{\mu,\nu}(\xi, \eta))\big), $  we have $ |\Phi_{\mu, \nu}^a ( \xi ,\eta)|\sim 2^k.$ Recall that $k \geq -2m/3- \alpha m $. To take advantage of high oscillation in time, we do normal form transformation. As a result, we have
\be\label{july21eqn1}
\begin{split}
 \int_{t_1}^{t_2}& \mathcal{J}_{k;k_1,k_2;l;1}^{a;\mu,\nu;j_1, j_2}(t, \xi) d t =\sum_{n=1,2} (-1)^n  \int_{\R^3} e^{i t_n \Phi_{\mu, \nu}^a ( \xi ,\eta)} \frac{q^a_{\mu, \nu} (\xi-\eta, \eta)}{i\Phi_{\mu, \nu}^a ( \xi ,\eta) } \widehat{h_{1;j_1,k_1}^{\mu}}(t_n , \xi-\eta) \widehat{h_{2;j_2,k_2}^{\nu}}(t_n , \eta)  \\ 
 & \quad \times \psi_k(\xi)    \varphi_{l;-m/2+\delta m}( l_{\mu,\nu}(\xi, \eta))   \varphi_{1}(\xi, \eta)  d\eta -  \int_{t_1}^{t_2} \int_{\R^3} e^{i t \Phi_{\mu, \nu}^i ( \xi ,\eta)} \frac{q^a_{\mu, \nu} (\xi-\eta, \eta)}{i\Phi_{\mu, \nu}^a ( \xi ,\eta) } \\ 
 & \quad \times \p_t \big[ \widehat{h_{1;j_1,k_1}^{\mu}}(t, \xi-\eta) \widehat{h_{2;j_2,k_2}^{\nu}}(t, \eta)\big]  \psi_k(\xi)     \varphi_{l;-m/2+\delta m}( l_{\mu,\nu}(\xi, \eta))   \varphi_{1}(\xi, \eta)  d\eta dt.
\end{split}
\ee

We first consider the threshold case i.e.,  $l=-m/2+\delta m$. From \eqref{july21eqn1}, the $L^2-L^\infty$-type estimate,  the estimate of symbol in \eqref{L1kernelloc},  the estimate of the volume of support of $\eta$ in  \eqref{2023jan30eqn71},    and the estimate \eqref{july23eqn91}  in Proposition \ref{Zest}, 
\be\label{july23eqn23}
\begin{split}
 \sum_{ \begin{subarray}{c} 
(k_1,k_2)\in \chi_k^1\\
  k_1\geq k+100 +5 \max\{\tilde{c}, -c\}
\end{subarray}}   2^{(1+\alpha)j} &\big\| \sum_{ \begin{subarray}{c}
j_i\in [-k_{i,-}, \infty)\cap \Z, i=1,2\\
\end{subarray} }  P_k\big[ \varphi_{j,k}(\cdot)\mathcal{F}^{-1}\big[    \int_{t_1}^{t_2} \mathcal{J}_{k;k_1,k_2;l;1}^{a;\mu,\nu;j_1, j_2}(t, \xi) d t  \big]\big](x)\big\|_{L^2}\\
& \lesssim   \sum_{ \begin{subarray}{c} 
(k_1,k_2)\in \chi_k^1\\
  k_1\geq k+100 +5 \max\{\tilde{c}, -c\}
\end{subarray}}     2^{-k+ l + (1+\alpha)j+2\alpha m }  2^{3k_1+2l-8k_{1,+} } \epsilon_1^2\lesssim \epsilon_1^2.
\end{split}
\ee

For the non-threshold case, $l\in (-m/2+\delta m, -100+c]\cap \Z,$ as in the obtained estimate \eqref{july20eqn16}, we can rule out the case $\max\{j_1,j_2\}\leq m + l -\delta m $ by doing integration by parts in ``$\eta$'' $\delta^{-2}$ times. Now, we focus on the case $\max\{j_1,j_2\}\geq m + l -\delta m $.   

Recall \eqref{systemeqnprof}. As $ |\Phi_{\mu, \nu}^a ( \xi ,\eta)|\sim 2^k$, from  the $L^2-L^\infty$ type bilinear estimate, the $L^\infty$ decay estimate \eqref{decayestimatefinal}, the estimate \eqref{linearwavedecay} in Lemma \ref{lineardecay}, and  the estimate \eqref{july13eqn58} in Lemma \ref{fixedtimeest}, we have
\be\label{july21eqn3}
\begin{split}
&  \sum_{ \begin{subarray}{c} 
(k_1,k_2)\in \chi_k^1\\
  k_1\geq k+100 +5 \max\{\tilde{c}, -c\}
\end{subarray}}   \sum_{ \begin{subarray}{c}
j_i\in [-k_{i,-}, \infty)\cap \Z, i=1,2\\
 \max\{j_1,j_2\}\geq   m+ l-\delta m \\
\end{subarray} } 2^{(1+\alpha)j}\|P_k\big[ \varphi_{j,k}(\cdot)\mathcal{F}^{-1}\big[    \int_{t_1}^{t_2} \mathcal{J}_{k;k_1,k_2;l;1}^{a;\mu,\nu;j_1, j_2}(t, \xi) d t \big]\big](x)\|_{L^2}\\
&  \lesssim   \sum_{ \begin{subarray}{c} 
(k_1,k_2)\in \chi_k^1\\
  k_1\geq k+100 +5 \max\{\tilde{c}, -c\}
\end{subarray}}    \sum_{ \begin{subarray}{c}
j_i\in [-k_{i,-}, \infty)\cap \Z, i=1,2\\
 \max\{j_1,j_2\}\geq   m+ l-\delta m \\
\end{subarray} }   2^{ m/4+ (\alpha+30\delta)m} 2^{-k+  l} 2^{(1+\alpha)j -(1+\alpha)  \max\{j_1,j_2\}  } \\ 
& \times  2^{-m+  k_2-4k_{2,+}} \epsilon_1^2 \lesssim  \epsilon_1^2. \\ 
\end{split}
\ee

  $\oplus$\qquad The estimate of $\mathcal{J}_{k;k_1,k_2;l;2}^{a;\mu,\nu;j_1, j_2}(t, \xi).$

 Recall the definition of the cutoff function $\varphi_{2}(\xi, \eta)$ in  \eqref{july20eqn4}.  As $l\leq -100  - 2\max\{\tilde{c}, -c\}$, $k\leq k_1-100 - 5\max\{\tilde{c}, -c\}$, for any $(\xi,\eta)\in supp\big(\varphi_{2}(\xi, \eta) \psi_k(\xi)\psi_{k_1}(\xi-\eta)\psi_{k_2}(\eta)  \varphi_{l;-m/2+\delta m}( l_{\mu,\nu}(\xi, \eta)) \big),   $ we have
\be\label{july21eqn11}
\begin{split}
 \forall a\in \{1,2\}, \quad |\Phi_{\mu, \nu}^a ( \xi ,\eta)| & = \big|\Lambda_a(\xi) - \mu \Lambda_2(\xi-\eta) -\nu \Lambda_1(\eta)\big| \\
&= \Big| \frac{2 \Lambda_a(\xi) |\eta|+2\mu\xi\cdot\eta+    (c_1^2-1)\mu(\xi_2-\eta_2)^2 + (c_2^2-1) \mu(\xi_3-\eta_3)^2  }{\mu\Lambda_1(\xi) +  \Lambda_2(\xi-\eta) +  \Lambda_1(\eta) } \Big|\\ 
&\gtrsim 2^{-k_2}\big( \frac{|\slashed \xi|^2 |\eta_1| }{|\xi_1|}  -2^{3+2\tilde{c}}\big( |\slashed \eta|^2 +  |\slashed \xi|^2\big)  \big) \gtrsim |\slashed \eta|^2 |\eta|^{-1} .\\ 
\end{split}
\ee
 
Motivated from the above estimate of phases, we localize further the size of $|\slashed \eta|/|\eta|$.  Recall    \eqref{july23eqn1}, for any $(\xi,\eta)\in supp\big(\varphi_{2}(\xi, \eta) \psi_k(\xi)\psi_{k_1}(\xi-\eta)\psi_{k_2}(\eta)\psi_{b }(|\slashed \eta|/|\eta|)   \varphi_{l;-m/2+\delta m}( l_{\mu,\nu}(\xi, \eta))\big), b\in \Z$,  we have $2b\leq l+10$.  

Moreover, recall  the detailed formula of $ q^a_{\mu, \nu} (\xi-\eta, \eta)$ in \eqref{june21eqn51}. From the estimate \eqref{addsymbolest}, the following estimate for the localized symbol,
\be\label{L1kernellocs}
 |  q^a_{\mu, \nu} (\xi-\eta, \eta) \varphi_{2}(\xi, \eta)  \psi_{b }(|\slashed \eta|/|\eta|)   \psi_k(\xi)\psi_{k_1}(\xi-\eta)\psi_{k_2}(\eta) |  \lesssim 2^{2b}.
\ee

Recall that $k_2  \geq -m/4-2\alpha m. $   Same as the previous case, from the above estimate of the phase,    to take advantage of high oscillation in time, we do normal form transformation. As a result, we have
\be\label{july21eqn5}
\begin{split}
 \int_{t_1}^{t_2} & \mathcal{J}_{k;k_1,k_2;l;2}^{a;\mu,\nu;j_1, j_2}(t, \xi) d t =\sum_{n=1,2} (-1)^n  \int_{\R^3} e^{i t_n \Phi_{\mu, \nu}^a ( \xi ,\eta)} \frac{q^a_{\mu, \nu} (\xi-\eta, \eta)}{i\Phi_{\mu, \nu}^a ( \xi ,\eta) } \widehat{h_{1;j_1,k_1}^{\mu}}(t_n , \xi-\eta) \widehat{h_{2;j_2,k_2}^{\nu}}(t_n , \eta)  \\ 
 & \quad \times \psi_k(\xi)   \varphi_{l;-m/2+\delta m}( l_{\mu,\nu}(\xi, \eta))   \varphi_{2}(\xi, \eta)  d\eta -  \int_{t_1}^{t_2} \int_{\R^3} e^{i t \Phi_{\mu, \nu}^a  ( \xi ,\eta)} \frac{q^a_{\mu, \nu} (\xi-\eta, \eta)}{i\Phi_{\mu, \nu}^a ( \xi ,\eta) } \\ 
 & \quad \times \p_t \big[ \widehat{h_{1;j_1,k_1}^{\mu}}(t, \xi-\eta) \widehat{h_{2;j_2,k_2}^{\nu}}(t, \eta)\big]  \psi_k(\xi)     \varphi_{l;-m/2+\delta m}( l_{\mu,\nu}(\xi, \eta))   \varphi_{2}(\xi, \eta)  d\eta dt.
\end{split}
\ee

We first consider the threshold case i.e.,  $l=-m/2+\delta m$.  From the estimate of phase \eqref{july21eqn11},  the estimate of the kernel in \eqref{L1kernellocs}, the estimate of the volume of support of $\eta$ in  \eqref{2023jan30eqn71},    the decay estimate \eqref{decayestimatefinal},  and the estimate \eqref{july23eqn91}  in Proposition \ref{Zest}, we have 
\be\label{july23eqn25}
\begin{split}
 \sum_{ \begin{subarray}{c} 
(k_1,k_2)\in \chi_k^1\\
  k_1\geq k+100 +5 \max\{\tilde{c}, -c\}
\end{subarray}}   2^{(1+\alpha)j} &\big\| \sum_{ \begin{subarray}{c}
j_i\in [-k_{i,-}, \infty)\cap \Z, i=1,2\\
\end{subarray} }  P_k\big[ \varphi_{j,k}(\cdot)\mathcal{F}^{-1}\big[    \int_{t_1}^{t_2}  \mathcal{J}_{k;k_1,k_2;l;2}^{a;\mu,\nu;j_1, j_2}(t, \xi) d t  \big]\big](x)\big\|_{L^2}\\
& \lesssim   \sum_{ \begin{subarray}{c} 
(k_1,k_2)\in \chi_k^1\\
  k_1\geq k+100 +5 \max\{\tilde{c}, -c\}
\end{subarray}}     2^{-k_2 + (1+\alpha)j+2\delta m } 2^{l/2} 2^{3k_1 + 2l-8k_{1,+}}   \epsilon_1^2\lesssim \epsilon_1^2.
\end{split}
\ee

For the non-threshold case, $l\in (-m/2+\delta m, -100  - 2\max\{\tilde{c}, -c\}]\cap \Z,$ as in the obtained estimate \eqref{july20eqn16}, we can rule out the case $\max\{j_1,j_2\}\leq m + l -\delta m $ by doing integration by parts in ``$\eta$'' $\delta^{-2}$ times. Now, we focus on the case $\max\{j_1,j_2\}\geq m + l -\delta m $. 

  From the estimate  \eqref{july21eqn11}, the $L^2-L^\infty$ type bilinear estimate, the $L^\infty$ decay estimate \eqref{decayestimatefinal}, and  the estimate \eqref{july13eqn58} in Lemma \ref{fixedtimeest}, we have
\[
 \sum_{ \begin{subarray}{c} 
(k_1,k_2)\in \chi_k^1\\
  k_1\geq k+100 +5 \max\{\tilde{c}, -c\}
\end{subarray}}   \sum_{ \begin{subarray}{c}
j_i\in [-k_{i,-}, \infty)\cap \Z, i=1,2\\
 \max\{j_1,j_2\}\geq   m+ l-\delta m \\
\end{subarray} } 2^{(1+\alpha)j}\|P_k\big[ \varphi_{j,k}(\cdot)\mathcal{F}^{-1}\big[    \int_{t_1}^{t_2}  \mathcal{J}_{k;k_1,k_2;l;2}^{a;\mu,\nu;j_1, j_2}(t, \xi) d t \big]\big](x)\|_{L^2}
\]
\be\label{july21eqn9}
\begin{split}
& \lesssim   \sum_{ \begin{subarray}{c} 
(k_1,k_2)\in \chi_k^1\\
  k_1\geq k+100 +5 \max\{\tilde{c}, -c\}
\end{subarray}}   \sum_{ \begin{subarray}{c}
j_i\in [-k_{i,-}, \infty)\cap \Z, i=1,2\\
 \max\{j_1,j_2\}\geq   m+ l-\delta m \\
\end{subarray} }         2^{-k_2 + m/4 + 2\alpha m  } 2^{(1+\alpha)j -(1+\alpha)  \max\{j_1,j_2\} } 2^{-m+  k_2-5k_{2,+}} \epsilon_1^2\\
& \lesssim  2^{- m/10}\epsilon_1^2. 
\end{split}
\ee

  $\oplus$\qquad The estimate of $\mathcal{J}_{k;k_1,k_2;l;3}^{a;\mu,\nu;j_1, j_2}(t, \xi).$

 Recall the definition of the cutoff function $\varphi_{3}(\xi, \eta)$ in  \eqref{july20eqn4}. From \eqref{july23eqn1}, for any $(\xi,\eta)\in supp\big(\varphi_{3}(\xi, \eta) $ $\psi_k(\xi)\psi_{k_1}(\xi-\eta)\psi_{k_2}(\eta)\psi_{b}( |\slashed\eta|/|\eta|)\big), b\in \Z$,  we have
\be\label{july21eqn23}
|\nabla_\eta\Phi_{\mu, \nu}^i ( \xi ,\eta) |\gtrsim 2^{2b} + \big|\frac{\slashed \xi -\slashed \eta}{|\xi-\eta|} + \frac{(c_1^2 \eta_2, c_2^2 \eta_3) }{|\eta|}\big|\gtrsim 2^{b}. 
  \ee
Based on the possible size of $|\slashed\eta|/|\eta|$, we decompose $\mathcal{J}_{k;k_1,k_2;l;3}^{a;\mu,\nu;j_1, j_2}(t, \xi) $ further as follows,
  \be\label{2023jan28eqn5}
  \begin{split}
\mathcal{J}_{k;k_1,k_2;l;3}^{a;\mu,\nu;j_1, j_2}(t, \xi) &= \sum_{b\in[-m/2+\delta m,2]\cap\Z} \mathcal{J}_{k;k_1,k_2;l;3;b}^{a;\mu,\nu;j_1, j_2}(t, \xi)\\
\mathcal{J}_{k;k_1,k_2;l;3;b}^{a;\mu,\nu;j_1, j_2}(t, \xi) &:=\int_{\R^3} e^{i t \Phi_{\mu, \nu}^a ( \xi ,\eta)} q^a_{\mu, \nu} (\xi-\eta, \eta) \widehat{h_{1;j_1,k_1}^{\mu}}(t, \xi-\eta) \widehat{h_{2;j_2,k_2}^{\nu}}(t, \eta) \psi_k(\xi) \varphi_{3}(\xi, \eta)  \\ 
 & \quad \times    \varphi_{l;-m/2+\delta m}( l_{\mu,\nu}(\xi, \eta))   \varphi_{ b;-m/2+\delta m}( |\slashed\eta|/|\eta|)  d\eta. 
\end{split}
\ee

We first consider the threshold case i.e.,  $b=-m/2+\delta m$. From the estimate of the symbol in \eqref{L1kernellocs},   and the volume of support of $\xi,\eta$,   we have 
\be\label{july23eqn38}
\begin{split}
 \sum_{ \begin{subarray}{c} 
(k_1,k_2)\in \chi_k^1\\
  k_1\geq k+100 +5 \max\{\tilde{c}, -c\}
\end{subarray}}  2^{(1+\alpha)j} &\big\| \sum_{ \begin{subarray}{c}
j_i\in [-k_{i,-}, \infty)\cap \Z, i=1,2\\
\end{subarray} }  P_k\big[ \varphi_{j,k}(\cdot)\mathcal{F}^{-1}\big[    \mathcal{J}_{k;k_1,k_2;3;b}^{\mu,\nu;j_1, j_2;l}(t, \xi) \big]\big](x)\big\|_{L^2}\\
& \lesssim  \sum_{ \begin{subarray}{c} 
(k_1,k_2)\in \chi_k^1\\
  k_1\geq k+100 +5 \max\{\tilde{c}, -c\}
\end{subarray}}     2^{2b + (1+\alpha)j} 2^{b }  2^{3k_1 + 2b- 8k_{1,+}}  \epsilon_1^2\lesssim 2^{-  5m/4}\epsilon_1^2.
\end{split}
\ee

For the non-threshold case, $b\in (-m/2+\delta m, 2]\cap \Z,$ from the   estimate \eqref{july21eqn23}, we can rule out the case $\max\{j_1,j_2\}\leq m +b-\delta m $ by doing integration by parts in ``$\eta$'' $\delta^{-2}$ times. Moreover, as in obtaining  
the rough estimate \eqref{july20eqn18}, from the estimate \eqref{L1kernellocs}, we rule out further the case $\max\{j_1,j_2\}\geq m +b-\delta m $ and $b\leq -10\delta m .$ 

Now, it would be sufficient to consider the case $\max\{j_1,j_2\}\geq m +b-\delta m $ and $b\in [ -10\delta m,2]\cap \Z.$   Due to symmetry between two inputs in the High $\times$ High type interaction, without loss of generality, we assume that $j_1=\max\{j_1,j_2\}.$    For any fixed $j\in [0, (1-\delta)^{-1}\big(\max\{m, -k_{-}\}+ \delta  m\big)  ]\cap \Z$, based on the size of $j_1$, we split into two sub-cases as follows. 

$\dagger$\qquad If $j_1\geq j+2b/(1+2\alpha).$

 From the estimate  \eqref{L1kernellocs}, the $L^2-L^\infty$ type bilinear estimate, and the $L^\infty$ decay estimate \eqref{decayestimatefinal},   we have
\[
 \sum_{ \begin{subarray}{c} 
(k_1,k_2)\in \chi_k^1\\
  k_1\geq k+100 +5 \max\{\tilde{c}, -c\}
\end{subarray}}   2^{(1+\alpha)j}\| \sum_{ \begin{subarray}{c}
j_i\in [-k_{i,-}, \infty)\cap \Z, i=1,2\\
j_1=\max\{j_1,j_2\}\geq   j+2b/(1+2\alpha)\\
\end{subarray} }  P_k\big[ \varphi_{j,k}(\cdot)\mathcal{F}^{-1}\big[    \int_{t_1}^{t_2} \mathcal{J}_{k;k_1,k_2;l;3;b}^{a;\mu,\nu;j_1, j_2}(t, \xi) d t \big]\big](x)\|_{L^2}
\]
\be\label{july21eqn21}
\lesssim    \sum_{ \begin{subarray}{c} 
(k_1,k_2)\in \chi_k^1\\
  k_1\geq k+100 +5 \max\{\tilde{c}, -c\}
\end{subarray}}    \sum_{ \begin{subarray}{c}
 \tilde{j}\geq     j+2b/(1+2\alpha) \\
\end{subarray} }    2^{m+ 2a} 2^{(1+\alpha)j -(1+\alpha)\tilde{j}+ \delta m } 2^{-m+  k_2-8k_{2,+}}\epsilon_1^2 \lesssim  2^{  \delta m +\delta a}\epsilon_1^2. 
\ee

$\dagger$\qquad If $j_1\leq j+2 b/(1+2\alpha).$

For this sub-case, we have $ m-100\delta m \leq j\leq m +10\delta m $. Recall \eqref{2023jan28eqn5}.  Note that, for any $\gamma\in \Z_+^3$, we have 
\be\label{july21eqn27}
\begin{split}
 \nabla_\xi^\gamma \mathcal{J}_{k;k_1,k_2;l;3;b}^{a;\mu,\nu;j_1, j_2}(t, \xi)  &=\sum_{\gamma_1+\gamma_2=\gamma} \widetilde{J}_{\gamma_1, \gamma_2}^\gamma(t, \xi) \\ 
 \widetilde{J}_{\gamma_1, \gamma_2}^\gamma(t, \xi)  &=\int_{\R^3} e^{i t \Phi_{\mu, \nu}^a ( \xi ,\eta)} \big(i t \nabla_{\xi}\Phi_{\mu, \nu}^a ( \xi ,\eta)\big)^{\gamma_1} \nabla_\xi^{\gamma_2}\big[ q^a_{\mu, \nu} (\xi-\eta, \eta) \widehat{h_{1;j_1,k_1}^{\mu}}(t, \xi-\eta)  \\ 
 & \quad \times  \psi_k(\xi)   \varphi_{3}(\xi, \eta)    \varphi_{l;-m/2+\delta m}( l_{\mu,\nu}(\xi, \eta))  \varphi_{b;-m/2+\delta m}( |\slashed\eta|/|\eta|)  \big] \widehat{h_{2;j_2,k_2}^{\nu}}(t, \eta)   d\eta. 
\end{split}
\ee

For $\widetilde{J}_{\gamma_1, \gamma_2}^\gamma(t, \xi) $, we do integration by parts in ``$\eta$'' $|\gamma_1|$ times. As a result, we have
\be\label{july21eqn28}
\begin{split}
 \widetilde{J}_{\gamma_1, \gamma_2}^\gamma(t, \xi)   &=\int_{\R^3} e^{i t \Phi_{\mu, \nu}^i ( \xi ,\eta)} \nabla_\eta \cdot \Big[ \frac{i \nabla_\eta \Phi_{\mu, \nu}^a ( \xi ,\eta) }{|\nabla_\eta \Phi_{\mu, \nu}^a ( \xi ,\eta)|^2} \circ\cdots \nabla_\eta \cdot \Big[ \frac{i \nabla_\eta \Phi_{\mu, \nu}^a ( \xi ,\eta) }{|\nabla_\eta \Phi_{\mu, \nu}^a ( \xi ,\eta)|^2} \big(i \nabla_{\xi}\Phi_{\mu, \nu}^a ( \xi ,\eta)\big)^{\gamma_1}\\ 
 &\qquad \times  \nabla_\xi^{\gamma_2}\big[ q^a_{\mu, \nu} (\xi-\eta, \eta) \varphi_{3}(\xi, \eta) \psi_k(\xi)  \widehat{h_{1;j_1,k_1}^{\mu}}(t, \xi-\eta)    \varphi_{l;-m/2+\delta m}( l_{\mu,\nu}(\xi, \eta))\\ 
 &\qquad \times       \varphi_{b;-m/2+\delta m}( |\slashed\eta|/|\eta|)  \big] \widehat{h_{2;j_2,k_2}^{\nu}}(t, \eta) \Big]\circ\cdots\Big]  d\eta. 
\end{split}
\ee

Recall the range of $k,j,j_1,k_1,$ and $l$. The worst scenario happens when $\nabla_\eta^\kappa \nabla_\xi^{\gamma_1}$ hits $ \widehat{h_{1;j_1,k_1}^{\mu}}(t, \xi-\eta)$, where $\gamma\in \Z_+^3,$ s.t., $|\kappa|=|\gamma_2|$. Moreover, from the estimate \eqref{july21eqn23} and \eqref{L1kernellocs}, for any $n\in \Z_+$,  from the estimate \eqref{july27eqn3} in Lemma \ref{kernelest},  the following estimate holds for the kernel of typical symbols appeared in \eqref{july21eqn28}, 
\[
\begin{split}
  \|\mathcal{F}^{-1}_{(\xi, \eta)\rightarrow (x,y)}\big[& \underbrace{ \frac{i \nabla_\eta \Phi_{\mu, \nu}^a ( \xi ,\eta) }{|\nabla_\eta \Phi_{\mu, \nu}^a ( \xi ,\eta)|^2} \otimes \circ\cdots \circ \otimes\frac{i\nabla_\eta \Phi_{\mu, \nu}^a ( \xi ,\eta) }{|\nabla_\eta \Phi_{\mu, \nu}^a ( \xi ,\eta)|^2}}_{\text{$n$-times }}     q^a_{\mu, \nu} (\xi-\eta, \eta)    \varphi_{l;-m/2+\delta m}( l_{\mu,\nu}(\xi, \eta))    \\ 
&\times  \psi_k(\xi)\psi_{k_1}(\xi-\eta)   \psi_{k_2}(\eta) \varphi_{b;-m/2+\delta m}( |\slashed\eta|/|\eta|)  \big]  
 \big]\|_{L^1_{x,y}}  
    \lesssim 2^{-(6+\delta)b - (n-2)  b}.
    \end{split}
\]

From the above estimate, the $L^2-L^\infty$ type bilinear estimate, and the $L^\infty$ decay estimate \eqref{decayestimatefinal}, for any $\gamma\in \Z_+^3$,  we have
\[
  \sum_{ \begin{subarray}{c} 
(k_1,k_2)\in \chi_k^1\\
  k_1\geq k+100 +5 \max\{\tilde{c}, -c\}
\end{subarray}}     2^{(1+ \alpha)j}\| \sum_{ \begin{subarray}{c}
j_i\in [-k_{i,-}, \infty)\cap \Z, i=1,2\\
j_1=\max\{j_1,j_2\}\leq    j+2b/(1+2\alpha)\\
\end{subarray} } P_k\big[ \varphi_{j,k}(\cdot)\mathcal{F}^{-1}\big[      \widetilde{J}_{\gamma_1, \gamma_2}^\gamma(t, \xi) \big]\big](x)\|_{L^2}
\]
\[
\lesssim    \sum_{ \begin{subarray}{c} 
(k_1,k_2)\in \chi_k^1\\
  k_1\geq k+100 +5 \max\{\tilde{c}, -c\}
\end{subarray}}    \sum_{ \begin{subarray}{c}
\tilde{j}\leq     j+2b /(1+2\alpha) \\
\end{subarray} }    2^{-(6+\delta)b - (|\gamma|-2) b + |\gamma|j_1}  2^{(1+\alpha)j -(1+\alpha)\tilde{j}+ \delta m } 2^{-m+  k_2-5k_{2,+}}\epsilon_1^2 
\]
\be\label{july21eqn31}
  \lesssim 2^{|\gamma|j -(6 +\delta)b -  |\gamma|  b + 2|\gamma|b/(1+2\alpha)} 2^{ -m+ \delta m}\epsilon_1^2     . 
\ee

 Now, we let $|\gamma|=10$, which is sufficient for our purpose.   Recall \eqref{july21eqn27}. From the estimate \eqref{july21eqn31}, we have
 \[
 \sum_{ \begin{subarray}{c} 
(k_1,k_2)\in \chi_k^1\\
  k_1\geq k+100 +5 \max\{\tilde{c}, -c\}
\end{subarray}}   2^{(1+\alpha)j}\| \sum_{ \begin{subarray}{c}
j_i\in [-k_{i,-}, \infty)\cap \Z, i=1,2\\
j_1=\max\{j_1,j_2\}\leq    j+2b/(1+2\alpha)\\
\end{subarray} } P_k\big[ \varphi_{j,k}(\cdot)\mathcal{F}^{-1}\big[     \mathcal{J}_{k;k_1,k_2;l;3;b}^{a;\mu,\nu;j_1, j_2}(t, \xi) \big]\big](x)\|_{L^2}
\]
 \[
\lesssim \sum_{ \begin{subarray}{c} \gamma, \gamma_1, \gamma_2\in \Z_+^3 \\ 
|\gamma|=10,
\gamma_1+\gamma_2=\gamma\\
\end{subarray}}  \sum_{ \begin{subarray}{c} 
(k_1,k_2)\in \chi_k^1\\
  k_1\geq k+100 +5 \max\{\tilde{c}, -c\}
\end{subarray}}  \|\sum_{ \begin{subarray}{c}
j_i\in [-k_{i,-}, \infty)\cap \Z, i=1,2\\
j_1=\max\{j_1,j_2\}\leq    j+2b/(1+2\alpha)\\
\end{subarray} } 2^{(1+\alpha)j-|\gamma|j} P_k\big[ \varphi_{j,k}(\cdot)\mathcal{F}^{-1}\big[     \widetilde{J}_{\gamma_1, \gamma_2}^\gamma(t, \xi) \big]\big](x)\|_{L^2}
\]
\be\label{july21eqn32}
\lesssim 2^{  -(6+\delta)b   + 10(1-2\alpha) b/(1+2\alpha)} 2^{ -m+ \delta m}\epsilon_1^2  \lesssim 2^{ -m+ \delta m+b}\epsilon_1^2     .
\ee

After summing up the case $b\in [ -10\delta m,2]\cap \Z.$  , $l\in [2b,  -100  - 2\max\{\tilde{c}, -c\}]\cap \Z$ in the obtained estimates \eqref{july21eqn21} and \eqref{july21eqn32}, we  finish the proof for the case  $l\in [ -m/2+\delta m,  -100  - 2\max\{\tilde{c}, -c\}]\cap \Z$. 

It remains to consider the case $l\in [- 100  - 5\max\{\tilde{c}, -c\} , 2]\cap \Z$.   Note that, from the rough estimate \eqref{july23eqn1}, the lower bound $2^{ l}$  plays the same role of  the lower bound $2^b$ we use in \eqref{july21eqn23} for  the estimate of $\mathcal{J}_{k;k_1,k_2;l;3;b}^{a;\mu,\nu;j_1, j_2}(t, \xi)$.  Hence, after rerun the argument used for the estimate $   \mathcal{J}_{k;k_1,k_2;l;3;b}^{a;\mu,\nu;j_1, j_2}(t, \xi) $, the estimate of  $   \mathcal{J}_{k;k_1,k_2  }^{a;\mu,\nu }(t, \xi) $ follows similarly for  the case    $l\in [- 100  - 2\max\{\tilde{c}, -c\} , 2]\cap \Z$.  Hence finishing the proof of our desired estimate \eqref{july3eqn41}. 
\end{proof}

\subsection{The High $\times$ Low type and the Low $\times$ High type interactions}

In this subsection, we consider the High $\times$ Low type and the Low $\times$ High type  interactions with one of  the input frequencies  much smaller than the out frequency, e.g., $\min\{k_1, k_2\}\leq k- 5\max\{\tilde{c}, -c\}  -100.$

In the following Lemma, we first consider the non-resonance case. 
\begin{lemma}\label{lowhighosc}
Under the bootstrap assumption \eqref{bootstrap}, for any $t \in [2^{m-1}, 2^m]\subset[0, T], m\in \Z_+$, we have
\be\label{july21eqn41}
\sup_{k\in \Z}\sum_{a=1,2}\sum_{\mu\in\{+,-\}} \sum_{\begin{subarray}{c}
(k_1,k_2)\in \chi_k^2\\ 
 k_1  \leq k- 5\max\{\tilde{c}, -c\}  -100\\
\end{subarray}
 } \|\int_{t_1}^{t_2} \mathcal{F}^{-1}\big[   \mathcal{J}_{k;k_1,k_2}^{a;\mu,- }(t, \xi) \big]  d t  \|_{Z} \lesssim \epsilon_1^2,
\ee
\be\label{july21eqn43}
\sup_{k\in \Z}\sum_{a=1,2}\sum_{\nu\in\{+,-\}} \sum_{\begin{subarray}{c}
(k_1,k_2)\in \chi_k^3\\ 
 k_2  \leq k- 5\max\{\tilde{c}, -c\}  -100\\
\end{subarray}
 } \|\int_{t_1}^{t_2} \mathcal{F}^{-1}\big[   \mathcal{J}_{k;k_1,k_2}^{a;-,\nu }(t, \xi) \big]  d t  \|_{Z} \lesssim  \epsilon_1^2.
\ee
\end{lemma}
\begin{proof}
Since only rough estimates of the phase $\Phi_{\mu, \nu}^a ( \xi ,\eta)$ will be used in the non-resonance case,  the desired estimates \eqref{july21eqn41} and \eqref{july21eqn43} can be proved similarly, we only focus on the proof of the estimate \eqref{july21eqn41}. 

 Recall \eqref{systemeqnprof}. Same as what we did in \eqref{july13eqn70}, we also do space localization for two inputs. As in the proof of Lemma \ref{fixedtimeest}, we first rule out the case   $j\geq  \max\{m, -k_{-}\}+ \delta (m+j) .$ It would be sufficient to consider the case  $j\leq  \max\{m, -k_{-}\}+ \delta (m+j) .$  As in the decomposition in  \eqref{july21eqn71},  based on the size of $l_{\mu, \nu}(\xi, \eta),$ we do dyadic decomposition for   $\mathcal{J}_{k;k_1,k_2}^{a;\mu,\nu;j_1, j_2}(t, \xi) $. 

  Moreover, as the estimates \eqref{july21eqn81} and \eqref{july21eqn84} are still valid, it would be sufficient to consider  the  non-threshold  case $l\in (-m/2+\delta m, \tilde{c}+10]\cap \Z$,  $\max\{j_1,j_2\}\geq (1-\delta)m + l$ and the  threshold case $l=-m/2+\delta m$.

Based on the size of $k$, we split into two cases as follows. 

 $\bullet$\quad If $k\leq -m/4 -2\alpha m .$

 We first consider the  threshold case $l=-m/2+\delta m$. From the $L^2-L^\infty$ type bilinear estimate, the volume of support of the input putted in $L^2$, and the decay estimate \eqref{decayestimatefinal},  since $k\leq -m/4-2 \alpha m$,   we have
\be\label{july23eqn41}
\begin{split}
 & \sum_{\begin{subarray}{c}
(k_1,k_2)\in \chi_k^2\\ 
 k_1    \leq k- 5\max\{\tilde{c}, -c\}  -100 
 \end{subarray}  }    2^{(1+\alpha)j}\big\| \sum_{ \begin{subarray}{c}
j_i\in [-k_{i,-}, \infty)\cap \Z, i=1,2\\
\end{subarray} }  P_k\big[ \varphi_{j,k}(\cdot)\mathcal{F}^{-1}\big[   \mathcal{J}_{k;k_1,k_2;l}^{\mu,\nu;j_1, j_2}(t, \xi) \big]\big](x)\big\|_{L^2}\\
 &  \lesssim     \sum_{\begin{subarray}{c}
(k_1,k_2)\in \chi_k^2\\ 
 k_1    \leq k- 5\max\{\tilde{c}, -c\}  -100 
 \end{subarray}  }     2^{l + (1+\alpha)j} 2^{ l/2}2^{-m+  k_1} \epsilon_1^2\lesssim 2^{- (1+10\delta)m}\epsilon_1^2.\\ 
 \end{split}
\ee
For   the  non-threshold case $l\in (-m/2+\delta m, \tilde{c}+10]\cap \Z$, by using the strategies used in obtaining \eqref{july13eqn16}, we have
\be\label{july22eqn2}
 \begin{split}
&     \sum_{\begin{subarray}{c}
(k_1,k_2)\in \chi_k^2\\ 
 k_1    \leq k- 5\max\{\tilde{c}, -c\}  -100 
 \end{subarray}  }    \sum_{ \begin{subarray}{c}
j_i\in [-k_{i,-}, \infty)\cap \Z, i=1,2\\
\max\{j_1,j_2\}\geq (1-\delta)m+  l \\
\end{subarray} } 2^{(1+\alpha)j}\|P_k\big[ \varphi_{j,k}(\cdot)\mathcal{F}^{-1}\big[   \mathcal{J}_{k;k_1,k_2;l}^{\mu,\nu;j_1, j_2}(t, \xi) \big]\big](x)\|_{L^2}\\ 
&\lesssim   \sum_{\begin{subarray}{c}
(k_1,k_2)\in \chi_k^2\\ 
 k_1    \leq k- 5\max\{\tilde{c}, -c\}  -100 
 \end{subarray}  }   \sum_{ \begin{subarray}{c}
j_i\in [-k_{i,-}, \infty)\cap \Z, i=1,2\\
\max\{j_1,j_2\}\geq (1-\delta)m + l \\
\end{subarray} }2^{(1+\alpha)j+ l} 2^{-(1+\alpha)\max\{j_1,j_2\} + \delta m }   \min\big\{2^{3k_1/2}, 2^{-m+  k}\big\} \epsilon_1^2\\
& \lesssim 2^{-m- \delta m}\epsilon_1^2.
 \end{split}
\ee
Hence finishing the proof of the case  $k\leq  -m/4 -2\alpha m $. 

 $\bullet$\quad If  $k\geq  -m/4 -2\alpha m$.

Since $ k_1    \leq k- 5\max\{\tilde{c}, -c\}  -100  $ and $\nu=-$,  for this case, we have
 \[
 \forall a \in\{1,2\}, \qquad |\Phi_{\mu, \nu}^a ( \xi ,\eta)|\sim 2^k\gtrsim 2^{-m/4-2\alpha m}. 
 \]

 To take advantage of high oscillation in time, we do normal form transformation once. As a result, we have
\be\label{2023jan16eqn34}
 \begin{split}
  &\int_{t_1}^{t_2}\mathcal{J}_{k;k_1,k_2;l}^{a;\mu,\nu;j_1, j_2}(t, \xi) d t  =\sum_{n=1,2} \int_{\R^3} e^{i t_n \Phi_{\mu, \nu}^a ( \xi ,\eta)} \frac{(-1)^nq^a_{\mu, \nu} (\xi-\eta, \eta)}{i \Phi_{\mu, \nu}^a ( \xi ,\eta) } \widehat{h_{1;j_1,k_1}^{\mu}}(t_n, \xi-\eta) \widehat{h_{2;j_2,k_2}^{\nu}}(t_n, \eta)  \\
  & \times      \varphi_{l;-m/2+\delta m}( l_{\mu,\nu}(\xi, \eta))  \psi_k(\xi)  d\eta -  \int_{t_1}^{t_2} \int_{\R^3} e^{i t \Phi_{\mu, \nu}^a ( \xi ,\eta)} \frac{ q^a_{\mu, \nu} (\xi-\eta, \eta)\psi_k(\xi) }{i \Phi_{\mu, \nu}^a ( \xi ,\eta) }\\ 
   & \times \p_t\big( \widehat{h_{1;j_1,k_1}^{\mu}}(t , \xi-\eta) \widehat{h_{2;j_2,k_2}^{\nu}}(t , \eta)\big)       \varphi_{l;-m/2+\delta m}( l_{\mu,\nu}(\xi, \eta))    d\eta dt.
   \end{split}
 \ee

  By using the same strategy used in   obtaining \eqref{july13eqn15}, from the estimate in Lemma  and the volume of support of frequencies,   for the  threshold case $l=-m/2+\delta m$, we have
\be\label{july23eqn44}
\begin{split}
 &  \sum_{\begin{subarray}{c}
(k_1,k_2)\in \chi_k^2\\ 
 k_1    \leq k- 5\max\{\tilde{c}, -c\}  -100 
 \end{subarray}  }   2^{(1+\alpha)j}\big\| \sum_{ \begin{subarray}{c}
j_i\in [-k_{i,-}, \infty)\cap \Z, i=1,2\\
\end{subarray} }  P_k\big[ \varphi_{j,k}(\cdot)\mathcal{F}^{-1}\big[      \int_{t_1}^{t_2}\mathcal{J}_{k;k_1,k_2;l}^{a; \mu,\nu;j_1, j_2}(t, \xi) d t  \big]\big](x)\big\|_{L^2}\\
 &\lesssim   \sum_{\begin{subarray}{c}
(k_1,k_2)\in \chi_k^2\\ 
 k_1    \leq k- 5\max\{\tilde{c}, -c\}  -100 
 \end{subarray}  }     2^{-k+ l + (1+\alpha)j} 2^{3k/2+ l/2}2^{-m+k_1-4k_+} \epsilon_1^2\lesssim  2^{-\delta m}\epsilon_1^2.\\ 
 \end{split}
\ee
For   the  non-threshold case $l\in (-m/2+\delta m, 2]\cap \Z$, by using the strategies used in obtaining \eqref{july13eqn16}, we have
\be\label{july22eqn13}
 \begin{split}
&   \sum_{\begin{subarray}{c}
(k_1,k_2)\in \chi_k^2\\ 
 k_1    \leq k- 5\max\{\tilde{c}, -c\}  -100 
 \end{subarray}  }   \sum_{ \begin{subarray}{c}
j_i\in [-k_{i,-}, \infty)\cap \Z, i=1,2\\
\max\{j_1,j_2\}\geq (1-\delta)m+ l \\
\end{subarray} } 2^{(1+\alpha)j}\|P_k\big[ \varphi_{j,k}(\cdot)\mathcal{F}^{-1}\big[     \int_{t_1}^{t_2}\mathcal{J}_{k;k_1,k_2;l}^{a;\mu,\nu;j_1, j_2}(t, \xi) d t  \big]\big](x)\|_{L^2}\\ 
&\lesssim   \sum_{\begin{subarray}{c}
(k_1,k_2)\in \chi_k^2\\ 
 k_1    \leq k- 5\max\{\tilde{c}, -c\}  -100 
 \end{subarray}  }  \sum_{ \begin{subarray}{c}
j_i\in [-k_{i,-}, \infty)\cap \Z, i=1,2\\
\max\{j_1,j_2\}\geq (1-\delta)m + l \\
\end{subarray} }2^{(1+\alpha)j+ l} 2^{-(1+\alpha)\max\{j_1,j_2\}   }   2^{-3m/4+(\alpha+10\delta)m-k} \\ 
&\times \min\{2^{3k_1/2+m},1\} \epsilon_1^2   \lesssim 2^{-m/4}\epsilon_1^2.
 \end{split}
\ee
Hence finishing the proof of the non-resonance case.
\end{proof}

Now, we consider the time-resonance case. 
\begin{lemma}\label{lowhighess}
Under the bootstrap assumption \eqref{bootstrap}, for any $t \in [2^{m-1}, 2^m]\subset[0, T], m\in \Z_+$,   we have
\be\label{july21eqn45}
\sup_{k\in \Z}\sum_{a=1,2} \sum_{\mu\in\{+,-\}}   \sum_{(k_1,k_2)\in \chi_k^2, k_1  \leq   k- 100-5\max\{\tilde{c}, -c\} } \|\int_{t_1}^{t_2} \mathcal{F}^{-1}\big[   \mathcal{J}_{k;k_1,k_2}^{a;\mu,+ }(t, \xi) \big]  d t  \|_{Z} \lesssim 2^{  \delta m }\epsilon_1^2,
\ee
\be\label{july21eqn46}
\sup_{k\in \Z}\sum_{a=1,2}  \sum_{\nu\in\{+,-\}} \sum_{(k_1,k_2)\in \chi_k^3, k_2  \leq   k- 100-5\max\{\tilde{c}, -c\} } \|\int_{t_1}^{t_2} \mathcal{F}^{-1}\big[   \mathcal{J}_{k;k_1,k_2}^{a;+,\nu }(t, \xi) \big]  d t  \|_{Z} \lesssim 2^{  \delta m }\epsilon_1^2.
\ee
\end{lemma}
\begin{proof}

Since after changing coordinates $(\xi,\eta)\rightarrow (\xi,\xi-\eta),$   the desired estimates \eqref{july21eqn45} and \eqref{july21eqn46} can be proved similarly, we only focus on the proof of the estimate \eqref{july21eqn46}.  Recall \eqref{systemeqnprof}. Same as what we did in \eqref{july13eqn70}, we also do space localization for two inputs. As in the proof of Lemma \ref{fixedtimeest}, we first rule out the case   $j\geq  \max\{m, -k_{-}\}+ \delta (m+j) .$ It would be sufficient to consider the case  $j\leq  \max\{m, -k_{-}\}+ \delta (m+j) .$ 

 As   in  \eqref{july21eqn71}, we do dyadic decomposition for  $\mathcal{J}_{k;k_1,k_2}^{a;\mu,\nu;j_1, j_2}(t, \xi) $.  
  Note that the obtained estimates \eqref{july21eqn81} and \eqref{july21eqn84} are still valid. Moreover, similar to what we did in the proof of Lemma \ref{highhighWfixed2},  from the $L^2-L^\infty$ type bilinear estimate, the estimate \eqref{decayestimatefinal}, and the volume of support, we can rule out the case $k\leq -m/4-2\alpha m $ and the case $k_2\leq -2m/3-\alpha m .$ 

To sum up, it would be sufficient to  consider $k\geq  -m/4-2\alpha m$ and  $k_2\geq -2m/3-\alpha m$ for  the non-threshold  case $l\in (-m/2+\delta m, \tilde{c}+10]\cap \Z$, $\max\{j_1,j_2\}\geq (1-\delta)m + l$, the  threshold case $l=-m/2+\delta m$.   

 We first consider the main  case $l\in [-m/2+\delta m, -100 -  5\max\{\tilde{c}, -c\}]$  .  The case  $l\in [-100 -  5\max\{\tilde{c}, -c\}, \tilde{c}+10]\cap \Z$   can be handled in the same way as one sub-case of the main case. We will elaborate this point later. 

 To see different behavior of phases,  we use the following partition of unity,
\be\label{feb2eqn1}
\begin{split}
1 & = \phi_{\nu,1}(\xi, \eta) +  \phi_{\nu,2}(\xi, \eta)+ \phi_{\nu,3}(\xi, \eta),\\
 \phi_{\nu,1}(\xi, \eta) &= \psi_{\geq -10}(1-\nu \xi_1 \eta_1/(|\xi||\eta|))  \psi_{[-10,10]}(|\slashed \xi-\slashed \eta|\Lambda_2(\eta)/(|\xi-\eta|\sqrt{c_1^2\eta_2^4 +c_2^4 \eta_3^2})),\\ 
  \phi_{\nu,2}(\xi,\eta)& = \psi_{< -10}(1-\nu  \xi_1 \eta_1/(|\xi||\eta|) )  \psi_{[-10,10]}(|\slashed \xi-\slashed \eta|\Lambda_2(\eta)/(|\xi-\eta|\sqrt{c_1^2\eta_2^4 +c_2^4 \eta_3^2})), \\ 
   \phi_{\nu,3}(\xi, \eta)&= 1- \psi_{[-10,10]}(|\slashed \xi-\slashed \eta|\Lambda_2(\eta)/(|\xi-\eta|\sqrt{c_1^2\eta_2^4 +c_2^4 \eta_3^2})).
  \end{split}
\ee

Correspondingly, we  split  $\mathcal{J}_{k;k_1,k_2;l}^{a;+,\nu;j_1, j_2}(t, \xi)$ into three pieces as follows, 
\be
\begin{split}
\mathcal{J}_{k;k_1,k_2;l}^{a;+,\nu;j_1, j_2}(t, \xi)& =\sum_{b\in\{1,2,3\}} \mathcal{J}_{k;k_1,k_2;l;b}^{a;+,\nu;j_1, j_2}(t, \xi)\\ 
\forall b\in\{1,2,3\}, \quad \mathcal{J}_{k;k_1,k_2;l;b}^{a;+,\nu;j_1, j_2}(t, \xi)&:=\int_{\R^3} e^{i t \Phi_{+, \nu}^a ( \xi ,\eta)} q^a_{+, \nu} (\xi-\eta, \eta) \widehat{h_{1;j_1,k_1}^{ }}(t, \xi-\eta) \widehat{h_{2;j_2,k_2}^{\nu}}(t, \eta)   \\
&\quad \times \psi_k(\xi) \varphi_{l;-m/2+\delta m}( l_{+,\nu}(\xi, \eta))    \phi_{\nu,b}(\xi, \eta) d\eta\\
\end{split}
\ee

$\bullet$\qquad If $b=1.$

Note that, for any $(\xi, \eta)\in supp\big( \psi_{k_1}(\xi-\eta)\psi_{k_2}(\eta)  \varphi_{l;-m/2+\delta m}( l_{\mu,\nu}(\xi, \eta)) \phi_{\nu,1}(\eta)    \big)$, we have
\[
\big|  \Phi_{  +, \nu}^a ( \xi ,\xi-\eta)\big|\sim |\eta|\sim 2^{k_2}. 
\]

Recall that $k_2\geq -m/4-2\alpha m $. To take advantage of high oscillation in time, as in \eqref{july21eqn1}, we do integration by parts in time once. As a result, similar to the obtained estimate \eqref{july23eqn23}, for the threshold case $l=-m/2+\delta m,$ we have
\be 
\begin{split}
   \sum_{\begin{subarray}{c}  
   (k_1,k_2)\in \chi_k^3\\ 
 k_2\geq -m/4-2\alpha m\\
    \end{subarray} } 2^{(1+\alpha)j} &\big\| \sum_{ \begin{subarray}{c}
j_i\in [-k_{i,-}, \infty)\cap \Z, i=1,2\\
\end{subarray} }  P_k\big[ \varphi_{j,k}(\cdot)\mathcal{F}^{-1}\big[    \int_{t_1}^{t_2} \mathcal{J}_{k;k_1,k_2; -m/2+\delta m ;1}^{a;+,\nu;j_1, j_2}(t, \xi) d t  \big]\big](x)\big\|_{L^2}\\
& \lesssim      \sum_{\begin{subarray}{c}  
   (k_1,k_2)\in \chi_k^3\\ 
 k_2\geq -m/4-2\alpha m\\
    \end{subarray} }     2^{-k_2+ l + (1+\alpha)j+2\alpha m }  2^{-m+l/2+  k_1-4k_{1,+} } \epsilon_1^2\lesssim \epsilon_1^2.
\end{split}
\ee
Similar to the obtained estimate  \eqref{july21eqn3}, for the non-threshold case $l\in (-m/2+\delta m,  \tilde{c}+10]\cap \Z$, we have
\[
    \sum_{\begin{subarray}{c}  
   (k_1,k_2)\in \chi_k^3\\ 
 k_2\geq -m/4-2\alpha m\\
    \end{subarray} }  \sum_{ \begin{subarray}{c}
j_i\in [-k_{i,-}, \infty)\cap \Z, i=1,2\\
 \max\{j_1,j_2\}\geq   m+ l-\delta m \\
\end{subarray} } 2^{(1+\alpha)j}\|P_k\big[ \varphi_{j,k}(\cdot)\mathcal{F}^{-1}\big[    \int_{t_1}^{t_2} \mathcal{J}_{k;k_1,k_2;l;1}^{a;+,\nu;j_1, j_2}(t, \xi) d t \big]\big](x)\|_{L^2}
\]
\[
\lesssim    \sum_{\begin{subarray}{c}  
   (k_1,k_2)\in \chi_k^3\\ 
 k_2\geq -m/4-2\alpha m\\
    \end{subarray} }    \sum_{ \begin{subarray}{c}
j_i\in [-k_{i,-}, \infty)\cap \Z, i=1,2\\
 \max\{j_1,j_2\}\geq   m+ l-\delta m \\
\end{subarray} }   2^{ m/4+ (\alpha+30\delta)m} 2^{-k_2+  l} 2^{(1+\alpha)j -(1+\alpha)  \max\{j_1,j_2\}  } 2^{-m -4k_{2,+}} \epsilon_1^2 \lesssim  \epsilon_1^2. 
\]

$\bullet$\qquad If $b=2.$

For this case, we have $|\xi_1|-|\xi_1-\eta_1|-\nu|\eta_1|=0.$
Recall \eqref{july23eqn1}. Note that, for any $(\xi, \eta)\in supp(  \varphi_{l;-m/2+\delta m}( l_{\mu,\nu}(\xi, \eta) )\phi_{\nu,2}(\eta)  \psi_k(\xi)\psi_{k_2}(\eta))$, we have 
\be\label{july22eqn21}
\begin{split}
\big|\nabla_\xi \big( \Phi_{  +, \nu}^a ( \xi ,\eta)\big)\big|  =   \big| \nabla_\xi \Lambda_a(\xi) - \frac{\xi-\eta}{|\xi-\eta|}\big| \lesssim 2^{l/2}, \\
\end{split}
\ee 
Moreover, if we change coordinates $(\xi, \eta)\longrightarrow(\xi, \xi-\eta)$, for any $(\xi, \eta)\in supp(  \varphi_{l;-m/2+\delta m}( l_{\mu,\nu}(\xi, \xi-\eta) )\phi_{\nu,2}(\xi-\eta)  \psi_k(\xi)\psi_{k_2}(\xi-\eta))$, we still have 
\be\label{2023jan29eqn31}
\begin{split}
\big|\nabla_\xi \big( \Phi_{  +, \nu}^a ( \xi ,\xi-\eta)\big) \big|  =   \big| \nabla_\xi \Lambda_a(\xi) - \nu  \frac{\xi-\eta}{\Lambda_2(\xi-\eta) }\big| \lesssim 2^{l/2}. \\ 
\end{split}
\ee

Thanks to the above improved estimates, we can improve the upper bound of $j$. As in the proof of Lemma \ref{fixedtimeest}, by doing integration by parts in ``$\xi$'' $\delta^{-2}$ times, we rule out further the case   $j\geq  \max\{m, -k_{-}\}+ l/2 + \delta m .$  It would be sufficient to consider the case $j\leq   \max\{m, -k_{-}\}+ l/2 + \delta m .$

$\dagger$ The threshold case, i.e., $l=-m/2+\delta m$.

Similar to what we did in \eqref{2023jan28eqn5}, we localize further based on the size of $\slashed \eta/|\eta|$   as follows,
  \be\label{2023jan29eqn40}
  \begin{split}
\mathcal{J}_{k;k_1,k_2;l;2}^{a;+,\nu;j_1, j_2}(t, \xi) &= \sum_{b\in[-m/2+\delta m,l/2+2]\cap\Z} \mathcal{J}_{k;k_1,k_2;l;2;b}^{a;+,\nu;j_1, j_2}(t, \xi)\\
\mathcal{J}_{k;k_1,k_2;l;2;b}^{a;+,\nu;j_1, j_2}(t, \xi) &:=\int_{\R^3} e^{i t \Phi_{+, \nu}^a ( \xi ,\eta)} q^a_{+, \nu} (\xi-\eta, \eta) \widehat{h_{1;j_1,k_1}^{ }}(t, \xi-\eta) \widehat{h_{2;j_2,k_2}^{\nu}}(t, \eta) \psi_k(\xi)    \\ 
 & \quad \times    \varphi_{l;-m/2+\delta m}( l_{+,\nu}(\xi, \eta))   \varphi_{ b;-m/2+\delta m}( |\slashed\eta|/|\eta|)  \phi_{\nu,2}(\xi, \eta)   d\eta. 
\end{split}
\ee

From the volume of support of $\xi,\eta$,  the estimate of symbol in  \eqref{july13eqn11}, and  the estimate of the volume of support of frequency variable in  \eqref{2023jan30eqn71}, we have 
\[
  \sum_{(k_1,k_2)\in \chi_k^3,  k_2  \leq   k- 100-5\max\{\tilde{c}, -c\}}   2^{(1+\alpha)j}\big\| \sum_{ \begin{subarray}{c}
j_i\in [-k_{i,-}, \infty)\cap \Z, i=1,2\\
\end{subarray} }  P_k\big[ \varphi_{j,k}(\cdot)\mathcal{F}^{-1}\big[  \mathcal{J}_{k;k_1,k_2;l;2;b}^{a;+,\nu;j_1, j_2}(t, \xi)\big]\big](x)\big\|_{L^2}
\]
\be\label{july13eqn65}
\lesssim     \sum_{(k_1,k_2)\in \chi_k^3,   k_2  \leq   k- 100-5\max\{\tilde{c}, -c\} }    2^{2b + (1+\alpha)j} 2^{ 3k/2+l/2}2^{k_2+2k +2l } 2^{-5(k_{1,+}+k_{2,+})}\epsilon_1^2.
\ee
The above estimate is sufficient to  rule out the case  $b\leq -m/4-2\alpha m $ or $k_1\notin [ -2\alpha m, 2\alpha m ]$.

 It remains to consider the case $b\in [-m/4-2\alpha m , -m/4+\delta m]\cap \Z $ and $k_1\in [ -2\alpha m, 2\alpha m ]\cap \Z . $   Note that, from Taylor expansion and the   approximations in \eqref{july23eqn1} and \eqref{2023jan30eqn51}, the following approximation holds for the resonance case, 
\be\label{2023jan29eqn61}
\begin{split} 
\Phi_{\mu, \nu }^1(\xi, \eta) & = \frac{(1-c_1^2)\eta_2^2}{2|\xi_1|}(1 + \frac{\mu \nu c_1^2 |\xi_1-\eta_1|}{|\eta_1|})+ \frac{(1-c_2^2)\eta_3^2}{2|\xi_1|}(1 + \frac{ \mu \nu  c_2^2 |\xi_1-\eta_1|}{|\eta_1|}) + \mathcal{O}(2^{-3m/4+10\alpha m }) \\ 
& =(1 + \frac{ \mu \nu  c_1^2 |\xi_1-\eta_1|}{|\eta_1|})^{-1}\frac{(1-c_1^2)\xi_2^2 }{2|\xi_1|} + (1 + \frac{ \mu \nu  c_2^2 |\xi_1-\eta_1|}{|\eta_1|})^{-1}\frac{(1-c_2^2)\xi_3^2}{2|\xi_1|} + \mathcal{O}(2^{-3m/4+10\alpha m }),\\ 
\Phi_{\mu, \nu }^2(\xi, \eta)& =   \frac{-\mu \nu |\xi_1-\eta_1| }{2|\eta_1||\xi_1| }\big[(1-c_1^2)\eta_2^2 (1 + \frac{\mu \nu c_1^2 |\xi_1-\eta_1|}{|\eta_1|}) + (1-c_2^2) \eta_3^2 (1 + \frac{ \mu \nu c_2^2  |\xi_1-\eta_1|}{|\eta_1|})  \big] 
\\
&+ \mathcal{O}(2^{-3m/4+10\alpha m })=\mathcal{O}(2^{-3m/4+10\alpha m }) \\
& +   \frac{-\mu \nu |\xi_1-\eta_1| }{2|\eta_1||\xi_1| }\big[(1-c_1^2)\xi_2^2 (1 + \frac{\mu \nu c_1^2 |\xi_1-\eta_1|}{|\eta_1|})^{-1} + (1-c_2^2) \xi_3^2 (1 + \frac{ \mu \nu c_2^2  |\xi_1-\eta_1|}{|\eta_1|})^{-1}  \big] ,\\
\end{split}
\ee
  We remark that, the above expansion only depends on the fact that  we have $|\xi|-\mu |\xi_1-\eta_1|-\nu |\eta_1|=0$ in the resonance case and the fact that $l< -m/2+\alpha m, k\leq 2\alpha m .$

From  \eqref{2023jan29eqn61},  $\forall a\in\{1,2\}$,  we have a lower bound for the size of phases for the case we are considering, 
\[
 |\Phi_{+, \nu }^a(\xi, \eta)| \varphi_{l;-m/2+\delta m}( l_{+,+}(\xi, \eta))   \varphi_{ b;-m/2+\delta m}( |\slashed\eta|/|\eta|)\gtrsim 2^{2b+2(k-k_2)}\gtrsim 2^{-m/2-10\alpha m }. 
\]
To utilize high oscillation in time, we do integration by parts in time once. As a result, from the volume of support of $\xi, \eta$, the estimate \eqref{2023jan30eqn71}, and the estimate \eqref{july23eqn91} in Proposition \ref{Zest}, we have
\[
  \sum_{(k_1,k_2)\in \chi_k^3, k_2  \leq   k- 100-5\max\{\tilde{c}, -c\} }   2^{(1+\alpha)j}\big\| \sum_{ \begin{subarray}{c}
j_i\in [-k_{i,-}, \infty)\cap \Z, i=1,2\\
\end{subarray} }  P_k\big[ \varphi_{j,k}(\cdot)\mathcal{F}^{-1}\big[ \int_{t_1}^{t_2}  \mathcal{J}_{k;k_1,k_2;l;2;b}^{a;+,\nu;j_1, j_2}(t, \xi) d t \big]\big](x)\big\|_{L^2}
\]
\be\label{2023jan29eqn65}
\lesssim    \sum_{(k_1,k_2)\in \chi_k^3,  k_2  \leq   k- 100-5\max\{\tilde{c}, -c\}} 2^{ 10\alpha m + (1+\alpha)j} 2^{ 3k/2+l/2}2^{k_2+2k +2l } 2^{-5(k_{1,+}+k_{2,+})}\epsilon_1^2\lesssim \epsilon_1^2.
\ee
Hence finishing the proof of the threshold case.

$\dagger$ The non-threshold case, i.e., $l\in  (-m/2+\delta m, -100 -  5\max\{\tilde{c}, -c\}]\cap \Z $.

  Based on the possible size of $\max\{j_1,j_2\}$, we split into two cases as follows. 

$\oplus$\quad If $\max\{j_1,j_2\}\geq j+ l/(1+2\alpha).$

If $j_1=\max\{j_1,j_2\}$, then we put $ e^{i t\Lambda_2}  h_{2;j_2,k_2}^{\mu} $ in $L^\infty$ and put $  h_{1;j_1,k_1}^{\mu}$ in $L^2$. As a result, from the estimate  \eqref{july13eqn11},  the $L^2-L^\infty$ type bilinear estimate, we have
\be\label{july22eqn22}
\begin{split}
& \| \sum_{ \begin{subarray}{c}
j_i\in [-k_{i,-}, \infty)\cap \Z, i=1,2\\
 j_1=\max\{j_1,j_2\}\geq   j+ l/(1+2\alpha)\\
\end{subarray} }  \sum_{}P_k\big[ \varphi_{j,k}(\cdot)\mathcal{F}^{-1}\big[    \int_{t_1}^{t_2}  \mathcal{J}_{k;k_1,k_2;l;2}^{a;+,\nu;j_1, j_2}(t, \xi) d t \big]\big](x)\|_{L^2} \\ 
& \lesssim 2^{ l}2^{-m+ k_2-8k_{1,+}}2^{ -(1+\alpha) ( j+ l/(1+2\alpha))+ \delta m }\epsilon_1^2. 
\end{split}
\ee

If $j_2=\max\{j_1,j_2\}$,  from the estimate   in \eqref{july13eqn11}, the orthogonality in $L^2$, and the super-localized decay estimate (\ref{linearwavedecay}) in Lemma \ref{lineardecay}, the following estimate holds after super-localizing the frequency of the input $h_{1;j_1,k_1}^{ }$ with the step size $2^{k_2}$, 
\be\label{july22eqn23}
\begin{split}
&  \| \sum_{ \begin{subarray}{c}
j_i\in [-k_{i,-}, \infty)\cap \Z, i=1,2\\
 j_2=\max\{j_1,j_2\}\geq   j+ l/(1+2\alpha)\\
\end{subarray} } P_k\big[ \varphi_{j,k}(\cdot)\mathcal{F}^{-1}\big[    \int_{t_1}^{t_2}  \mathcal{J}_{k;k_1,k_2;l;2}^{a;+,\nu;j_1, j_2}(t, \xi) d t \big]\big](x)\|_{L^2} \\
& \lesssim\sum_{j_2\geq  j+ l/(1+2\alpha)} \big( \sum_{a \in [2^{k-k_2-10}, 2^{k-k_2+10}]\cap \Z} \|\int_{\R^3} e^{   i t \Phi_{\mu, \nu}^a ( \xi , \eta)} q^a_{\mu, \nu} ( \xi-\eta,\eta)\widehat{h_{1;\leq j_2,k_1}^{ }}(t,\xi- \eta) \\ 
& \times     \widehat{h_{2;j_2,k_2}^{\nu}}(t, \eta) \psi_k(\xi)  \varphi_{l;-m/2+\delta m}( l_{\mu,\nu}(\xi, \eta)) \psi_{k_1}(|\xi-\eta|-a 2^{k_2})d\eta\|_{L^2_\xi}^2\big)^{1/2}\\ 
& \lesssim \sum_{j_2\geq  j+ l/(1+2\alpha)}  2^{ l-4k_{+}}   2^{-m+ k_2/3}   \epsilon_1 \| h_{2 ;j_2,k_2}^{\mu}\|_{L^2}\lesssim  2^{ l  -4k_{+}}  2^{-m+ k_2/3}  2^{ -(1+\alpha) ( j+ l/(1+2\alpha))+ \delta m }\epsilon_1^2.\\
\end{split}
\ee
After combining the obtained estimates \eqref{july22eqn22} and \eqref{july22eqn23}, we have
 \be\label{july22eqn51}
 \begin{split}
  & \sum_{\begin{subarray}{c}
(k_1,k_2)\in \chi_k^3 \\   
k_2  \leq   k- 100-5\max\{\tilde{c}, -c\}
\end{subarray}} 2^{(1+\alpha)j}\| \sum_{ \begin{subarray}{c}
j_i\in [-k_{i,-}, \infty)\cap \Z, i=1,2\\
 \max\{j_1,j_2\}\geq   j+ l/(1+2\alpha)\\
\end{subarray} }  P_k\big[ \varphi_{j,k}(\cdot)\mathcal{F}^{-1}\big[    \int_{t_1}^{t_2}   \mathcal{J}_{k;k_1,k_2;l;2}^{a;+,\nu;j_1, j_2}(t, \xi)  d t \big]\big](x)\|_{L^2}\\ 
&\lesssim     \sum_{\begin{subarray}{c}
(k_1,k_2)\in \chi_k^3 \\   
k_2  \leq   k- 100-5\max\{\tilde{c}, -c\}
\end{subarray}}      2^{m+ l} 2^{(1+\alpha)j -(1+\alpha) ( j+ l/(1+2\alpha))+ \delta m } 2^{ -m+k_1/3-4k_{+}}  \epsilon_1^2\lesssim  2^{  \delta m +\delta l}\epsilon_1^2. \\
\end{split}
\ee

$\oplus$\quad If $\max\{j_1,j_2\}\leq j+ l/(1+2\alpha).$

For this case, we will use same strategy used in the proof of Lemma \ref{highhighWfixed2} for the   estimate of $\mathcal{J}_{k;k_1,k_2;3}^{\mu,\nu;j_1, j_2;l}(t, \xi)$, see \eqref{july20eqn4}. 

We first rule out some trivial cases. By using the same strategy used in obtaining the estimate \eqref{july22eqn51}, the following estimate holds if $k_2\leq - 3\alpha m$ or $l\leq -\alpha m $,
 \be\label{july22eqn53}
 \begin{split}
 &   \sum_{\begin{subarray}{c}
(k_1,k_2)\in \chi_k^3 \\   
k_2  \leq   k- 100-5\max\{\tilde{c}, -c\}
\end{subarray}}  2^{(1+\alpha)j}\|  \sum_{ \begin{subarray}{c}
j_i\in [-k_{i,-}, \infty)\cap \Z, i=1,2\\
 \max\{j_1,j_2\}\geq  (1-\delta)m +  l \\
\end{subarray} } P_k\big[ \varphi_{j,k}(\cdot)\mathcal{F}^{-1}\big[    \int_{t_1}^{t_2}    \mathcal{J}_{k;k_1,k_2;l;2}^{a;+,\nu;j_1, j_2}(t, \xi)  d t \big]\big](x)\|_{L^2}\\
&\lesssim      \sum_{(k_1,k_2)\in \chi_k^3, k_2  \leq   k- 100-5\max\{\tilde{c}, -c\}}     2^{m+ l} 2^{(1+\alpha)j -(1+\alpha) ((1-\delta)m + l)+ \delta m }
 2^{-m+k_2/3} \epsilon_1^2 \lesssim  2^{   -10\delta m }\epsilon_1^2. \\ 
 \end{split}
\ee

Now, we focus on the case $k_1\geq -3\alpha m $ and $l\geq -\alpha m .$  Similar to \eqref{july21eqn27}, for any $\gamma\in \Z_+^3$, we have 
\be\label{july22eqn60}
\begin{split}
 \nabla_\xi^\gamma   \mathcal{J}_{k;k_1,k_2;l;2}^{a;+,\nu;j_1, j_2}(t, \xi) &=\sum_{\gamma_1+\gamma_2=\gamma} \mathfrak{J}_{\gamma_1, \gamma_2}^\gamma(t, \xi) \\ 
 \mathfrak{J}_{\gamma_1, \gamma_2}^\gamma(t, \xi)  &=\int_{\R^3} e^{i t \Phi_{+, \nu}^a ( \xi ,\eta)} \big(i t \nabla_{\xi}\Phi_{+, \nu}^a ( \xi ,\eta)\big)^{\gamma_1} \nabla_\xi^{\gamma_2}\big[ q^a_{+, \nu} (\xi-\eta, \eta) \widehat{h_{1;j_1,k_1}^{ }}(t, \xi-\eta) \psi_k(\xi)  \\ 
 & \quad \times     \varphi_{l;-m/2+\delta m}( l_{+,\nu}(\xi, \eta))    \phi_{\nu,2}(\xi, \eta) \big] \widehat{h_{2;j_2,k_2}^{\nu}}(t, \eta)   d\eta. 
\end{split}
\ee

For $\widetilde{J}_{\gamma_1, \gamma_2}^\gamma(t, \xi) $, we do integration by parts in ``$\eta$'' $|\gamma_1|$ times. As a result, we have
\be\label{july22eqn61}
\begin{split}
\mathfrak{J}_{\gamma_1, \gamma_2}^\gamma(t, \xi)   &=\int_{\R^3} e^{i t \Phi_{+, \nu}^a ( \xi ,\eta)} \\ 
&\times  \nabla_\eta \cdot \Big[ \frac{i  \nabla_\eta \Phi_{+, \nu}^a ( \xi ,\eta) }{|\nabla_\eta \Phi_{+, \nu}^a ( \xi ,\eta)|^2} \circ\cdots \nabla_\eta \cdot \Big[ \frac{i  \nabla_\eta \Phi_{+, \nu}^a ( \xi ,\eta) }{|\nabla_\eta \Phi_{+, \nu}^a ( \xi ,\eta)|^2} \big(i \nabla_{\xi} \Phi_{+, \nu}^a ( \xi ,\eta)\big)^{\gamma_1} \nabla_\xi^{\gamma_2}\big[ q^a_{+, \nu} (\xi-\eta, \eta)  \\ 
 & \qquad \times  \widehat{h_{1;j_1,k_1}^{}}(t, \xi-\eta) \psi_k(\xi)   \varphi_{l;-m/2+\delta m}( l_{+,\nu}(\xi, \eta))   \phi_{\nu,2}(\xi, \eta) \big] \widehat{h_{2;j_2,k_2}^{\nu}}(t, \eta) \Big]\circ\cdots\Big]  d\eta. 
\end{split}
\ee

Recall the range of $k,j,j_1,k_1,$ and $l$. The worst scenario happens when $\nabla_\eta^\kappa \nabla_\xi^{\gamma_1}$ hits $ \widehat{h_{1;j_1,k_1}^{ }}(t,  \eta)$ if $j_1=\max\{j_1,j_2\}$ ($ \widehat{h_{2;j_2,k_2}^{\mu}}(t, \xi-\eta)$ if $j_2=\max\{j_1,j_2\}$), where $\gamma\in \Z_+^3,$ s.t., $|\kappa|=|\gamma_2|$.  Moreover, from the estimates \eqref{july22eqn21} and \eqref{july21eqn84}, for any $n\in \Z_+$, the following estimate holds for the kernel of typical symbols appeared in \eqref{july21eqn28}, 
\[
\begin{split}
  \|\mathcal{F}^{-1}_{(\xi, \eta)\rightarrow (x,y)}&\big[ \underbrace{ \frac{i   \nabla_\eta  \Phi_{+, \nu}^a  ( \xi ,\eta) }{|\nabla_\eta   \Phi_{+, \nu}^a ( \xi ,\eta)|^2} \otimes \circ\cdots \circ \otimes\frac{i  \nabla_\eta   \Phi_{+, \nu}^a  ( \xi ,\eta) }{|\nabla_\eta   \Phi_{+, \nu}^a  ( \xi ,\eta)|^2}}_{\text{$n$-times }}       \varphi_{l;-m/2+\delta m}( l_{\mu,\nu}(\xi, \eta)) \big(i \nabla_{\xi}  \Phi_{+, \nu}^a  ( \xi ,\eta)\big)^{\gamma_1} \\ 
&\times  \psi_{k_1}(\xi-\eta)\psi_{k_2}(\eta) \psi_k(\xi) q^a_{+, \nu} (\xi-\eta, \eta)
\big]\|_{L^1_{x,y}}  
    \lesssim 2^{-(6+\delta)l -  (n-1) l+|\gamma_1|l/2 }.
    \end{split}
\]

From the above estimate of   kernels,   by using the same strategy used in obtaining the estimate \eqref{july22eqn51}, the following estimate holds for any $\gamma\in \Z_+^3$,  
\[
\sum_{\begin{subarray}{c}
(k_1,k_2)\in \chi_k^3 \\   
k_2  \leq   k- 100-5\max\{\tilde{c}, -c\}
\end{subarray}}  2^{(1+ \alpha)j}\|\sum_{ \begin{subarray}{c}
j_i\in [-k_{i,-}, \infty)\cap \Z, i=1,2\\
 \max\{j_1,j_2\}\leq    j+ l/(1+2\alpha)\\
\end{subarray} }P_k\big[ \varphi_{j,k}(\cdot)\mathcal{F}^{-1}\big[      \mathfrak{J}_{\gamma_1, \gamma_2}^\gamma(t, \xi) \big]\big](x)\|_{L^2}
\]
\[
\lesssim    \sum_{(k_1,k_2)\in \chi_k^3,  k_2  \leq   k- 100-5\max\{\tilde{c}, -c\}}  \sum_{ \begin{subarray}{c}
 \tilde{j}\leq     j+ l/(1+2\alpha) \\
\end{subarray} }    2^{-(6+\delta)l -  (n/2-1) l }  2^{  |\gamma|   \tilde{j}  }  2^{(1+\alpha)j -(1+\alpha) \tilde{j} + \delta m }
\]
\be\label{july22eqn81}
\times 2^{-m+k_2/3} 2^{-4k_{1,+}} \epsilon_1^2  \lesssim 2^{|\gamma|j -(6+\delta)l -  |\gamma| l/2 +  |\gamma|l/(1+2\alpha)} 2^{ -m+ \delta m}\epsilon_1^2     . 
\ee
 Recall \eqref{july22eqn60}. From the above estimate, we have
 \[
\sum_{\begin{subarray}{c}
(k_1,k_2)\in \chi_k^3 \\   
k_2  \leq   k- 100-5\max\{\tilde{c}, -c\}
\end{subarray}}  2^{(1+\alpha)j}\|  \sum_{ \begin{subarray}{c}
j_i\in [-k_{i,-}, \infty)\cap \Z, i=1,2\\
 \max\{j_1,j_2\}\leq    j+ l/(1+2\alpha)\\
\end{subarray} } P_k\big[ \varphi_{j,k}(\cdot)\mathcal{F}^{-1}\big[       \mathcal{J}_{k;k_1,k_2;l;2}^{a;+,\nu;j_1, j_2}(t, \xi)  \big]\big](x)\|_{L^2}
\]
 \[
\lesssim \sum_{ \begin{subarray}{c} \gamma, \gamma_1, \gamma_2\in \Z_+^3 \\ 
|\gamma|=20,
\gamma_1+\gamma_2=\gamma\\
\end{subarray}}  \sum_{\begin{subarray}{c}
(k_1,k_2)\in \chi_k^3 \\   
k_2  \leq   k- 100-5\max\{\tilde{c}, -c\}
\end{subarray}}  2^{(1+\alpha)j-|\gamma|j}\|  \sum_{ \begin{subarray}{c}
j_i\in [-k_{i,-}, \infty)\cap \Z, i=1,2\\
 \max\{j_1,j_2\}\leq    j+ l/(1+2\alpha)\\
\end{subarray} }  P_k\big[ \varphi_{j,k}(\cdot)\mathcal{F}^{-1}\big[     \mathfrak{J}_{\gamma_1, \gamma_2}^\gamma(t, \xi) \big]\big](x)\|_{L^2}
\]
\be 
\lesssim 2^{  - (6+\delta) l    + 10(1-4\alpha) l/(2+4\alpha)} 2^{ -m+ \delta m}\epsilon_1^2  \lesssim 2^{ -m+ \delta m+ \delta l}\epsilon_1^2     .
\ee

$\bullet$\qquad If $b=3$ or $l\in[ -100 -  5\max\{\tilde{c}, -c\}, \tilde{c}+10 ]\cap \Z $.

Recall the definition of the cutoff function $   \phi_{\nu,3}(\xi, \eta)$ in \eqref{feb2eqn1}. Since now  $(\xi_2-\eta_2, \xi_3-\eta_3)/|\xi-\eta|$ and $(c_1^2 \eta_2, c_2^2 \eta_3)/\Lambda_2(\eta)$ are not comparable, an improved bound as in \eqref{july21eqn23} holds. Thanks to the improved estimate,  with minor modifications in the above argument and the argument used for the estimate of $\mathcal{J}_{k;k_1,k_2;l;3}^{a;\mu,\nu;j_1, j_2}(t, \xi)$ in the proof of Lemma \ref{highhighWfixed2}, the case $b=3$ and the case $l\in[ -100 -  5\max\{\tilde{c}, -c\}, \tilde{c}+10 ]\cap \Z $ can be handled similarly. Hence finishing the proof.
\end{proof}
 
\subsection{The frequencies-comparable case}\label{frecompcase}
In this subsection, we consider the case that    the input frequencies and the  out frequency are all comparable, e.g., $\min\{k_1, k_2\}\geq k-100-5\max\{\tilde{c},-c\}$ and $k\geq \max\{k_1,k_2\}-100-5\max\{\tilde{c},-c\},$ where absolute constants  $\tilde{c}$ and $c$ are defined in \eqref{2023jan28eqn1}. 
 
 We single out this case because, unlike previous two subsections, the separation between the resonance case and the non-resonance case  is  not entirely determined by the sign of $\mu, \nu$. Instead, we will utilize a new partition of unity functions, see \eqref{july7eqn51}. Moreover, as one might notice when comparing the obtained estimate \eqref{2023jan30eqn71} in the non-comparable case and the obtained estimate  \eqref{2023jan30eqn91} in the comparable case,  there is also a difference. These estimates play roles in  the threshold case, i.e., $l=-m/2+\delta m$,   for which we will elaborate here.

\begin{lemma}\label{comparableest}
 Under the bootstrap assumption \eqref{bootstrap}, for any $t \in [2^{m-1}, 2^m]\subset[0, T], m\in \Z_+$, we have 
 \be\label{july7eqn4}
\sup_{k\in \Z}\sum_{ \begin{subarray}{c} 
 a=1,2\\ 
\mu, \nu\in\{+,-\}
\end{subarray}} \sum_{ \begin{subarray}{c}
 (k_1,k_2)\in \cup_{i=1,2,3}\chi_k^i \\ 
 \max\{k_1,k_2\}-100-5\max\{\tilde{c}, -c\}\leq k \\ 
 k\leq \min\{k_1,k_2\}+100+5\max\{\tilde{c}, -c\}\\
 \end{subarray}} \|\int_{t_1}^{t_2} \mathcal{F}^{-1}\big[   \mathcal{J}_{k;k_1,k_2}^{a;\mu,\nu }(t, \xi) \big]  d t  \|_{Z} \lesssim 2^{  \delta m }\epsilon_1^2.
\ee
\end{lemma}

\begin{proof} Recall \eqref{systemeqnprof}. Same as what we did in \eqref{july13eqn70}, we also do space localization for two inputs. As in the proof of Lemma \ref{fixedtimeest}, we first rule out the case   $j\geq  \max\{m, -k_{-}\}+ \delta (m+j) .$ It would be sufficient to consider the case  $j\leq  \max\{m, -k_{-}\}+ \delta (m+j) .$ 

 As   in  \eqref{july21eqn71}, we do dyadic decomposition for  $\mathcal{J}_{k;k_1,k_2}^{a;\mu,\nu;j_1, j_2}(t, \xi) $.  
  Note that the obtained estimates \eqref{july21eqn81} and \eqref{july21eqn84} are still valid. Moreover, similar to what we did in the proof of Lemma \ref{highhighWfixed2},  from the $L^2-L^\infty$ type bilinear estimate, the estimate \eqref{decayestimatefinal}, and the volume of support, we can rule out the case $k\leq -m/4-2\alpha m . $  

If $l\in [-100-5\max\{\tilde{c},-c\}, \tilde{c}+10]\cap \Z$, then $l$ and the relative sizes between the input frequencies and the  out frequency don't play a role.  Hence the same argument used for the estimate of $\mathcal{J}_{k;k_1,k_2;l;3}^{a;\mu,\nu;j_1, j_2}(t, \xi)$ in the proof of Lemma \ref{highhighWfixed2} still holds.  

From now on, we focus on  the  case $l\in [-m/2+\delta m, -100-5\max\{\tilde{c},-c\}]\cap \Z$. Recall \eqref{phases}, from Taylor expansion,  for any $\xi, \eta \in supp(\psi_k(\xi) \psi_{k_1}(\xi-\eta)\psi_{k_2}(\eta)     \varphi_{l;-m/2+\delta m}( l_{\mu,\nu}(\xi, \eta))  )$,  we have 
\[
 \Phi_{\mu, \nu}^a(\xi, \eta) =  |\xi_1|- \mu |\xi_1-\eta_1| - \nu |\eta_1| + \mathcal{O}(2^{k+l}).
\]

Based on the possible sign of $\mu, \nu,\xi_1, \eta_1$, we define the following partition of the unity functions,
\be\label{july7eqn51}
\begin{split}
\forall \mu, \nu\in\{+, -\}, \quad  1&=\varphi_{\mu,\nu}^1(\xi, \eta) + \varphi_{\mu,\nu}^2(\xi, \eta),\\
\varphi_{+,+}^1(\xi, \eta):= \mathbf{1}_{[0, \infty)}( (\xi_1-\eta_1)\eta_1),\qquad & \varphi_{+,+}^2(\xi, \eta):= \mathbf{1}_{ (-\infty,0)}( (\xi_1-\eta_1)\eta_1),\\ 
\varphi_{+,-}^1(\xi, \eta):=  \mathbf{1}_{ (-\infty,0)} (  \xi_1\eta_1),\qquad & \varphi_{+,+}^2(\xi, \eta):= \mathbf{1}_{[0, \infty)} (  \xi_1\eta_1),\\ 
\varphi_{-,+}^1(\xi, \eta):=  \mathbf{1}_{ (-\infty,0)} (   (\xi_1-\eta_1)\xi_1),\qquad & \varphi_{+,+}^2(\xi, \eta):=  \mathbf{1}_{[0, \infty)} ( (\xi_1-\eta_1)\xi_1),\\ 
\varphi_{-,-}^1(\xi, \eta):= 0,\qquad & \varphi_{+,+}^2(\xi, \eta):= 1.\\ 
\end{split}
\ee
where the cutoff function $\varphi_{\mu,\nu}^1(\xi, \eta)$ determines the resonance case and the cutoff function $\varphi_{\mu,\nu}^2(\xi, \eta)$ determines the non-resonance case.

With the partition of unity in \eqref{july7eqn51}, we split $ \mathcal{J}_{k;k_1,k_2;l}^{\mu,\nu;j_1, j_2}(t, \xi)$ into two parts as follows, 
\be
\begin{split}
\mathcal{J}_{k;k_1,k_2;l}^{a;\mu,\nu;j_1, j_2}(t, \xi)  =\sum_{i'=1,2}   &\mathcal{J}_{k;k_1,k_2;l}^{a;\mu,\nu;j_1, j_2;i'}(t, \xi) , \quad    
\mathcal{J}_{k;k_1,k_2;l}^{a;\mu,\nu;j_1, j_2;i'}(t, \xi)  :=  \int_{\R^3} e^{i t \Phi_{\mu, \nu}^a ( \xi ,\eta)} q^a_{\mu, \nu} (\xi-\eta, \eta) \\ 
& \times  \widehat{h_{1,k_1}^{\mu}}(t, \xi-\eta) \widehat{h_{2,k_2}^{\nu}}(t, \eta)   
    \psi_k(\xi)      \varphi_{l;-m/2+\delta m}( l_{\mu,\nu}(\xi, \eta))    \varphi_{\mu,\nu}^{i'}(\xi, \eta) d\eta. 
\end{split}
\ee

$\bullet$\qquad If $i'=1$, i.e., the resonance case. 

For any $    \mu, \nu\in\{+, -\},a\in\{1,2\}, $  $\xi, \eta \in supp(\psi_k(\xi) \psi_{k_1}(\xi-\eta)\psi_{k_2}(\eta)   $ $  \varphi_{l;-m/2+\delta m}( l_{\mu,\nu}(\xi, \eta))    \varphi_{\mu,\nu}^{1}(\xi, \eta)  )$,  we have
\[
 \big|\nabla_\xi \big( \Phi_{ \mu, \nu}^a ( \xi ,\eta)\big)\big| +  \big|\nabla_\xi \big( \Phi_{ \mu, \nu}^a ( \xi ,\xi-\eta)\big)\big| \lesssim 2^{l/2}.
\]
Hence, as in the proof of Lemma \ref{lowhighess}, we can rule out the case $j\geq m+l/2+\delta m $.  
Thanks to the  improved  upper bound of $j$.  , as in the proof of Lemma \ref{lowhighess}, the non-threshold case, i.e., $l\in (-m/2+\delta m, -100-5\max\{\tilde{c},-c\}]\cap \Z$ follows by the same argument. 

Now, we focus on the threshold case, i.e., $l= -m/2+\delta m.$ We localize further localize further based on the size of $\xi_2/|\eta|, \xi_3/|\eta|$   as follows,
 \be\label{2023jan30eqn100}
  \begin{split}
\mathcal{J}_{k;k_1,k_2;l}^{a;\mu,\nu;j_1, j_2;1}(t, \xi) &= \sum_{b_1, b_2\in[-m/2+\delta m,l/2+2]\cap\Z} \mathcal{J}_{k;k_1,k_2;l;b_1,b_2}^{a;\mu,\nu;j_1, j_2;1}(t, \xi),\\
\mathcal{J}_{k;k_1,k_2;l;b_1,b_2}^{a;\mu,\nu;j_1, j_2;1}(t, \xi)&:=\int_{\R^3} e^{i t \Phi_{\mu, \nu}^a ( \xi ,\eta)} q^a_{\mu, \nu} (\xi-\eta, \eta) \widehat{h_{1;j_1,k_1}^{ }}(t, \xi-\eta) \widehat{h_{2;j_2,k_2}^{\nu}}(t, \eta) \psi_k(\xi)  \varphi_{\mu,\nu}^{1}(\xi, \eta)   \\ 
 & \quad \times \psi_{b_1,b_2}(\xi)   \varphi_{l;-m/2+\delta m}( l_{\mu,\nu}(\xi, \eta))   d\eta,\\ 
\psi_{b_1,b_2}(\xi) &:=   \varphi_{ b_1;-m/2+\delta m}( \xi_2/|\xi|)\varphi_{ b_2;-m/2+\delta m}( \xi_3/|\xi|).\\
\end{split}
\ee
Recall the approximations in  \eqref{2023jan30eqn51}. From  the estimate of symbol in  \eqref{july13eqn11}, $L^2-L^\infty$-type bilinear estimate, and the volume of support of the frequency variable of the input putted in $L^2$, we have 
\be
\begin{split}
&\big\| \sum_{ \begin{subarray}{c}
j_i\in [-k_{i,-}, \infty)\cap \Z, i=1,2\\
\end{subarray} }  P_k\big[ \varphi_{j,k}(\cdot)\mathcal{F}^{-1}\big[ \mathcal{J}_{k;k_1,k_2;l;b_1,b_2}^{a;\mu,\nu;j_1, j_2;1}(t, \xi)  \big]\big](x)\big\|_{L^2}\\
&\lesssim     2^{ l + (1+\alpha)j} 2^{ 3k/2}\big( (2^{b_1}+2^{k+l}) (2^{b_2}+2^{k+l}) \big)^{1/2} 2^{-m+k-5k_{+}} \epsilon_1^2.\\
\end{split}
\ee
From the above estimate, we can rule out the case $b_1\leq -m/4-2\alpha m $ and   the case $b_2\leq -m/4-2\alpha m$. 

Now, it remains  to consider the case $b_1\geq  -m/4-2\alpha m$ and $b_2 \geq -m/4-2\alpha m$, for which we have $|\xi_2|, |\xi_3|\in [ 2^{-m/4-2\alpha m+k}, 2^{-m/4+\delta m +k} ]$. From  the approximations in  \eqref{2023jan30eqn51},  due to the  lower bound of $|\xi_2|, |\xi_3|$, we have
\be\label{2023jan30eqn110}
\big| 1+  {\mu \nu  c_1^2 |\xi_1-\eta_1|}{|\eta_1|^{-1}}\big|\gtrsim 2^{-3\alpha m  }, \quad \big| 1+  {\mu \nu  c_2^2 |\xi_1-\eta_1|}{|\eta_1|^{-1}}\big|\gtrsim 2^{-3\alpha m  }. 
\ee

Thanks to the above lower bound, the following estimate holds for the volume of support of $\eta$ for any fixed $\xi$, 
\be\label{2023jan31eqn1}
\begin{split}
&\big|\{  \eta :  |l_{\mu, \nu}(\xi,\eta)|\leq 2^l, |\eta|\sim 2^{k_2}, |\xi-\eta|\sim 2^{k_1}, |\xi|\sim 2^k \}\big| \lesssim 2^{3k+2l+6\alpha m }.
\end{split}
\ee

 Now, based on the size of phases $\Phi_{\mu, \nu}^a ( \xi ,\eta)$,  we split $\mathcal{J}_{k;k_1,k_2;l;b_1,b_2}^{a;\mu,\nu;j_1, j_2;1}(t, \xi)$ further into two parts as follows, 
\be\label{2023jan30eqn101}
  \begin{split}
\mathcal{J}_{k;k_1,k_2;l;b_1,b_2}^{a;\mu,\nu;j_1, j_2;1}(t, \xi)& = \sum_{a'=1,2} \mathcal{J}_{k;k_1,k_2;l;a';b_1,b_2}^{a;\mu,\nu;j_1, j_2;1}(t, \xi),\\
\mathcal{J}_{k;k_1,k_2;l;1;b_1,b_2}^{a;\mu,\nu;j_1, j_2;1}(t, \xi) &:=\int_{\R^3} e^{i t \Phi_{\mu, \nu}^a ( \xi ,\eta)} q^a_{\mu, \nu} (\xi-\eta, \eta) \widehat{h_{1;j_1,k_1}^{ }}(t, \xi-\eta) \widehat{h_{2;j_2,k_2}^{\nu}}(t, \eta) \psi_k(\xi)  \varphi_{\mu,\nu}^{1}(\xi, \eta)     \\ 
 &  \times  \psi_{b_1,b_2}(\xi)  \varphi_{l;-m/2+\delta m}( l_{\mu,\nu}(\xi, \eta))   \psi_{\leq -3m/4+ 10\alpha m}( \Phi_{\mu, \nu}^a ( \xi ,\eta)) d\eta,\\
 \mathcal{J}_{k;k_1,k_2;l;2;b_1,b_2}^{a;\mu,\nu;j_1, j_2;1}(t, \xi) &:=\int_{\R^3} e^{i t \Phi_{\mu, \nu}^a ( \xi ,\eta)} q^a_{\mu, \nu} (\xi-\eta, \eta) \widehat{h_{1;j_1,k_1}^{ }}(t, \xi-\eta) \widehat{h_{2;j_2,k_2}^{\nu}}(t, \eta) \psi_k(\xi)  \varphi_{\mu,\nu}^{1}(\xi, \eta)     \\ 
 &  \times  \psi_{b_1,b_2}(\xi)  \varphi_{l;-m/2+\delta m}( l_{\mu,\nu}(\xi, \eta))   \psi_{\geq -3m/4+ 10\alpha m}( \Phi_{\mu, \nu}^a ( \xi ,\eta)) d\eta.\\
\end{split}
\ee

Moreover, note that,  the approximation of phases in \eqref{2023jan29eqn61}  are still valid. For any fixed $\xi_1, \xi_2, \eta\in \R^3$, thanks to  the   estimate of lower bounds in \eqref{2023jan30eqn110}, we know that $\xi_3$ is localized inside a much smaller ball. More precisely,
\be\label{2023jan31eqn2}
\begin{split}
\big|\{ \xi_3:|l_{\mu,\nu}(\xi, \eta)|\sim 2^l, |\xi_2|\sim 2^{b_1+k},& |\xi_3|\sim 2^{b_2 + k}, |\xi|, |\xi-\eta|, |\eta|\sim 2^k, \\ 
& | \Phi_{\mu, \nu}^a ( \xi ,\eta)|\lesssim 2^{-3m/4+10\alpha m } \} \big|\lesssim 2^{-m/2+20\alpha m }.\\ 
\end{split}
\ee
From the above estimates of volume of support \eqref{2023jan31eqn1} and \eqref{2023jan31eqn2}, the estimate of symbol in \eqref{july13eqn11},   we have 
\be
\begin{split}
 &\sum_{ \begin{subarray}{c}
 (k_1,k_2)\in \cup_{i=1,2,3}\chi_k^i \\ 
 \max\{k_1,k_2\}-100-5\max\{\tilde{c}, -c\}\leq k \\ 
 k\leq \min\{k_1,k_2\}+100+5\max\{\tilde{c}, -c\}\\
 \end{subarray}}   \big\| \sum_{ \begin{subarray}{c}
j_i\in [-k_{i,-}, \infty)\cap \Z, i=1,2\\
\end{subarray} }  P_k\big[ \varphi_{j,k}(\cdot)\mathcal{F}^{-1}\big[   \mathcal{J}_{k;k_1,k_2;l;1;b_1,b_2}^{a;\mu,\nu;j_1, j_2;1}(t, \xi) \big]\big](x)\big\|_{L^2}  \\
&\lesssim 2^{(1+\alpha)j+ l } (2^{-m/2+20\alpha m }2^{b_1} )^{1/2} 2^{3k+2l+6\alpha  m - 8k_{+}}\epsilon_1^2 \lesssim 2^{- m-\alpha m}\epsilon_1^2. \\
\end{split}
\ee

For $\mathcal{J}_{k;k_1,k_2;l;2;b_1,b_2}^{a;\mu,\nu;j_1, j_2;1}(t, \xi) $, due to high oscillation in time, we do integration by parts in time once. As a result, from \eqref{2023jan31eqn1}, after using the volume of support of $\xi,\eta$, the estimate \eqref{july23eqn91} in Proposition \ref{Zest}, we have 
\be
\begin{split}
 &\sum_{ \begin{subarray}{c}
 (k_1,k_2)\in \cup_{i=1,2,3}\chi_k^i \\ 
 \max\{k_1,k_2\}-100-5\max\{\tilde{c}, -c\}\leq k \\ 
 k\leq \min\{k_1,k_2\}+100+5\max\{\tilde{c}, -c\}\\
 \end{subarray}}   \big\| \sum_{ \begin{subarray}{c}
j_i\in [-k_{i,-}, \infty)\cap \Z, i=1,2\\
\end{subarray} }  P_k\big[ \varphi_{j,k}(\cdot)\mathcal{F}^{-1}\big[  \int_{t_1}^{t_2} \mathcal{J}_{k;k_1,k_2;l;2;b_1,b_2}^{a;\mu,\nu;j_1, j_2;1}(t, \xi) \big]\big](x)\big\|_{L^2}  \\
&\lesssim 2^{(1+\alpha)j+ l +3m/4-10\alpha m } (2^{ b_1+b_2 } )^{1/2} 2^{3k+2l+6\alpha  m - 8k_{+}}\epsilon_1^2 \lesssim 2^{ -\alpha m}\epsilon_1^2. \\
\end{split}
\ee
$\bullet$\qquad If $i'=2$,  i.e., the non-resonance case. 

Recall \eqref{july7eqn51}, $ \forall a\in \{1,2\}, $ we have 
\[
  \big|\Phi_{\mu, \nu}^a(\xi, \eta)\varphi_{\mu,\nu}^2(\xi, \eta)\big| \sim 2^{k}.
\]
Hence, to utilize high oscillation in time, we do integration by parts in time once. As a result, by using the same argument used in the proof of Lemma \ref{lowhighosc}, we have
\[
 \sum_{ \begin{subarray}{c}
 (k_1,k_2)\in \cup_{i=1,2,3}\chi_k^i \\ 
 \max\{k_1,k_2\}-100-5\max\{\tilde{c}, -c\}\leq k \\ 
 k\leq \min\{k_1,k_2\}+100+5\max\{\tilde{c}, -c\}\\
 \end{subarray}}\| \mathcal{F}^{-1}\big[    \int_{t_1}^{t_2}  \mathcal{J}_{k;k_1,k_2;l}^{a;\mu,\nu;j_1, j_2;2}(t, \xi)  d t  \big]   \|_{Z}  \lesssim \epsilon_1^2.
\]
To sum up, our desired estimate \eqref{july7eqn4} holds. 
\end{proof}

\end{document}